\documentclass[12pt,a4paper,oneside]{report} 
\usepackage{bezier,amsmath,amssymb,stmaryrd,upgreek,amsthm}
\usepackage[colorlinks, linkcolor=black, citecolor=black, urlcolor=black, hyperfootnotes=true]{hyperref}

\usepackage{cite}

\input xy
\xyoption{all}
\xyoption{2cell}
\UseAllTwocells

\renewcommand{\baselinestretch}{1.1}
\theoremstyle{definition}
\newtheorem{Satz}{Satz}

\newtheorem{lem}[Satz]{Lemma}
\newtheorem{thrm}[Satz]{Theorem}

\newtheorem{defn}[Satz]{Definition}
\newtheorem{lemdefn}[Satz]{Lemma/Definition}
\newtheorem{thrmdefn}[Satz]{Theorem/Definition}

\newtheorem{prop}[Satz]{Proposition}
\newtheorem{rem}[Satz]{Remark}

\newtheorem{nota}[Satz]{Notation}
\newtheorem{kor}[Satz]{Corollary}
\newtheorem{exa}[Satz]{Example}

\newcommand{\sA}{\mathcal A}
\newcommand{\sB}{\mathcal B}
\newcommand{\sC}{\mathcal C}
\newcommand{\sD}{\mathcal D}
\newcommand{\sR}{\mathcal R}
\newcommand{\sL}{\mathcal L}
\newcommand{\sN}{\mathcal N}
\newcommand{\Mor}{\operatorname{Mor}}
\newcommand{\Hom}{\operatorname{Hom}}
\newcommand{\Ob}{\operatorname{Ob}}
\newcommand{\R}{\operatorname{R}}
\newcommand{\F}{\operatorname{F}}
\newcommand{\Adel}{\operatorname{Adel}}
\newcommand{\C}{\operatorname{C}}
\newcommand{\K}{\operatorname{K}}
\renewcommand{\Im}{\operatorname{Im}}
\newcommand{\oS}{\operatorname{S}}
\renewcommand{\H}{\operatorname{H}}
\newcommand{\p}{\operatorname{p}}
\newcommand{\iso}{\operatorname{iso}}
\renewcommand{\i}{\operatorname{i}}
\newcommand{\N}{\textbf{N}}
\newcommand{\Z}{\textbf{Z}}
\newcommand{\op}{\operatorname{op}}
\newcommand{\D}{\operatorname{D}}
\newcommand{\I}{\operatorname{I}}
\newcommand{\J}{\operatorname{J}}

\newcommand{\nZ}{\operatorname{Z}}

\newcommand{\Zf}{\operatorname{\textbf{Z}-free}}
\newcommand{\Zm}{\operatorname{\textbf{Z}-mod}}
\newcommand{\Zzm}{\operatorname{\textbf{Z}/2-mod}}

\renewcommand{\k}{\operatorname{k}}
\renewcommand{\c}{\operatorname{c}}
\renewcommand{\i}{\operatorname{i}}

\newcommand{\Freyd}{\operatorname{Freyd}}

\newcommand{\arm}{\ar~+{|*\dir{*}}} 
\newcommand{\armfl}[1]{\ar~+{|(#1)*\dir{*}}} 
\newcommand{\are}{\ar~+{|*\dir{|}}} 
\newcommand{\arefl}[1]{\ar~+{|(#1)*\dir{|}}}

\newcommand{\sm}[1]{\ensuremath{\left(\begin{smallmatrix}#1\end{smallmatrix}\right)}}
\newcommand{\moplus}{\mathord \oplus}

\usepackage{datetime}
\newdateformat{wddate}{\monthname[\THEMONTH] \THEYEAR}
\begin{document} 
\title{Bachelor's Thesis \\[4mm] \Huge
\sc Adelman's Abelianisation of an Additive Category}

\author{Nico Stein\\University of Stuttgart}

\date{November 2012 \\[40mm]
Advisor: Dr.\ Matthias K\"unzer}

\maketitle

\newpage



\begin{footnotesize}
\renewcommand{\baselinestretch}{0.7}%
\parskip0.0ex%
\tableofcontents%
\thispagestyle{empty}
\parskip1.2ex%
\renewcommand{\baselinestretch}{1.0}%
\end{footnotesize}%

\setcounter{page}{2} 





\chapter*{Introduction}
\addcontentsline{toc}{chapter}{Introduction} 

The category $\Zf$ of finitely generated free abelian groups has as objects the groups of the form $\Z^{\oplus k}$, where $k \in \N_0$, and as morphisms matrices with entries in $\Z$. We can add morphisms with same source and target by addition of matrices. Moreover, we have direct sums of objects. So $\Zf$ is an example of an additive category. In additive categories the morphisms with same source and target form an  abelian group under addition, which is bilinear with respect to the composition, and finite direct sums of objects exist.

Abelian categories are additive categories in which there exist kernels and cokernels such that a kernel-cokernel-factorisation property is fulfilled. They form the classical setting for homological algebra. The category $\Zm$ of finitely generated abelian groups and group homomorphisms is the archetypical example for an abelian category. In contrast, the category $\Zf$ mentioned above is additive but not abelian.

In \cite{adelman}, Adelman gives a construction that extends a given additive category $\sA$ such that the resulting category $\Adel(\sA)$ is abelian and satisfies the following universal property.\\
 For every additive functor $F \colon \sA \to \sB$ with $\sB$ abelian, there exists an exact functor $\hat F \colon \Adel(\sA) \to \sB$, unique up to isotransformation, such that $F = \hat F \circ \I_{\sA}$, where $\I_{\sA}$ is the inclusion functor from $\sA$ to $\Adel(\sA)$.
\begin{align*}
\xymatrix{
\sA \ar[r]^{F} \ar[d]_{\I_{\sA}}& \sB\\
\Adel(\sA) \ar[ur]_{\hat F}
}
\end{align*}
For example, the (additive) inclusion functor $E$ from $\Zf$ to $\Zm$ yields an exact functor $\hat E \colon \Adel(\Zf) \to \Zm$. This functor cannot be an equivalence since there exist non-kernel-preserving functors from $\Zf$ to abelian categories or, alternatively, since $\Zm$ has not enough injectives. Cf. example \ref{exa:Z}.

Already in \cite[th~4.1]{freyd}, Freyd shows the existence of such a universal abelian category. To this end, he embeds the additive category $\sA$ into a product of abelian categories  and then finds the desired abelian category as a subcategory of this product. \\
He also constructs a category $\Freyd(\sA)$ which is abelian if and only if $\sA$ has weak kernels \cite[th~3.2]{freyd}.
Beligiannis shows in \cite[th~6.1]{beligiannis} that the category $\Freyd^{\op}(\Freyd(\sA))$ has this universal property, where $\Freyd^{\op}$ denotes the dual version. An equivalent two-step construction is described by Krause in \cite[universal~property~2.10]{krause}.

We follow the construction given by Adelman in \cite{adelman}, calling our resulting abelian category the \emph{Adelman category} $\Adel(\sA)$ of $\sA$. It is obtained as a factor category of the functor category $\sA^{\Delta_2}$. So the objects of $\Adel(\sA)$ are given as diagrams of the form $\big(\xymatrix@1{X_0\, \ar[r]^{x_0} &\, X_1 \, \ar[r]^{x_1} & \, X_2}\big)$, and the morphisms from $\big(\xymatrix@1{X_0\, \ar[r]^{x_0} &\, X_1\, \ar[r]^{x_1} &\, X_2}\big)$ to $\big(\xymatrix@1{Y_0\, \ar[r]^{y_0} &\, Y_1\, \ar[r]^{y_1} &\, Y_2}\big)$ are equivalence classes of commutative diagrams of the following form.
\begin{align*}
\xymatrix{
X_0 \ar[r]^{x_0} \ar[d]^{f_0} & X_1 \ar[r]^{x_1}\ar[d]^{f_1} & X_2\ar[d]^{f_2}\\
Y_0 \ar[r]^{y_0} & Y_1 \ar[r]^{y_1} & Y_2
}
\end{align*}
Such a diagram morphism represents the zero morphism if and only if there exist morphisms $s \colon X_1 \to Y_0$ and $t \colon X_2 \to Y_1$ in $\sA$ satisfying $s y_0 + x_1 t = f_1$. \\
We give an alternative description of the ideal in $\sA^{\Delta_2}$ that is factored out to obtain the Adelman category: the diagram morphisms representing zero are precisely those that factor through direct sums of objects of the forms $\big(\xymatrix@1{X\,\ar@{=}[r]&\,X\,\ar[r] &\,Y}\big)$ and $\big(\xymatrix@1{X\,\ar[r]&\,Y\,\ar@{=}[r] &\,Y}\big)$ in $\sA^{\Delta_2}$. Cf. remark \ref{rem:adelideal}.\\
Adelman gives explicit formulas for kernels and cokernels using only direct sums \cite[th~1.1]{adelman}. He introduces full subcategories of certain projective objects and certain injective objects to show that $\Adel(\sA)$ is abelian \cite[p.~101, l.~10]{adelman}, that it has enough projectives and injectives \cite[p.~108, l.~6]{adelman}, and that its projective and injective dimensions are at most two \cite[prop~1.3]{adelman}. These subcategories are also essential in his proof of the uniqueness in the universal property.
To construct the induced morphism in the universal property, he extends the given additive functor $F$ to an exact functor $\Adel(F)$ between the corresponding Adelman categories and then composes with a (generalised) homology functor \cite[below def~1.13]{adelman}.

\bigskip

We construct a kernel-cokernel-factorisation for an arbitrary morphism in $\Adel(\sA)$ and give an explicit formula for the induced morphism and its inverse. Cf. corollary \ref{kor:kernelcokernelfactorisation}.

We extend Adelman's construction, including the universal property, such that it also applies to transformations between additive functors. Cf. definition \ref{defn:adelfunctortransformation} and theorem \ref{thrm:universalproperty}.

\section*{Outline of the thesis}
\addcontentsline{toc}{section}{Outline of the thesis}

We summarise the required preliminaries in chapter \ref{ch:preliminaries}. Most of them are basic facts from homological algebra.

In chapter \ref{ch:factorcategories}, we define ideals in additive categories and give the definition and properties of factor categories. At the end of this chapter, we formulate the universal property of a factor category.

We study Adelman's construction in detail in chapter \ref{ch:adelmansconstruction}. In section \ref{sec:definitions}, we define the Adelman category of an additive category $\sA$ and give two descriptions of the ideal that is factored out. We see that the dual of the Adelman category of $\sA$ is isomorphic to the Adelman category of $\sA^{\op}$, which allows us to use duality arguments. We construct kernels and cokernels in section \ref{sec:kernelsandcokernels} and obtain a criteria for a morphism being a mono- or epimorphism. In section \ref{sec:adelmanabelian}, we give explicit formulas for the induced morphism of a kernel-cokernel-factorisation and its inverse and we conclude that the Adelman category is in fact abelian. We define subcategories of projective and injective objects in section~\ref{sec:projectivesandinjectives} and use them to prove that the Adelman category has enough projectives and injectives and that its projective and injective dimensions are at most two. Moreover, we see that the elements of $\sA$ become both projective and injective in $\Adel(\sA)$.
In section \ref{sec:additivefunctorsandtransformations}, we extend Adelman's construction to additive functors and transformations between them. 

The aim of chapter \ref{ch:universalproperty} is to prove the universal property of the Adelman category. At first, in section \ref{sec:abeliancase}, we prove some rather technical lemmata to construct the homology functor from $\Adel(\sB)$ to $\sB$ in the particular case of an abelian category $\sB$. We formulate the universal property of the Adelman category in the last section \ref{sec:universalproperty} and give some direct applications.

\bigskip

\bigskip

I would like to thank my advisor Matthias K\"unzer for an uncountable amount of very useful suggestions and for his continual support.

I also thank Theo B\"uhler for pointing out Adelman's construction.

\section*{Conventions}
\addcontentsline{toc}{section}{Conventions}
\label{conventions}

We assume the reader to be familiar with elementary category theory. An introduction to category theory can be found in \cite{schubert}. Some basic definitions and notations are given in the conventions below.

Suppose given categories $\sA$, $\sB$ and $\sC$.

\begin{enumerate}
\item 
All categories are supposed to be small with respect to a sufficiently big universe,
cf. \cite[\S3.2~and~\S3.3]{schubert}.

\item The set of objects in $\sA$ is denoted by $\Ob \sA$. The set of morphisms in $\sA$ is denoted by $\Mor \sA$. Given $A,B \in \Ob \sA$, the set of morphisms from $A$ to $B$ is denoted by $\Hom_{\sA}(A,B)$ and the set of isomorphisms from $A$ to $B$ is denoted by $\Hom_{\sA}^{\iso}(A,B)$. The identity morphism of $A \in \Ob \sA$ is denoted by $1_A$. We write $1:=1_A$ if unambiguous. Given $f \in \Hom_{\sA}^{\iso}(A,B)$, we denote its inverse by $f^{-1} \in \Hom_{\sA}^{\iso}(B,A)$.


\item The composition of morphisms is written naturally:\\
$\big(\xymatrix@1{A\, \ar[r]^f&\,B\, \ar[r]^g& \,C}\big)=\big(\xymatrix@1{A \,\ar[r]^{fg}&\, C}\big) =\big(\xymatrix@1{A \,\ar[r]^{f\cdot g}& \,C}\big)$ in $\sA$.

\item The composition of functors is written traditionally: $\big(\xymatrix@1{\sA \,\ar[r]^F&\,\sB \,\ar[r]^G&  \,\sC}\big)=\big(\xymatrix@1{\sA \,\ar[r]^{G \circ F}& \,\sC}\big)$.

\item In diagrams, we sometimes denote monomorphisms by $\xymatrix@1@!0{\arm[r]&}$ and epimorphisms by $\xymatrix@1@!0{\are[r]&}$.

\item Given $A,B \in \Ob \sA$, we write $A \cong B$ if $A$ and $B$ are isomorphic in $\sA$.

\item The opposite category (or dual category) of $\sA$ is denoted by $\sA^{\op}$. Given $f \in \Mor \sA$, the corresponding morphism in $\sA^{\op}$ is denoted by $f^{\op}$. Cf. remark \ref{rem:dual}.

\item We call $\sA$ \emph{preadditive} if $\Hom_\sA(A,B)$ carries the structure of an abelian group for \linebreak $A,B \in \Ob \sA$, written additively, and if $f(g+g')h=fgh+fg'h$ holds for  \\ $\xymatrix@1{A\,\ar[r]^{f}&  \,B\,\ar@<2pt>[r]^{g} \ar@<-2pt>[r]_{g'}  &\,C\, \ar[r]^{h}  & \,D }$ in $\sA$.

The zero morphisms in a preadditive category $\sA$ are denoted by $0_{A,B} \in \Hom_{\sA}(A,B)$ for $A,B \in \Ob \sA$. We write $0_{A}:= 0_{A,A}$ and $0:=0_{A,B}$ if unambiguous.

\item The set of integers is denoted by $\Z$. The set of non-negative integers is denoted by $\N_0$.

\item Given $a,b \in \Z$, we write $[a,b]:= \{ z \in \Z \colon a \le z \le b\}$.

\item \label{conv:delta} For $n \in \N_0$, let $\Delta_n$ be the poset category of $[0,n]$ with the partial order inherited from $\Z$. This category has $n+1$ objects $0,1,2, \dots, n$ and morphisms $\xymatrix@1@C=4mm{i\, \ar[r] &\,j}$ for $i,j \in \Z$ with $i \le j$.

\item \label{conv:directsum} Suppose $\sA$ to be preadditive and suppose given $A,B \in \Ob \sA$. A \emph{direct sum} of $A$ and $B$ is an object $C \in \Ob \sA$ together with morphisms $\xymatrix@1{A\,\ar@<2pt>[r]^i   &\,C\,\ar@<2pt>[l]^{p} \ar@<-2pt>[r]_q&\,B\ar@<-2pt>[l]_{j}}$ in $\sA$ satisfying $ip=1_A$, $jq=1_B$ and $pi+qj=1_C$. \\
This is generalized to an arbitrary finite number of objects as follows.

Suppose given $n \in \N_0$ and $A_k \in \Ob \sA$ for $k \in [1,n]$. A direct sum of $A_1, \dots
, A_n$ is a tuple $\big( C,(i_k)_{k \in [1,n]},(p_k)_{k\in [1,n]} \big)$ with $C \in \Ob \sA$ and morphisms $\xymatrix@1{A_k \, \ar@<2pt>[r]^{i_k}   &\, C\ar@<2pt>[l]^{p_k}}$ in $\sA$ satisfying $i_k p_k = 1_{A_k}$, $i_k p_{\ell}=0_{A_k,A_\ell}$ and $\sum_{m=1}^n p_m i_m = 1_C$ for $k,\ell \in [1,n]$ with $k \neq \ell$. Cf. remark \ref{rem:directsumcompatible}.

Sometimes we abbreviate $C:=\big( C,(i_k)_{k \in [1,n]},(p_k)_{k\in [1,n]} \big)$ for such a direct sum.

We often use the following matrix notation for morphisms between direct sums.

Suppose given $n,m \in \N_0$, a direct sum $\big( C,(i_k)_{k \in [1,n]},(p_k)_{k\in [1,n]} \big)$ of $A_1, \dots, A_n$ in $\sA$ and a direct sum $\big( D,(j_k)_{k \in [1,m]},(q_k)_{k\in [1,m]} \big)$ of $B_1, \dots, B_m$ in $\sA$.

Any morphism $f \colon C \to D$ in $\sA$ can be written as $f=\sum_{k=1}^n \sum_{\ell=1}^m p_k f_{k\ell}j_\ell$ with unique morphisms $f_{k\ell}=i_k f q_\ell \colon A_k \to B_\ell$ for $k \in [1,n]$ and $\ell \in [1,m]$.
We write 
\begin{align*}
f=\big(f_{k\ell}\big)_{\substack{k \in [1,n],\\\ell \in [1,m]}} = \left(\begin{smallmatrix}f_{11}&{\textstyle\cdots}&f_{1m}\\\vspace{1.8mm}\vdots&\ddots&\vdots\\f_{n1}&{\textstyle\cdots}&f_{nm}\end{smallmatrix}\right).
\end{align*}
Omitted matrix entries are zero.

\item An object $A \in \Ob \sA$ is called a \emph{zero object} if for every $B \in \Ob \sA$, there exists a single morphism from $A$ to $B$ and there exists a single morphism from $B$ to $A$.

If $\sA$ is preadditive, then $A \in \Ob \sA$ is a zero object if and only if $1_A = 0_A$ holds.

\item We call $\sA$ \emph{additive} if $\sA$ is preadditive, if $\sA$ has a zero object and if for $A,B \in \Ob \sA$, there exists a direct sum of $A$ and $B$ in $\sA$. In an additive category direct sums of arbitrary finite length exist.

\item \label{conv:directsumchoice} Suppose $\sA$ to be additive. We choose a zero object $0_\sA$ and write $0:=0_\sA$ if unambiguous. We choose $0_{\sA^{\op}} = 0_{\sA}$, cf. remark \ref{rem:dual}~(c).

For $n \in \N_0$ and $A_1, \dots, A_n \in \Ob\sA$, we choose a direct sum 
\begin{align*}
\Bigg( \bigoplus_{k=1}^n A_k , \Big(i_{\ell}^{(A_k)_{k \in [1,n]}} \Big)_{\ell \in [1,n]} ,\Big(p_{\ell}^{(A_k)_{k \in [1,n]}} \Big)_{\ell \in [1,n]}  \Bigg).
\end{align*}
We sometimes write $A_1 \oplus \dots \oplus A_n := \bigoplus_{k=1}^n A_k$.

Note that $p_{\ell}^{(A_k)_{k \in [1,n]}} = \sm{0 \\ \vspace{1.8mm}\vdots \\ 0 \vspace{0.5mm}\\\vspace{0.5mm} 1 \\0 \\\vspace{1.8mm} \vdots \\ 0 }$ and $i_{\ell}^{(A_k)_{k \in [1,n]}} = \sm{0 &{\textstyle\cdots}& 0& 1& 0& {\textstyle \cdots}& 0}$ in matrix notation, where the ones are in the $\ell$th row resp. column for $\ell \in [1,n]$.

In functor categories we use the direct sums of the target category to choose the direct sums, cf. remark \ref{rem:functorcategory}~(f).


\item \label{conv:genfulladdsubcategory} Suppose $\sA$ to be additive. Suppose given $M \subseteq \Ob \sA$. The full subcategory $\langle M \rangle$ of $\sA$ defined by 
\begin{align*}
\Ob \langle M \rangle := \left\{Y \in \Ob \sA \colon Y \cong \bigoplus_{i=1}^\ell X_i \text{ for some } \ell \in \N_0 \text{ and some } X_i \in M \text{ for } i \in [1,\ell] \right\}
\end{align*}
 shall be the \emph{full additive subcategory of $\sA$ generated by} $M$.

\item Suppose $\sA$ to be preadditive. Suppose given $f \colon A \to B$ in $\sA$. A \emph{kernel} of $f$ is a morphism $k \colon K \to A$ in $\sA$ with $kf=0$ such that the following factorisation property holds.

Given $g\colon X \to A$ with $gf=0$, there exists a unique morphism $u \colon X \to K$ such that $uk=g$ holds.

Sometimes we refer to $K$ as the kernel of $f$.

The dual of a kernel is called a \emph{cokernel}.

Note that kernels are monomorphic and that cokernels are epimorphic.

\item \label{conv:kernelcokernelfactorisation} Suppose $\sA$ to be preadditive. Suppose given $f \colon A \to B$ in $\sA$, a kernel $k \colon K \to A$ of $f$, a cokernel $c \colon B \to C$ of $f$, a cokernel $p \colon A \to J$ of $k$ and a kernel $i \colon I \to B$ of $c$.

There exists a unique morphism $\hat f \colon J \to I$ such that the following diagram commutes.
\begin{align*}
\xymatrix{
K \arm[r]^k & A \are[dr]_{p} \ar[rrr]^f & & & B \are[r]^*+<0.8mm,0.8mm>{\scriptstyle c} &C
\\
&& J \ar[r]^{\hat f} & I \arm[ru]_i
}
\end{align*}
We call such a diagram a \emph{kernel-cokernel-factorisation} of $f$, and we call $\hat f$ the induced morphism of the kernel-cokernel-factorisation. Cf. remark \ref{rem:kernelcokernelfactorisation}.

\item We call $\sA$ \emph{abelian} if $\sA$ is additive, if each morphism in $\sA$ has a kernel and a cokernel and if for each morphism $f$ in $\sA$, the induced morphism $\hat f$ of each kernel-cokernel-factorisation of $f$ is an isomorphism.

\item Suppose $\sA$ and $\sB$ to be preadditive. A functor $F \colon \sA \to \sB$ is called \emph{additive} if $F(f+g)=F(f)+F(g)$ holds for $\xymatrix@1{X\,\ar@<2pt>[r]^f \ar@<-2pt>[r]_g&\,Y}$ in $\sA$, cf. proposition \ref{prop:additivefunctor}.

\item \label{conv:transformation} Suppose given functors $F,G \colon \sA \to \sB$.\\ A tuple $\alpha=(\alpha_X)_{X \in \Ob \sA}$, where $\alpha_X \colon F(X) \to G(X)$ is a morphism in $\sB$  for $X \in \Ob \sA$, is called \emph{natural} if  $ F(f) \alpha_Y = \alpha_X G(f)$ holds for $\xymatrix@1{X\, \ar[r]^f&\,Y}$ in $\sA$. 
A natural tuple is also called a \emph{transformation}. We write $\alpha \colon F \to G$, $\alpha \colon F \Rightarrow G$ or $\xymatrix@1{
\sA\, \ar@/^/[r]^F="F"\ar@/_/[r]_G="G"  &\,\sB 
\ar@2"F"+<0mm,-2.6mm>;"G"+<0mm,2.4mm>_\alpha
} $.

A transformation $\alpha \colon F \Rightarrow G$ is called an \emph{isotransformation} if and only if $\alpha_X$ is an isomorphism in $\sB$ for $X \in \Ob \sA$. The isotransformations from $F$ to $G$ are precisely the elements of $\Hom^{\iso}_{\sB^{\sA}}(F,G)$, cf. convention \ref{conv:functorcategory}.

Suppose given $\xymatrix@1{ \sA  \,\ar@/^5mm/[rr]^F="F"   \ar[rr]|{\phantom{\rule{0.4mm}{0mm}}G\phantom{\rule{0.4mm}{0mm}}}="G" \ar@/_5mm/[rr]_H="H" & &\,\sB\, \ar@/^/[r]^K="K" \ar@/_/[r]_L="L" &\, \sC 
\ar@2"F"+<0mm,-2.6mm>;"G"+<0mm,1.4mm>_\alpha
\ar@2"G"+<0mm,-1.5mm>;"H"+<0mm,2.4mm>_\beta
\ar@2"K"+<0mm,-2.6mm>;"L"+<0mm,2.4mm>_\gamma}$.

We have the transformation $\alpha  \beta \colon F \Rightarrow H$ with $(\alpha  \beta)_X := \alpha_X \beta_X$ for $X \in \Ob \sA$.

We have the transformation $\gamma \star \alpha \colon K \circ F \Rightarrow L \circ G$ with $(\gamma \star \alpha)_X := K(\alpha_X) \gamma_{G(X)} = \gamma_{F(X)} L(\alpha_X)$ for $X \in \Ob \sA$.
\begin{align*}
\xymatrix{
(K \circ F)(X) \ar[r]^{\gamma_{F(X)}} \ar[d]_{K(\alpha_X)} & (L\circ F)(X) \ar[d]^{L(\alpha_X)}\\
(K \circ G)(X) \ar[r]_{\gamma_{G(X)}} & (L \circ G)(X)
}
\end{align*}
We have the transformation $1_F \colon F \Rightarrow F$ with $(1_F)_X:=1_{F(X)}$ for $X \in \Ob \sA$.

We set $\gamma \star F := \gamma \star 1_F$. 

Cf. lemma \ref{lem:transformations}.

\item \label{conv:functorcategory} Let $\sA^{\sC}$ denote the \emph{functor category} whose objects are the functors from $\sC$ to $\sA$ and whose morphisms are the transformations between such functors. Cf. remark \ref{rem:functorcategory}.

\item The \emph{identity functor} of $\sA$ shall be denoted by $1_{\sA} \colon \sA \to \sA$.

\item A functor $F  \colon \sA \to \sB$ is called an \emph{isomorphism of categories} if there exists a functor $G \colon \sB \to \sA$ with $F \circ G = 1_{\sB}$ and $G \circ F = 1_{\sA}$. In this case, we write $F^{-1}:=G$.

\item Suppose given $\xymatrix@1{A \,\ar[r]^f&\, B}$ in $\sA$. The morphism $f$ is called a \emph{retraction} if there exists a morphism $g \colon B\to A$ with $gf=1_Y$. The dual of a retraction is called a \emph{coretraction}.

\item An object $P \in \Ob \sA$ is called \emph{projective} if for  each morphism $g \colon P \to Y$ and each epimorphism $f \colon X \to Y$, there exists $h \colon P \to X$ with $hf=g$.
\begin{align*}
\xymatrix{
&P \ar[dl]_h \ar[d]^g \\
X \are[r]_*+<0.9mm,0.9mm>{\scriptstyle f} &  Y
}
\end{align*}
 The dual of a projective object is called an \emph{injective} object.

\item Suppose $\sA$ to be abelian. Suppose given $\xymatrix@1{X\, \ar[r]^f &\,Y}$ in $\sA$. A diagram $\xymatrix@1{X\, \are[r]^*+<0.7mm,0.7mm>{\scriptstyle p} & \,I\,\arm[r]^i & \,Y}$ is called an \emph{image} of $f$ if $p$ is epimorphic, if $i$ is monomorphic and if $pi=f$ holds.

\item Suppose $\sA$ to be abelian. A sequence $\xymatrix@1{X\, \ar[r]^f & \,Y \,\ar[r]^g &\,Z}$ in $\sA$ is called \emph{left-exact} if $f$ is a kernel of $g$ and \emph{right-exact} if $g$ is a cokernel of $f$.
Such a sequence is called \emph{short exact} if it is left-exact and right-exact.

Suppose given $n \in \N_0$ and a sequence $\xymatrix@1{A_0 \,\ar[r]^{a_0} &\, A_1\, \ar[r]^{a_1} & \,A_2\, \ar[r]^{a_2}&\,\cdots\, \ar[r]^-{a_{n-2}}&\, A_{n-1} \,\ar[r]^{a_{n-1}} &\, A_n}$ in $\sA$.
The sequence is called \emph{exact} if for each choice of images
\begin{align*}
\big(\xymatrix{A_k \ar[r]^-{a_k} & A_{k+1} }\big) = \big(\xymatrix{A_k \are[r]^*+<0.8mm,0.8mm>{\scriptstyle p_k} & I_{k} \armfl{0.41}[r]^-{i_k} & A_{k+1}}\big)
\end{align*} 
for $k\in[0,n-1]$, the sequences $\xymatrix@1{I_k \,\ar[r]^-{i_k}&\,A_{k+1}\, \ar[r]^-{p_{k+1}} & \,I_{k+1}}$ are short exact for $k\in[0,n-2]$.

Cf. remark \ref{rem:exactsequence}.

\item Suppose $\sA$ and $\sB$ to be abelian. \\
An additive functor $F \colon \sA \to \sB$ is called \emph{left-exact} if for each left-exact sequence $\xymatrix@1{X \,\arm[r]^f & \,Y\, \ar[r]^g &\,Z}$ in $\sA$, the sequence $\xymatrix@1{F(X)\, \ar[r]^-{F(f)} &\, F(Y)\, \ar[r]^-{F(g)} &\,F(Z)}$ is also left-exact.

An additive functor $F \colon \sA \to \sB$ is called \emph{right-exact} if for each right-exact sequence\\ $\xymatrix@1{X \,\ar[r]^f & \,Y\, \are[r]^*+<0.8mm,0.8mm>{\scriptstyle g} &\,Z}$ in $\sA$, the sequence $\xymatrix@1{F(X)\, \ar[r]^-{F(f)} &\, F(Y) \,\ar[r]^-{F(g)} &\,F(Z)}$ is also right-exact.

An additive functor $F \colon \sA \to \sB$ is called $\emph{exact}$ if it is left-exact and right-exact.

\item Suppose given a subcategory $\sA' \subseteq \sA$ and a subcategory $\sB' \subseteq \sB$.
Suppose given a functor $F \colon \sA \to \sB$ with $F(f) \in \Mor \sB'$ for $f \in \Mor \sA'$.

Let $F|_{\sA'}^{\sB'} \colon \sA' \to \sB'$ be defined by $F|_{\sA'}^{\sB'}(A)=F(A)$ for $A \in \Ob \sA'$ and $F|_{\sA'}^{\sB'}(f)=F(f)$ for $f \in \Mor \sA'$. 

Let $F|_{\sA'} := F|_{\sA'}^{\sB}$ and let $F|^{\sB'} := F|_{\sA}^{\sB'}$.
\end{enumerate}

\chapter{Preliminaries}
\thispagestyle{empty}
\label{ch:preliminaries}

\section{A remark on duality}

\begin{rem} \label{rem:dual}
Suppose given a category $\sA$. Some facts about the opposite category $\sA^{\op}$ are listed below.
\begin{itemize}
\item[(a)] We have $(\sA^{\op})^{\op} = \sA$.
	\item[(b)] If $\sA$ is preadditive, then $\sA^{\op}$ is preadditive with addition $f^{\op} + g^{\op} = (f+g)^{\op}$ for $\xymatrix@1{A\,\ar@<2pt>[r]^{f} \ar@<-2pt>[r]_{g}  &\,B }$ in $\sA$.
	\item[(c)] An object $A \in \Ob \sA$ is a zero object in $\sA$ if and only if it is a zero object in $\sA^{\op}$.
	\item[(d)] Suppose $\sA$ to be preadditive. Suppose given  $n \in \N_0$, $A_k \in \Ob \sA$ for $k \in [1,n]$ and a tuple $\big( C,(i_k)_{k \in [1,n]},(p_k)_{k\in [1,n]} \big)$ with $C \in \Ob \sA$ and morphisms $\xymatrix@1{A_k\, \ar@<2pt>[r]^{i_k}   &\,C\ar@<2pt>[l]^{p_k}}$ in $\sA$ for $k\in [1,n]$. The tuple $\big( C,(i_k)_{k \in [1,n]},(p_k)_{k\in [1,n]} \big)$ is a direct sum of $A_1, \dots, A_n$ in $\sA$ if and only if the tuple $\big( C,(p_k^{\op})_{k\in [1,n]},(i_k^{\op})_{k \in [1,n]} \big)$ is a direct sum of $A_1, \dots, A_n$ in $\sA^{\op}$.
\item[(e)] If $\sA$ is additive, then $\sA^{\op}$ is additive. 
\item[(f)] If $\sA$ is abelian, then $\sA^{\op}$ is abelian. 
\item[(g)] Suppose given a category $\sB$ and a functor $F \colon \sA \to \sB$. The opposite functor \linebreak $F^{\op} \colon \sA^{\op} \to \sB^{\op}$ is defined by $F^{\op}(A) = F(A)$ for $A \in \Ob \sA$ and $F^{\op}(f^{\op})=F(f)^{\op}$ for $f \in \Mor \sA$. We have $(F^{\op})^{\op} = F$.

Suppose $\sA$ and $\sB$ to be preadditive. The functor $F$ is additive if and only if $F^{\op}$ is additive.

Suppose $\sA$ and $\sB$ to be abelian. The functor $F$ is left-exact if and only if $F^{\op}$ is right-exact. \\ The functor $F$ is exact if and only if $F^{\op}$ is exact.
\end{itemize}

\end{rem}

\section{Additive lemmata}

\begin{rem}
\label{rem:directsumcompatible}
Suppose given a preadditive category $\sA$.

The two definitions of direct sums given in the conventions are compatible:\\
Given a diagram $\xymatrix@1{A\,\ar@<2pt>[r]^i   &\,C\,\ar@<2pt>[l]^{p} \ar@<-2pt>[r]_q&\,B\ar@<-2pt>[l]_{j}}$ in $\sA$ satisfying $ip=1_A$, $jq=1_B$ and $pi+qj=1_C$, we have $iq=0$ and $jp=0$. Cf. convention \ref{conv:directsum}.
\end{rem}

\begin{proof}
Indeed we have
\begin{align*}
0 = iq-iq = i(pi+qj)q - iq = ipiq+iqjq-iq= iq + iq - iq = iq
\end{align*}
and similarly 
\begin{align*}
0=jp - jp = j(pi+qj)p-jp = jpip+jqjp - jp =  jp +jp -jp=jp.
\end{align*}
\end{proof}

\begin{prop} Suppose given a preadditive category $\sA$.

Suppose given $n \in \N_0$ and a direct sum $C=\big( C,(i_k)_{k \in [1,n]},(p_k)_{k\in [1,n]} \big)$ of $A_1, \dots, A_n$ in $\sA$, cf. convention \ref{conv:directsum}.

Then $C$ satisfies the universal properties of a product and a coproduct:
\begin{itemize}
	\item[(a)] Given morphisms $f_k \colon X \to A_k$ for $k \in [1,n]$, there exists a unique morphism $g \colon X \to C$ satisfying $f_k=gp_k$ for $k \in [1,n]$.
\begin{align*}
\xymatrix{&A_k\\X\ar[ur]^{f_k} \ar[r]_g & C \ar[u]_{p_k}}
\end{align*}
\item[(b)] Given morphisms $f_k \colon A_k \to X$ for $k \in [1,n]$, there exists a unique morphism $g \colon C \to X$ satisfying $f_k=i_kg$ for $k \in [1,n]$.
\begin{align*}
\xymatrix{A_k\ar[dr]^{f_k} \ar[d]_{i_k} &\\ C \ar[r]_{g}&X}
\end{align*}
\end{itemize}

\end{prop}

\begin{proof}
Ad (a).
If $g \colon X \to C$ is a morphism with $f_k=gp_k$ for $k \in [1,n]$, we necessarily have $g=g(\sum_{\ell=1}^n p_\ell i_\ell)=\sum_{\ell=1}^n gp_\ell i_\ell  = \sum_{\ell=1}^n f_\ell i_\ell$, so $g$ is completely determined and therefore unique.

Let $g:=\sum_{\ell=1}^n f_\ell i_\ell$. We have $g p_k = (\sum_{\ell=1}^n f_\ell i_\ell)p_k = \sum_{\ell=1}^n f_\ell i_\ell p_k = f_k$ for $k \in [1,n]$.

Ad (b). This is dual to (a).
\end{proof}

\begin{prop} \label{prop:additivefunctor}
Suppose given additive categories $\sA$ and $\sB$. Suppose given an additive functor $F \colon \sA \to \sB$.

We have $F(0_{\sA})\cong 0_{\sB}$ in  $\sB$.

Suppose given $A,B \in \Ob \sA$. Let $p := p_1^{(A,B)}$, $q :=p_2^{(A,B)}$,  $i := i_1^{(A,B)}$ and  $j := i_2^{(A,B)}$, cf. convention \ref{conv:directsumchoice}. So $\xymatrix@1{A\,\ar@<2pt>[r]^-i   &\,A \oplus B\, \ar@<2pt>[l]^-{p} \ar@<-2pt>[r]_-q&\,B\ar@<-2pt>[l]_-{j}}$ is a direct sum of $A$ and $B$ in $\sA$.

The morphism $ \left(\begin{smallmatrix}F(p) &F(q)\end{smallmatrix}\right) \colon F(A\oplus B)\to F(A) \oplus F(B)$ is an isomorphism in $\sB$ with inverse $\left(\begin{smallmatrix}F(i)\\F(j)\end{smallmatrix}\right)$.
\end{prop}

\begin{proof}
For $A \in \Ob \sA$ and $0=0_A$, we have $F(0) = F(0) + F(0) - F(0) = F(0+0) - F(0) = F(0)-F(0)=0$. In particular, we have $1_{F(0_{\sA})}=F(1_{0_{\sA}})=F(0)=0$ and therefore $F(0_{\sA})\cong 0_{\sB}$ holds.

Calculating
\begin{align*}
\left(\begin{smallmatrix}F(p) &F(q)\end{smallmatrix}\right) \left(\begin{smallmatrix}F(i)\\F(j)\end{smallmatrix}\right) =  F(p) F( i) + F(q) F(j) =F(p i + q j) = F(1)=1
\end{align*}
and
\begin{align*}
\left(\begin{smallmatrix}F(i)\\F(j)\end{smallmatrix}\right)\left(\begin{smallmatrix}F(p) &F(q)\end{smallmatrix}\right) = \left(\begin{smallmatrix}F(i)F(p) &F(i)F(q)\\F(j) F(p)&F(j) F(q)\end{smallmatrix}\right) = \left(\begin{smallmatrix}F(i p) &F(i q)\\F(j p)&F(j q)\end{smallmatrix}\right)=  \left(\begin{smallmatrix}F(1) &F(0)\\F(0)&F(1)\end{smallmatrix}\right) =\left(\begin{smallmatrix}1 &0\\0&1\end{smallmatrix}\right)
\end{align*}
proves the second claim, cf. remark \ref{rem:directsumcompatible}.
\end{proof}

\begin{rem} \label{rem:functorcategory}
Suppose given a category $\sC$ and an additive category $\sA$. Some facts about the functor category $\sA^{\sC}$ are listed below. Cf. convention \ref{conv:functorcategory}.

\begin{itemize}
	\item[(a)] The category $\sA^{\sC}$ is additive.
	\item[(b)] Suppose given $\xymatrix@1{F\, \ar[r]^{\alpha} &\,G\,\ar[r]^{\beta} & \,H}$ in $\sA^{\sC}$. \\ The composite $\alpha \beta$ is given by $\alpha\beta = (\alpha_X \beta_X)_{X \in \Ob \sC}$. 
	\item[(c)] Suppose given $F \in \Ob(\sA^{\sC})$. The identity morphism $1_F$ is given by $1_F = (1_{F(X)})_{X \in \Ob \sC}$. 
	\item[(d)] Suppose given $\xymatrix@1{ F \,\ar@<2pt>[r]^{\alpha} \ar@<-2pt>[r]_{\beta} & \,G}$ in $\sA^{\sC}$. The sum $\alpha+\beta$ is given by $\alpha +\beta = (\alpha_X + \beta_X)_{X \in \Ob \sC}$.
	\item[(e)] Suppose given $F,G \in \Ob (\sA^{\sC})$. \\
	The zero morphism $0_{F,G}$ is given by $0_{F,G} = (0_{F(X),G(X)})_{X \in \Ob \sC}$.
	\item[(f)] Suppose given $n \in \N_0$ and $F_k \in \Ob (\sA^{\sC})$ for $k \in [1,n]$. We choose the direct sum $\Bigg( \bigoplus\limits_{k=1}^n F_k , \Big(i_{\ell}^{(F_k)_{k \in [1,n]}} \Big)_{\ell \in [1,n]} ,\Big(p_{\ell}^{(F_k)_{k \in [1,n]}} \Big)_{\ell \in [1,n]}  \Bigg)$ of $F_1, \dots, F_n$, where 
	\begin{align*}
	\left( \bigoplus_{k=1}^n F_k \right) \big(\xymatrix{X \ar[r]^f &Y} \big) =
 \raisebox{6mm}{$\left( \rule{0cm}{13mm} \right.$}
 \xymatrix{\bigoplus\limits_{k=1}^n F_k(X) \ar[rrrr]^{\sm{F_1(f) & & & \\ &F_2(f) & & \\ & & \ddots & \\&&&F_n(f) }}
	 & & & & \bigoplus\limits_{k=1}^n F_k(Y)} 
 \raisebox{6mm}{$\left. \rule{0cm}{13mm} \right)$}
	\end{align*}
	for $\xymatrix@1{X \,\ar[r]^f & \,Y}$ in $\sC$,
\begin{align*}
i_{\ell}^{(F_k)_{k \in [1,n]}} = \Big( i_{\ell}^{(F_k(X))_{k \in [1,n]}} \Big)_{X \in \Ob \sC}
\end{align*}
for $\ell \in [1,n]$ and
\begin{align*}
p_{\ell}^{(F_k)_{k \in [1,n]}} = \Big( p_{\ell}^{(F_k(X))_{k \in [1,n]}} \Big)_{X \in \Ob \sC}
\end{align*}
for $\ell \in [1,n]$.

Cf. conventions \ref{conv:directsum} and \ref{conv:directsumchoice}.
\end{itemize}
\end{rem}

\begin{lem}
\label{lem:transformations}
Suppose given $\xymatrix@1{ \sA\,  \ar@/^5mm/[rr]^F="F"   \ar[rr]|{\phantom{\rule{0.4mm}{0mm}}G\phantom{\rule{0.4mm}{0mm}}}="G" \ar@/_5mm/[rr]_H="H" & &
\,\sB\, \ar@/^5mm/[rr]^K="K"   \ar[rr]|{\phantom{\rule{0.4mm}{0mm}}L\phantom{\rule{0.4mm}{0mm}}}="L" \ar@/_5mm/[rr]_M="M" & & \,\sC\, \ar@/^/[r]^N="N" \ar@/_/[r]_P="P" & \,\sD 
\ar@2"F"+<0mm,-2.6mm>;"G"+<0mm,1.4mm>_\alpha
\ar@2"G"+<0mm,-1.5mm>;"H"+<0mm,2.4mm>_\beta
\ar@2"K"+<0mm,-2.6mm>;"L"+<0mm,1.4mm>_\gamma
\ar@2"L"+<0mm,-1.5mm>;"M"+<0mm,2.4mm>_\delta
\ar@2"N"+<0mm,-2.6mm>;"P"+<0mm,2.4mm>_\varepsilon}$.

The equations $\varepsilon \star (\gamma \star \alpha) = (\varepsilon \star \gamma) \star \alpha$, $1_{K \circ F} = 1_K \star 1_F$ and $(\gamma \delta) \star(\alpha \beta) = (\gamma \star \alpha)(\delta \star \beta)$ hold.
\end{lem}

\begin{proof}
Suppose given $X \in \Ob \sA$. We have
\begin{align*}
(\varepsilon \star(\gamma \star \alpha))_X &= N( (\gamma \star \alpha)_X) \varepsilon_{(L\circ G)(X)} = N(K(\alpha_X)\gamma_{G(X)}) \varepsilon_{L(G(X))}\\
& = N(K(\alpha_X)) N(\gamma_{G(X)}) \varepsilon_{L(G(X))} = (N \circ K)(\alpha_X) (\varepsilon \star \gamma)_{G(X)}\\
&=((\varepsilon\star \gamma)\star \alpha)_X,
\end{align*}
\begin{align*}
(1_{K \circ F})_X = 1_{(K\circ F)(X)} = K(1_{F(X)}) (1_K)_{F(X)}= (1_K \star 1_F)_X
\end{align*}
and
\begin{align*}
((\gamma \delta) \star (\alpha \beta))_X &= K( (\alpha\beta)_X) (\gamma \delta)_{H(X)} = K(\alpha_X)K(\beta_X) \gamma_{H(X)} \delta_{H(X)} \\
&=K(\alpha_X) \gamma_{G(X)} L(\beta_X) \delta_{H(X)} = (\gamma \star \alpha)_X (\delta\star \beta)_X\\
&=((\gamma \star \alpha)(\delta \star \beta))_X.
\end{align*}
\end{proof}

\section{Abelian lemmata}

\begin{lem} \label{lem:inducedmorphisms}
Suppose given a preadditive category $\sA$ and the following commutative diagram in $\sA$.
\begin{align*}
\xymatrix{
A \ar[r]^f \ar[d]^g & B \ar[d]^h\\
A' \ar[r]^{f'} & B'
}
\end{align*}
\begin{itemize}
	\item[(a)] Suppose given a kernel $k\colon K \to A$ of $f$ and a kernel $k'\colon K' \to A'$ of $f'$.
	
	There exists a unique morphism $u \colon K \to K'$ such that $kg=uk'$ holds. The morphism $u$ is called the induced morphism between the kernels.
	\begin{align*}
\xymatrix{
 K \arm[r]^k  \ar[d]^u & A \ar[r]^f \ar[d]^g & B \ar[d]^h\\
 K' \arm[r]^{k'}  & A' \ar[r]^{f'} & B'
}
\end{align*}
	If $g$ and $h$ are isomorphisms, then $u$ is an isomorphism too.
	\item[(b)] Suppose given a cokernel $c \colon B \to C$ of $f$ and cokernel $c' \colon B' \to C'$ of $f'$.
	
	There exists a unique morphism $u \colon C \to C'$ such that $c u = h c'$ holds. The morphism $u$ is called the induced morphism between the cokernels.
	\begin{align*}
\xymatrix{
A \ar[r]^f \ar[d]^g & B \ar[d]^h \are[r]^*+<0.8mm,0.8mm>{\scriptstyle c} & C \ar[d]^u\\
A' \ar[r]^{f'} & B' \are[r]^*+<0.8mm,0.8mm>{\scriptstyle c'} & C'
}
\end{align*}
	If $g$ and $h$ are isomorphisms, then $u$ is an isomorphism too.
\end{itemize}
\end{lem}

\begin{proof}
Ad (a). We have $k g f' = k f h = 0$. Since $k'$ is a kernel of $f'$, there exists a morphism $u \colon K \to K'$ such that $kg= uk'$ holds. The morphism $u$ is unique because $k'$ is a monomorphism.

Now suppose $g$ and $h$ to be isomorphisms. Consider the following diagram.
\begin{align*}
\xymatrix{
A' \ar[r]^{f'} \ar[d]^{g^{-1}} & B' \ar[d]^{h^{-1}}\\
A \ar[r]^{f} & B
}
\end{align*}
This diagram commutes, since we have $g^{-1} f = g^{-1} f h h^{-1} = g^{-1} g f' h^{-1} = f' h^{-1}$.

As seen above, there exists a morphism $u' \colon K' \to K$ such that $k' g^{-1} = u'k$ holds.

 We have $u u' k = u k'g^{-1} = k g g^{-1} = k$ and $u' u k' = u' k g = k' g^{-1} g = k'$. Since $k$ and $k'$ are monomorphic, we conclude that $u u' = 1$ and $u' u =1$ hold.
 
 Ad (b). This is dual to (a).
\end{proof}

\begin{rem} \label{rem:kernelcokernelfactorisation}
Suppose given a preadditive category $\sA$ and $f \in \Mor \sA$.
Suppose given the following kernel-cokernel-factorisations of $f$, cf. convention \ref{conv:kernelcokernelfactorisation}.
\begin{align*}
\xymatrix{
K \arm[r]^k & A \are[dr]_p \ar[rrr]^f & & & B \are[r]^*+<0.8mm,0.8mm>{\scriptstyle c} &C
\\
&& J \ar[r]^{\hat f} & I \arm[ru]_i
}
& \phantom{abcdefg} \xymatrix{
K' \arm[r]^{k'} & A \are[dr]_{p'} \ar[rrr]^f & & & B \are[r]^*+<0.8mm,0.8mm>{\scriptstyle c'} &C'
\\
&& J' \ar[r]^{\hat{f'}} & I' \arm[ru]_{i'}
}
\end{align*}

By adding the induced morphisms between the kernels resp. cokernels, we obtain the following diagram.

\begin{align*}
\xymatrix{
K \arm[rr]^k \ar[dd]^{u} && A \ar[dd]_{1_A}\are[drr]_p \ar[rrrrr]^{f} & & & & & B \ar[dd]^{1_B} \are[rr]^*+<0.8mm,0.8mm>{\scriptstyle c}  &&C\ar[dd]_{v}
\\
&&&& J \ar[r]^{\hat f} \ar[dd]_(0.3){x} & I \arm[rur]_i \ar[dd]^(0.3){y}
\\
K' \arm[rr]^{k'} && A \are[drr]_{p'} \ar[rrrrr]|(0.39)\hole|(0.612)\hole^(.15)f & & & & & B \are[rr]^*+<0.8mm,0.8mm>{\scriptstyle c'} &&C'
\\
&&&& J' \ar[r]^{\hat{f'}} & I' \arm[rur]_{i'}
}
\end{align*}

This diagram commutes, in particular we have $\hat{f} y= x\hat{f'}$. 
Since $u$, $v$, $x$ and $y$ are isomorphisms, the morphism $\hat{f}$ is an isomorphism if and only if $\hat{f'}$ is an isomorphism. Cf. lemma \ref{lem:inducedmorphisms}.

\end{rem}

\begin{proof}
We have $p \hat f y i' = p \hat f i = f = p' \hat{f'} i' = px \hat{f'} i'$. Since $p$ is epimorphic and $i'$ is monomorphic, we obtain $\hat f y = x \hat{f'}$.
\end{proof}

\begin{rem}\label{rem:kernelcokernelabeliancategory}
Suppose given an abelian category $\sA$.
\begin{itemize}
	\item[(a)]  Suppose given an epimorphism $f \colon A \to B$ and a kernel $k \colon K \to A$ of $f$. Then $f$ is a cokernel of $k$.
	\item[(b)] Suppose given a monomorphism $f \colon A \to B$ and a cokernel $c \colon B \to C$ of $f$. Then $f$ is a kernel of $c$.
\end{itemize}
\end{rem}

\begin{proof} Ad (a).
We have a kernel-cokernel-factorisation of $f$ as follows.
\begin{align*}
\xymatrix{
K \arm[r]^k & A \are[dr]_p \are[rrr]^*+<0.8mm,0.8mm>{\scriptstyle f} & & & B \are[r]^*+<0.8mm,0.8mm>{\scriptstyle 0} &0
\\
&& J \ar[r]^{\hat f} & B \arm[ru]_{1}
}
\end{align*}
Since $\hat f$ is an isomorphism and $p$ is a cokernel of $k$, we conclude that $f$ is a cokernel of $k$ too.

Ad (b). This is dual to (a).
\end{proof}

\begin{lem} \label{lem:imagekernelcokernel}
Suppose given an abelian category $\sA$. Suppose given $\xymatrix@1{A \,\ar[r]^f &\,B}$ in $\sA$, a kernel $k$ of $f$, a cokernel $c$ of $f$ and an image $\xymatrix@1{A \,\are[r]^*+<0.6mm,0.6mm>{\scriptstyle p} & \,I\, \arm[r]^i & \,B}$ of $f$. 
\begin{itemize}
	\item[(a)] The morphism $p$ is a cokernel of $k$.
	\item[(b)] The morphism $i$ is a kernel of $c$.
\end{itemize}
\end{lem}

\begin{proof} Ad (a).
The morphism $k$ is also a kernel of $p$ since $i$ is monomorphic. Therefore $p$ is a cokernel of $k$, cf. remark \ref{rem:kernelcokernelabeliancategory}~(a).

Ad (b). This is dual to (a).
\end{proof}

\begin{lem} \label{lem:inducedimage}
Suppose given an abelian category $\sA$ and the following commutative diagram in $\sA$.
\begin{align*}
\xymatrix{
A \ar[r]^f \ar[d]^x & B \ar[d]^y\\
C \ar[r]^{g} & D
}
\end{align*}
Suppose given an image $\xymatrix@1{A \,\are[r]^*+<0.6mm,0.6mm>{\scriptstyle p} &\, I\, \arm[r]^i &\,B}$ of $f$ and an image $\xymatrix@1{C \,\are[r]^*+<0.6mm,0.6mm>{\scriptstyle q} &\,J\, \arm[r]^j &\,D}$ of $g$.
	
	There exists a unique morphism $u \colon I \to J$ such that the following diagram commutes.
	\begin{align*}
\xymatrix{
 A \are[r]^*+<0.6mm,0.6mm>{\scriptstyle p} \ar[d]^x & I\ar[d]^u \arm[r]^i &B\ar[d]^y\\
C \are[r]^*+<0.6mm,0.6mm>{\scriptstyle q} & J \arm[r]^j &D
}
\end{align*}
The morphism $u$ is called the induced morphism between the images.

If $x$ and $y$ are isomorphisms, then $u$ is an isomorphism too.
\end{lem}

\begin{proof}
Let $k \colon K \to A$ be a kernel of $f$ and let $\ell \colon L \to C$ be a kernel of $g$. Let $v \colon K \to L$ be the induced morphism between the kernels, cf. lemma \ref{lem:inducedmorphisms}. Lemma \ref{lem:imagekernelcokernel} says that $p$ is a cokernel of $k$ and that $q$ is a cokernel of $\ell$. Let $u\colon I \to J$ be the induced morphism between the cokernels.
	\begin{align*}
\xymatrix{
 K \ar[d]^v\arm[r]^k & A \are[r]^*+<0.6mm,0.6mm>{\scriptstyle p} \ar[d]^x & I\ar[d]^u \arm[r]^i &B\ar[d]^y\\
L \arm[r]^\ell  &C \are[r]^*+<0.6mm,0.6mm>{\scriptstyle q} & J \arm[r]^j &D
}
\end{align*}
It remains to show that $uj=iy$ holds. We have $p u j = x q j = xg=fy= p i y$, so $uj=iy$ is true since $p$ is epimorphic.

The morphism $u$ is unique because $p$ is epimorphic and we necessarily have $pu=xq$.

If $x$ and $y$ are isomorphisms, then $v$ is an isomorphism too. In this case $v$ and $x$ are isomorphisms, so $u$ is an isomorphism too. Cf. lemma \ref{lem:inducedmorphisms}.
\end{proof}

\begin{rem} \label{rem:exactsequence}
Suppose given an abelian category $\sA$ and the following commutative diagram in $\sA$.
\begin{align*}
\xymatrix{
A\ar[rr]^f \are[dr]_p &&B \ar[rr]^g \are[dr]_q &&C
\\
&I \arm[ur]_i && J \arm[ur]_j
}
\end{align*}
The following statements are equivalent.
\begin{itemize}
	\item[(a)] The sequence $\xymatrix@1{A\, \ar[r]^f &\,B\, \ar[r]^g &\, C}$ is exact.
	\item[(b)] The sequence $\xymatrix@1{I \,\arm[r]^i &\,B\, \are[r]^*+<0.6mm,0.6mm>{\scriptstyle q} &\, J}$ is short exact.
	\item[(c)] The sequence $\xymatrix@1{I\, \arm[r]^i &\,B\, \are[r]^*+<0.6mm,0.6mm>{\scriptstyle q} &\, J}$ is left-exact.
	\item[(d)] The sequence $\xymatrix@1{I \,\arm[r]^i &\,B\, \are[r]^*+<0.6mm,0.6mm>{\scriptstyle q} & \,J}$ is right-exact.
	\item[(e)] The sequence $\xymatrix@1{I\, \arm[r]^i &\,B\, \ar[r]^g & \,C}$ is left-exact.
	\item[(f)] The sequence $\xymatrix@1{A\, \ar[r]^f &\,B\, \are[r]^*+<0.6mm,0.6mm>{\scriptstyle q} &\, J}$ is right-exact.
\end{itemize}
\end{rem}

\begin{lem} \label{lem:exactfunctor}
Suppose given abelian categories $\sA$ and $\sB$ and an additive functor $F \colon \sA \to \sB$.
\begin{itemize}
	\item[(a)] Suppose that for  $\xymatrix@1{X \,\ar[r]^f&\,Y}$ in $\sA$, there exists a kernel $k\colon K \to X$ of $f$ such that the sequence $\xymatrix@1{F(K) \,\ar[r]^-{F(k)} &\, F(X) \,\ar[r]^-{F(f)} & \,F(Y)}$ is left-exact in $\sB$. Then $F$ is left-exact.
	\item[(b)] Suppose that for  $\xymatrix@1{X \,\ar[r]^f&\,Y}$ in $\sA$, there exists a cokernel $c\colon Y \to C$ of $f$ such that the sequence $\xymatrix@1{ F(X)\, \ar[r]^-{F(f)} & \,F(Y)\, \ar[r]^-{F(c)}&\,F(C)}$ is right-exact in $\sB$. Then $F$ is right-exact.
\end{itemize}
\end{lem}

\begin{proof} Ad (a).
Suppose given a left-exact sequence $\xymatrix@1{L\, \ar[r]^\ell &\, X\,\ar[r]^f & \,Y}$ in $\sA$. There exists a kernel $k\colon K \to X$ of $f$ such that the sequence $\xymatrix@1{F(K)\, \ar[r]^-{F(k)} &\, F(X)\, \ar[r]^-{F(f)} & \,F(Y)}$ is left-exact. There exists an isomorphism $u \colon K \to L$ such that $u\ell = k$, cf. lemma \ref{lem:inducedmorphisms}~(a). Since $F(u)$ is also an isomorphism and $F(u) F(\ell) = F(k)$ holds, we conclude that $F(\ell)$ is a kernel of $F(f)$, so $\xymatrix@1{F(L)\, \ar[r]^-{F(\ell)} & \,F(X)\, \ar[r]^-{F(f)} & \,F(Y)}$ is left-exact too.

Ad (b). This is dual to (a).
\end{proof}

\chapter{Factor categories}
\thispagestyle{empty}
\label{ch:factorcategories}

Suppose given an additive category $\sA$.

\begin{defn}[Ideal]
\label{defn:ideal}
Suppose given $J \subseteq \Mor \sA$. We set $\Hom_{\sA,J}(A,B):=\Hom_\sA(A,B) \cap J$ for $A,B \in \Ob \sA$.\\
We say that $J$ is an \emph{ideal} in $\sA$ if it satisfies the following two conditions.
\begin{itemize}
\item[(I1)] \label{I1} Given $\xymatrix@1{A\,\ar[r]^a &\, B\,\ar[r]^j &\, C\,\ar[r]^c & \,D}$ in $\sA$ with $j \in J$, we have $ajc \in J$.
\item[(I2)] \label{I2} $\Hom_{\sA,J}(A,B)$ is a subgroup of $\Hom_{\sA}(A,B)$ for $A,B \in \Ob \sA$.
\end{itemize}
\end{defn}

\begin{rem}
\label{rem:I2}
Suppose given $J \subseteq \Mor \sA$ satisfying condition \hyperref[I1]{(I1)} from the previous definition~\ref{defn:ideal}. Then condition \hyperref[I2]{(I2)} is equivalent to the following two conditions.
\begin{itemize}
\item[(I2$'$)]\label{I2'}  $\Hom_{\sA,J}(A,B) \not = \emptyset$ for $A,B \in \Ob \sA$.
\item[(I2$''$)]\label{I2''} Given $\xymatrix@1{A\,\ar@<2pt>[r]^{j} \ar@<-2pt>[r]_{k}  &\,B}$ in $\sA$ with $j,k\in J$, we have $\left(\begin{smallmatrix}j & 0\\0 & k\end{smallmatrix}\right) \in \Hom_{\sA,J}(A\oplus A,B \oplus B)$.
\end{itemize}
\end{rem}

\begin{proof}
Suppose that $J$ satisfies the conditions \hyperref[I2']{(I2$'$)} and \hyperref[I2'']{(I2$''$)}. Given $\xymatrix@1{A\,\ar@<2pt>[r]^{j} \ar@<-2pt>[r]_{k}  &\,B}$ in $\sA$  with $j,k \in J$, conditions \hyperref[I1]{(I1)} and \hyperref[I2'']{(I2$''$)} imply $j-k= \left(\begin{smallmatrix}1&1\end{smallmatrix}\right) \left(\begin{smallmatrix}j & 0\\0 & k\end{smallmatrix}\right) \left(\begin{smallmatrix}1\\-1\end{smallmatrix}\right) \in J$. Condition  \hyperref[I2']{(I2$'$)} says that $\Hom_{\sA,J}(A,B) \not = \emptyset$. Therefore $\Hom_{\sA,J}(A,B)$ is a subgroup of $\Hom_{\sA}(A,B)$. We conclude that condition \hyperref[I2]{(I2)} holds.

\smallskip

Suppose that $J$ satisfies the condition \hyperref[I2]{(I2)}. Then condition \hyperref[I2']{(I2$'$)} holds. Given $\xymatrix@1{A\,\ar@<2pt>[r]^{j} \ar@<-2pt>[r]_{k}  &\,B}$ in $\sA$  with $j,k \in J$, condition \hyperref[I1]{(I1)} implies $\left(\begin{smallmatrix}j & 0\\0 & 0\end{smallmatrix}\right) = \left(\begin{smallmatrix}1\\0\end{smallmatrix}\right)j \left(\begin{smallmatrix}1 & 0\end{smallmatrix}\right) \in \Hom_{\sA,J}(A\oplus A,B \oplus B)$ and 
$\left(\begin{smallmatrix}0 & 0\\0 & k\end{smallmatrix}\right) = \left(\begin{smallmatrix}0\\1\end{smallmatrix}\right)k \left(\begin{smallmatrix}0&1\end{smallmatrix}\right) \in \Hom_{\sA,J}(A\oplus A,B \oplus B)$.\\
 Therefore we have $\left(\begin{smallmatrix}j & 0\\0 & k\end{smallmatrix}\right) = \left(\begin{smallmatrix}j & 0\\0 & 0\end{smallmatrix}\right) + \left(\begin{smallmatrix}0 & 0\\0 & k\end{smallmatrix}\right) \in \Hom_{\sA,J}(A\oplus A,B \oplus B)$ according to \hyperref[I2]{(I2)}. So condition \hyperref[I2'']{(I2$''$)} holds.
\end{proof}

Suppose given an ideal $J$ in $\sA$ throughout the rest of this chapter \ref{ch:factorcategories}.

\begin{lemdefn}[Factor category]
\label{lemdefn:factorcategory}
The \emph{factor category} $\sA/J$ shall be defined as follows.
\begin{itemize}
\item Let $\Ob (\sA /J) := \Ob \sA$.
\item Let $\Hom_{\sA/J}(A,B) := \Hom_\sA(A,B)/\Hom_{\sA,J}(A,B)$ for $A,B \in \Ob (\sA/J)$, cf. definition~\ref{defn:ideal}.\\
 We write $[a]:=a+\Hom_{\sA,J}(A,B)$ for $a \in  \Hom_\sA(A,B)$.
\item Composites and identities are defined via representatives:\\
 Let $[a][b]:=[ab]$ for $\xymatrix@1{A\,\ar[r]^{a} &\, B\,\ar[r]^{b} &\, C}$ in $\sA$.

Let $1_A:=[1_A]$  for $A \in \Ob (\sA/J)=\Ob \sA$.
\end{itemize}
This in fact defines a category.
\end{lemdefn}

\begin{proof}
First, we show the well-definedness of the composition:\\
Given $\xymatrix@1{A\,\ar@<2pt>[r]^a \ar@<-2pt>[r]_{a'} &\, B\,\ar@<2pt>[r]^b \ar@<-2pt>[r]_{b'} &\, C}$ in $\sA$ with $[a]=[a']$ and $[b]=[b']$, 
we have 
\begin{align*}
ab-a'b'=ab-ab'+ab'-a'b'=a(b-b')1_C + 1_A (a-a')b' \in J
\end{align*}
because of $b-b' \in J$, $a-a' \in J$ and the conditions \hyperref[I1]{(I1)} and \hyperref[I2]{(I2)}. We conclude that $[ab]=[a'b']$ holds.\\
Associativity and properties of the identities are inherited from $\sA$:\\
Given $\xymatrix@1{A\,\ar[r]^{a} &\, B\,\ar[r]^{b} & \,C\,\ar[r]^{c} &\, D}$ in $\sA$, we obtain 
$([a][b])[c]=[(ab)c]=[a(bc)]=[a]([b][c])$  and $1_A [a]=[1_A a]=[a] = [a 1_B] = [a] 1_B$.
\end{proof}

\begin{nota}
\label{nota:brackets}
When writing $[a] \in \Mor(\sA/J)$, we always suppose $a \in \Mor \sA$.
\end{nota}

\begin{rem} \label{rem:factorfullsubcategory}
Suppose given a full additive subcategory $\sN$ of $\sA$. Let $\F_\sN$ be the set of all $ f \in \Mor \sA $ such that there exists a factorisation $\big(\xymatrix@1{A \,\ar[r]^{f} & \,B}\big) =  \big(\xymatrix@1{A\, \ar[r]^{g} &\,N\, \ar[r]^{h} &\, B}\big)$  with $N \in \Ob \sN$. The set $\F_\sN$ is an ideal in $\sA$. We shortly write $\sA / \sN := \sA / \F_{\sN}$.
%
\end{rem}

\begin{proof}
We use remark \ref{rem:I2} to show that $\F_{\sN}$ is an ideal in $\sA$.

Given $\xymatrix@1{A\, \ar[r]^a &\,B\, \ar[r]^j &\,C\, \ar[r]^c &\, D}$ in $\sA$ with $j \in \F_{\sN}$, there exists a factorisation $\big(\xymatrix@1{B\, \ar[r]^j&\,C} \big) = \big( \xymatrix@1{B\, \ar[r]^g &\,N\, \ar[r]^h &\, C}\big)$ with $N \in \Ob \sN$.

We have 
\begin{align*}
\big( \xymatrix{A \ar[r]^a &B \ar[r]^j &C \ar[r]^c & D} \big) = \big( \xymatrix{A \ar[r]^a &B \ar[r]^g &N\ar[r]^h&C \ar[r]^c & D} \big) = \big(\xymatrix{A \ar[r]^{ag} &N \ar[r]^{hc} & D} \big).
\end{align*}

Therefore $ajc \in \F_{\sN}$ is true, so condition \hyperref[I1]{(I1)} holds.

There exists a zero object $0 \in \Ob \sN$ since $\sN$ is a full additive subcategory of $\sA$. We conclude that $\Hom_{\sA,\sN}(A,B) \neq \emptyset$ holds for $A,B \in \Ob \sA$ since we have the factorisation $\big( \xymatrix@1{A\, \ar[r]^0 &\,B}\big) = \big(\xymatrix@1{A\, \ar[r]^{0} & \,0 \,\ar[r]^{0} & \,B}\big)$ with $0 \in \Ob \sN$, so condition \hyperref[I2']{(I2$'$)} holds.

Suppose given $\xymatrix@1{A\,\ar@<2pt>[r]^{j} \ar@<-2pt>[r]_k&\,B}$ in $\sA$ with $j,k \in \F_{\sN}$. There exist factorisations $\big(\xymatrix@1{A\, \ar[r]^j &\,B} \big) = \big(\xymatrix@1{A\,\ar[r]^g &\,N\, \ar[r]^h&\,B } \big)$ and $\big(\xymatrix@1{A\, \ar[r]^k &\,B}\big) = \big( \xymatrix@1{A \,\ar[r]^u &\,S\, \ar[r]^v &\,B} \big)$ with $N,S \in \Ob \sN$. A direct sum $C \cong N \oplus S$ of $N$ and $S$ exists in $\sN$  since $\sN$ is a full additive subcategory of $\sA$.

Finally, we have
\begin{align*}
\Big( \xymatrix{A \oplus A \ar[r]^{\sm{j&0\\0&k}}&B \oplus B} \Big) = \Big(\xymatrix{A \oplus A \ar[r]^-{\sm{g&0\\0&u}}&C\ar[r]^-{\sm{h&0\\0&v}}& B\oplus B} \Big)
\end{align*}
with $C \in \Ob \sN$, so condition \hyperref[I2'']{(I2$''$)} holds.
\end{proof}

\begin{prop}
The factor category $\sA/J$ is additive.
\label{prop:factorcategoryadditive}
\end{prop}

\begin{proof}
First of all $\Hom_{\sA/J}(A,B)$ is an abelian group for $A,B \in \Ob \sA/J$ as a factor group of the abelian group $\Hom_\sA(A,B)$. The compatibility with the composition is seen as follows.\\
Given $\xymatrix@1{A\,\ar[r]^{[a]}& \, B\,\ar@<2pt>[r]^{[b]} \ar@<-2pt>[r]_{[b']}  &\,C \,\ar[r]^{[c]}  &\, D }$ in $\sA/J$, we have
\begin{align*}
[a]([b]+[b'])[c]=[a(b+b')c]=[abc+ab'c]=[a][b][c]+[a][b'][c].
\end{align*}
The zero object $0$ of $\sA$ is also a zero object in $\sA/J$ since factors of one-element hom-groups contain precisely one element.\\
For any two objects $A,B \in \Ob \sA/J=\Ob \sA$, we have a direct sum $\xymatrix@1{A\,\ar@<2pt>[r]^i   &\,C\,\ar@<2pt>[l]^{p} \ar@<-2pt>[r]_q&\,B\ar@<-2pt>[l]_{j}}$ in $\sA$ with $ip=1_A$, $jq=1_B$ and $pi+qj=1_C$.\\
We \emph{claim} that $\xymatrix@1{A\,\ar@<2pt>[r]^{[i]}   &\,C\,\ar@<2pt>[l]^{[p]} \ar@<-2pt>[r]_{[q]}&\,B\ar@<-2pt>[l]_{[j]}}$ defines a direct sum in $\sA/J$ as well. Indeed we have $[i][p]=[ip]=1_A$, $[j][q]=[jq]=1_B$ and $[p][i]+[q][j]=[pi+qj]=1_C$. This proves the \emph{claim}.
\end{proof}

\begin{lemdefn}[Residue class functor] \label{def:residueclassfunctor}
The \emph{residue class functor} $\R_{\sA,J} \colon \sA \to \sA/J$ of $J$ in $\sA$ is defined by $\R_{\sA,J}(A):=A$ for $A \in \Ob \sA$ and $\R_{\sA,J}(a):=[a]$ for $a \in \Mor \sA$. The residue class functor is additive. We write $\R:=\R_{\sA,J}$ if unambiguous.
\end{lemdefn}

\begin{proof}
We have $\R(ab)=[ab]=[a][b]=\R(a)\R(b)$, $\R(1_A)=[1_A]=1_{\R(A)}$ and $\R(a+a')=[a+a']=[a]+[a']=\R(a)+\R(a')$ for 
$\xymatrix@1{A\,\ar@<2pt>[r]^{[a]} \ar@<-2pt>[r]_{[a']}  &\,B\,\ar[r]^{[b]}&\,C}$ in 
in $\sA$.
\end{proof}

\begin{thrm}[Universal property of the factor category] \label{thrm:upf}
Recall that $\sA$ is an additive category and $J$ is an ideal in $\sA$.
Suppose given an additive category $\sB$.

\begin{itemize}
	\item[(a)] Suppose given an additive functor $F \colon \sA \to \sB$ with $F(j)=0$ for $j \in J$. 
	
	There exists a unique additive functor $\hat F \colon \sA/J \to \sB$ satisfying $\hat F \circ \R=F$.
\begin{align*}
\xymatrix{\sA \ar[r]^{F} \ar[d]_{\R}& \sB \\ \sA /J \ar[ur]_{\hat F}}
\end{align*}
The equations $\hat F(X)=F(X)$ and $\hat F([f]) = F(f)$ hold  for $X \in  \Ob \sA$ and $f \in \Mor \sA$ .
	\item[(b)] Suppose given additive functors $F,G \colon \sA \to \sB$ with $F(j)=0$ and $G(j)=0$ for $j \in J$ and a transformation $\alpha \colon F \Rightarrow G$.
	
	There exists a unique transformation $\hat \alpha \colon \hat F \Rightarrow \hat G$ satisfying $\hat \alpha \star \R = \alpha$.

\begin{align*}
\xymatrix{\sA  \ar@/^/[rr]^F="F"\ar@/_/[rr]_G="G" \ar@2"F"+<0mm,-2.6mm>;"G"+<0mm,2.4mm>_\alpha \ar[dd]_{\R}& &\sB\\\\ 
 \sA /J \ar@/^/[uurr]^*+<-2mm,-2mm>{\scriptstyle \hat F}="hF"\ar@/_/[uurr]_*+<-2.3mm,-2.3mm>{\scriptstyle \hat G}="hG" \ar@2"hF"+<1.7mm,-1.8mm>;"hG"+<-1.3mm,1.6mm>_{\hat \alpha}
 }
\end{align*}
	
	The equation $\hat \alpha_X=\alpha_X$ holds  for $X \in \Ob \sA$.
\end{itemize} 
\end{thrm}

\begin{proof} 
Ad (a).	Given a functor $\tilde{F} \colon \sA/J \to \sB$ with $\tilde F \circ \R = F$, we necessarily have $\tilde F(X) = \tilde F(\R(X))= F(X)$ for $X \in \Ob \sA$ and $\tilde F([f])=\tilde F(\R(f))=F(f)$ for $f \in \Mor \sA$, so $\tilde F$ is determined by $F$ and therefore unique.
	
	Now we show the existence of such a functor. 
	
	Let $\hat F(X):=F(X)$ for $X \in \Ob\sA$ and $\hat F([f]):=F(f)$ for $f \in \Mor \sA$.
	
	$\hat F$ is well-defined since in case $[f]=[g]$ for $f,g\in \Mor \sA$, we have $f-g \in J$ and therefore $F(f-g)=0$. The additivity of $F$ then implies $F(f)=F(g)$.
	
	We verify that $\hat F$ is an additive functor:

Given $\xymatrix@1{X\,\ar@<2pt>[r]^{[f]} \ar@<-2pt>[r]_{[f']}  &\,Y\,\ar[r]^{[g]}&\,Z}$ in $\sA/J$, we have 
\begin{align*}
\hat F([f][g])=\hat F([fg])=F(fg)=F(f)F(g)=\hat F([f]) \hat F([g]), \quad \hat F(1_X) = F(1_X)=1_{\hat F(X)}
\end{align*}
and
\begin{align*}
\hat F([f]+[f'])=\hat F([f+f'])=F(f+f')=F(f)+F(f')=\hat F([f]) + \hat F([f']).
\end{align*}

Ad (b). Given a transformation $\tilde \alpha \colon \hat F \to \hat G$ with $\tilde \alpha  \star \R = \alpha$, we necessarily have $\tilde \alpha_X = \tilde \alpha_{\R(X)} = \tilde \alpha_{\R(X)} \hat G( 1_{\R(X)}) = (\tilde \alpha \star \R)_X = \alpha_X$ for $X \in \Ob \sA$, so $\tilde \alpha$ is determined by $\alpha$ and therefore unique.

Now we show the existence of such a transformation. 
	
	Let $\hat \alpha_X = \alpha_X$ for $X \in \Ob\sA$. This defines a transformation since
	\begin{align*}
	\hat \alpha_X \hat G([f]) = \alpha_X G(f) = F(f) \alpha_Y = \hat F([f]) \hat \alpha_Y
	\end{align*}
	holds for $\xymatrix@1{X\,\ar[r]^{[f]}&\,Y}$ in $\sA/J$.
	
	We have $(\hat \alpha \star \R)_X = \hat \alpha_{\R(X)} \hat G (1_{\R(X)}) = \alpha_X$ for $X \in \Ob \sA$, so $\hat \alpha \star \R = \alpha$ holds.
\end{proof}

\chapter{Adelman's construction}
\thispagestyle{empty}
\label{ch:adelmansconstruction}
Suppose given an additive category $\sA$.


Recall that $\Delta_2$ denotes the poset category of $[0,2] \subseteq \Z$. This category has three objects $0$, $1$, $2$ and three non-identical morphisms $\xymatrix@1@C=4mm{0\, \ar[r] &\,1}$, $\xymatrix@1@C=4mm{1\, \ar[r] &\,2}$, $\xymatrix@1@C=4mm{0 \,\ar[r] &\,2}$. Cf. convention \ref{conv:delta}.

\smallskip

Recall that $\sA^{\Delta_2}$ denotes the functor category of all functors from $\Delta_2$ to $\sA$, cf. convention \ref{conv:functorcategory}.

\section{Definitions, notations and duality}

\label{sec:definitions}

\begin{nota}
\label{nota:functor}
Given a diagram $\xymatrix@1{ A_0 \,\ar[r]^{a_0} &\, A_1\, \ar[r]^{a_1} & \, A_2}$ in $\sA$, we obtain a functor $A \in \Ob(\sA^{\Delta_2})$ by setting $A(0):=A_0$, $A(1):=A_1$, $A(2):=A_2$, $A(\xymatrix@1@C=4mm{0\, \ar[r] &\,1}):=a_0$, $A(\xymatrix@1@C=4mm{1\, \ar[r] \,&2}):=a_1$, $A(\xymatrix@1@C=4mm{0 \,\ar[r] &\,2}):=a_0a_1$, $A(\xymatrix@1@C=4mm{0 \,\ar[r]^{1} &\,0}):=1_{A_0}$, $A(\xymatrix@1@C=4mm{1\, \ar[r]^{1} &\,1}):=1_{A_1}$ and $A(\xymatrix@1@C=4mm{2 \,\ar[r]^{1} &\,2}):=1_{A_2}$.

For $A \in \Ob(\sA^{\Delta_2})$, we therefore set $A_0:=A(0)$, $A_1:=A(1)$, $A_2:=A(2)$, $a_0:=A(\xymatrix@1@C=4mm{0\, \ar[r] &\,1})$, $a_1:=A(\xymatrix@1@C=4mm{1 \,\ar[r] &\,2})$ and write $A=\big(\xymatrix@1{ A_0\, \ar[r]^{a_0} &\, A_1 \,\ar[r]^{a_1} &\,  A_2}\big)$.

\smallskip

For a transformation $f \in \Hom_{\sA^{\Delta_2}}(A,B)$, we write $f=(f_0,f_1,f_2)$ instead of $f=
(f_i)_{i \in \Ob(\Delta_2)}$. We also write
\begin{align*}
 \left(\begin{smallmatrix}{\xymatrix{A\ar[d]^f \\B}}\end{smallmatrix} \! \right) = \left(\begin{smallmatrix}\xymatrix{A_0\ar[d]^{f_0} \ar[r]^{a_0} & A_1\ar[d]^{f_1} \ar[r]^{a_1} &  A_2\ar[d]^{f_2} \\ B_0 \ar[r]^{b_0} & B_1 \ar[r]^{b_1} &  B_2}\end{smallmatrix}\right).
\end{align*}

\smallskip

Given $A,B \in \Ob(\sA^{\Delta_2})$ and morphisms $f_0 \colon A_0 \to B_0$, $f_1 \colon A_1 \to B_1$, $f_2 \colon A_2 \to B_2$ in $\sA$ satisfying $a_0 f_1 = f_0 b_0$ and $ a_1 f_2 = f_1 b_1 $, we obtain a transformation $f = (f_0,f_1,f_2)\in \Hom_{\sA^{\Delta_2}}(A,B)$.
\end{nota}

\begin{defn}[null-homotopic]
Suppose given $A,B \in \Ob(\sA^{\Delta_2})$.

A transformation $f \in \Hom_{\sA^{\Delta_2}}(A,B)$ is said to be \emph{null-homotopic} if there exist morphisms

$s \colon A_1 \to B_0$ and $t \colon A_2 \to B_1$ in $\sA$ satisfying $s b_0 + a_1 t = f_1$.
\end{defn}

\begin{lemdefn}[Adelman category] \label{defn:adel}
The functor category $\sA^{\Delta_2}$ is an additive category and $\J_\sA := \{f\in \Mor(\sA^{\Delta_2}) \colon f \text{ is null-homotopic}\}$ is an ideal in $\sA^{\Delta_2}$, cf. definition \ref{defn:ideal}. We call 
\begin{align*}\Adel(\sA) := \sA^{\Delta_2} / \J_\sA
\end{align*}
 the \emph{Adelman category} of $\sA$, cf. definition \ref{lemdefn:factorcategory}. The category $\Adel(\sA)$ is additive.
 
 We abbreviate $\R_{\sA} := \R_{\sA^{\Delta_2}, \J_{\sA}}$, cf. definition \ref{def:residueclassfunctor}.
\end{lemdefn}

\begin{proof}
The additivity of $\sA^{\Delta_2}$ is inherited from $\sA$, cf. remark \ref{rem:functorcategory}.

We show that $\J_{\sA}$ is an ideal in $\sA^{\Delta_2}$. Suppose given $\xymatrix@1{A\,\ar[r]^{f}&\,B\,\ar[r]^j&\,C\,\ar[r]^g&\,D}$ in $\sA^{\Delta_2}$. Since $j$ is null-homotopic, there exist morphisms $s \colon B_1 \to C_0$ and $t \colon B_2  \to C_1$ in $\sA$ with $s c_0 + b_1 t = j_1$.

Using $f_1sg_0\colon A_1 \to D_0$ and $f_2tg_1 \colon A_2 \to D_1$, we obtain 
\begin{align*}
f_1sg_0 d_0 + a_1 f_2tg_1 &= f_1s c_0 g_1 + f_1b_1tg_1\\
&=f_1(sc_0 + b_1 t )g_1\\
&=f_1j_1g_1\\
&=(fjg)_1.
\end{align*}
Therefore $fjg$ is null-homotopic, so condition \hyperref[I1]{(I1)} holds.
\begin{align*}
\xymatrix{A_0\ar[d]^{f_0} \ar[rr]^{a_0} && A_1\ar[d]^{f_1} \ar[rr]^{a_1} &&  A_2\ar[d]^{f_2} \\ B_0 \ar[rr]^{b_0} \ar[d]^{j_0}&& B_1\ar[d]^{j_1} \ar[rr]^{b_1} \ar[dll]_{s}&&  B_2\ar[d]^{j_2}\ar[dll]_{t}\\ C_0\ar[d]^{g_0} \ar[rr]_{c_0} && C_1\ar[d]^{g_1} \ar[rr]_{c_1} &&  C_2 \ar[d]^{g_2} \\ D_0 \ar[rr]^{d_0} && D_1 \ar[rr]^{d_1} &&  D_2}
\end{align*}
We have $\Hom_{\sA^{\Delta_2},\J_{\sA}}(A,B) \neq \emptyset$, since $0=0_{A,B}$ is null-homotopic: $0_{A_1,B_0}b_0 + a_1 0_{A_2,B_1} = 0_{A_1,B_1}=0_1$. 

Suppose given $x,y \in \Hom_{\sA^{\Delta_2},J}(A,B)$. Since $x$ and $y$ are null-homotopic, there exist morphisms $s,u \colon A_1 \to B_0$ and $t,v \colon A_2 \to B_1$ in $\sA$ with $sb_0 + a_1t=x_1$ and $ub_0+ a_1v=y_1$.

Using $s-u \colon A_1 \to B_0$ and $t-v \colon A_2 \to B_1$ in $\sA$, we obtain
\begin{align*}
(s-u)b_0 + a_1(t-v)
=\ (sb_0 +a_1t) -(ub_0 + a_1v)
=\ x_1-y_1
=\ (x-y)_1.
\end{align*}
Therefore $x-y$ is null-homotopic, so condition \hyperref[I2]{(I2)} holds.

We conclude that $\J_{\sA}$ is in fact an ideal in $\sA^{\Delta_2}$. 

Proposition \ref{prop:factorcategoryadditive} now implies that $\Adel(\sA)$ is additive.
\end{proof}

\begin{nota}
\label{nota:adel}
Since every morphism in $\Adel(\sA)$ is of the form $[f]$ with $f\in \Mor(\sA^{\Delta_2})$, we will always suppose $f \in \Mor(\sA^{\Delta_2})$ when writing $[f]$ in $\Adel(\sA)$, cf. notation \ref{nota:brackets}. Often we shortly write $[f]=[f_0,f_1,f_2]$ instead of $[f]=[(f_0,f_1,f_2)]$ for $\xymatrix@1{ A \,\ar[r]^{[f]} & \,B}$ in $\Adel(\sA)$.
\end{nota}

\begin{rem} \mbox{} \label{rem:adelbasic}
\begin{itemize}
	\item[(a)] By applying the previous notation \ref{nota:adel}, we have \begin{align*}
& 1_A=[1_{A_0},1_{A_1},1_{A_2}], \quad 0_{A,B}=[0_{A_0,B_0},0_{A_1,B_1},0_{A_2,B_2}],
\\& [f_1,f_2,f_3]\cdot [g_1,g_2,g_3] = [f_1g_1,f_2g_2,f_3g_3] \text{ and }\\ &[f_1,f_2,f_3]+[h_1,h_2,h_3]=[f_1+h_1,f_2+h_2,f_3+h_3]
\end{align*}
for $\xymatrix@1{A \,\ar@<0.4ex>[r]^{[f]}\ar@<-0.4ex>[r]_{[h]}&\,B\,\ar[r]^{[g]}&\,C}$ in $\Adel(\sA)$.
\item[(b)] A morphism $\xymatrix@1{ A \,\ar[r]^{[f]} & \,B}$ in $\Adel(\sA)$ is equal to $0$ if and only if there exist morphisms

 $s\colon A_1 \to B_0$ and $t\colon A_2 \to B_1$ in $\sA$ with $sb_0+a_1t=f_1$.
\begin{align*}
\xymatrix{ A_0 \ar[r]^{a_0} \ar[d]_{f_0}& A_1 \ar[dl]_{s} \ar[r]^{a_1} \ar[d]_{f_1} &  A_2\ar[dl]_{t}  \ar[d]^{f_2} \\
B_0 \ar[r]_{b_0} & B_1 \ar[r]_{b_1} &  B_2}
\end{align*}
Consequently, $[f]=[h]$ is true for $\xymatrix@1{A \,\ar@<0.4ex>[r]^{[f]}\ar@<-0.4ex>[r]_{[h]}&\,B}$ in $\Adel(\sA)$ if and only if there exist morphisms $s\colon A_1 \to B_0$ and $t\colon A_2 \to B_1$ in $\sA$ with $sb_0+a_1t=f_1-h_1$.
\item[(c)] An object $A \in \Ob(\Adel(\sA))$ is a zero object if and only if $1_A$ is equal to $0_A$. Consequently, $A \in \Ob(\Adel(\sA))$ is a zero object if and only if there exists morphism $s \colon A_1 \to A_0$ and $t \colon A_2 \to A_1$ with $s a_0 + a_1 t = 1$.
\begin{align*}
\xymatrix{
A_0 \ar@<0.4ex>[r]^{a_0} & A_1 \ar@<0.4ex>[l]^s  \ar@<0.4ex>[r]^{a_1} & A_2 \ar@<0.4ex>[l]^t  
}
\end{align*}
\end{itemize}
\end{rem}

\begin{lemdefn}[Inclusion functor] \label{defn:inclusionfunctor}
The functor $\tilde \I_{\sA} \colon \sA \to \sA^{\Delta_2}$ shall be defined by\\
 $\tilde \I_{\sA}(A): =\big(\xymatrix@1{0\,\ar[r]^0&\,A\,\ar[r]^0&\,0}\big)$ for $A \in \Ob\sA$ and $\tilde \I_{\sA}(f) := (0,f,0) \in \Hom_{\sA^{\Delta_2}}(\tilde\I_{\sA}(A),\tilde\I_{\sA}(B))$ for $\xymatrix@1{A\, \ar[r]^f&\,B}$ in $\sA$. The functor $\tilde \I_{\sA}$ is additive.

Let $\I_{\sA} := \R_{\sA} \circ \tilde \I_{\sA}$. We call $\I_{\sA}$ the \emph{inclusion functor} of $\sA$.

The inclusion functor $\I_{\sA}$ is a full and faithful additive functor.

Let $\I_{\sA}(\sA)$ be the full subcategory of $\Adel(\sA)$ defined by $\Ob(\I_{\sA}(\sA)):=\{\I_{\sA}(A) \colon A \in \Ob\sA\}$.
\end{lemdefn}

\begin{proof}
We have $0_{0,A} \cdot f = 0_{0,B}= 0_{0,0} \cdot 0_{0,B}$, $0_{A,0} \cdot 0_{0,0} = 0_{A,0}= f \cdot 0_{B,0}$, 
\begin{align*}
\tilde \I_\sA(fg)=(0,fg,0) =(0,f,0)(0,g,0) = \tilde\I_\sA(f)\tilde\I_\sA(g),
\end{align*}
\begin{align*}
\tilde \I_\sA(1_A)=(0,1_A,0)=(1_0,1_A,1_0)=1_{\tilde \I_\sA(A)}\end{align*} 
and 
\begin{align*}
\tilde \I_\sA(f+h) = (0,f+h,0)=(0,f,0)+(0,h,0)=\tilde \I_\sA(f)+ \tilde \I_\sA(h)
\end{align*}
 for $\xymatrix@1{A\,\ar@<0.4ex>[r]^{f}\ar@<-0.4ex>[r]_{h}&\,B\,\ar[r]^g&\,C}$ in $\sA$, therefore $\tilde \I_\sA$ is a well-defined additive functor.
 
We conclude that $\I_{\sA}$ is an additive functor as well.

Now suppose given $A,B \in \Ob \sA$. Any morphism from $\I_\sA(A)$ to $\I_\sA(B)$ in $\Adel(\sA)$ is necessarily of the form $[0,f,0]=\I_\sA(f)$ with $f \in \Hom_\sA(A,B)$, which shows that $\I_\sA$ is full.

Suppose $\I_\sA(f)=\I_\sA(g)$ for $f,g \in \Hom_\sA(A,B)$. We necessarily have $0_{A,0}$ and $0_{0,B}$ with\linebreak $0_{A,0}\cdot 0_{0,B}+0_{A,0}\cdot 0_{0,B}=f-g$, cf. notation \ref{nota:adel}~(b). We conclude that $f=g$ holds, so $\I_\sA$ is faithful. 
\end{proof}

\begin{rem} \label{rem:adelideal}
Let 
\begin{align*}
\oS_\sA := &\Big\{\big( \xymatrix{X \ar[r]^1 & X\ar[r]^s&Y  }\big) \in \Ob(\sA^{\Delta_2})  \colon \big(\xymatrix{X\ar[r]^s&Y}\big) \in \Mor \sA \Big\}\\
& \cup \Big\{\big( \xymatrix{X \ar[r]^s & Y\ar[r]^1&Y  }\big) \in \Ob(\sA^{\Delta_2}) \colon \big(\xymatrix{X\ar[r]^s&Y}\big) \in \Mor \sA \Big\}.
\end{align*}
Recall that $\langle \oS_\sA \rangle$ denotes the full additive subcategory of $\sA^{\Delta_2}$ generated by $\oS_\sA$, cf convention~\ref{conv:genfulladdsubcategory}. 

The equation
\begin{align*}\F_{\langle  \oS_\sA \rangle} = \J_\sA
\end{align*}
holds and therefore
\begin{align*}
\Adel(\sA) = \sA^{\Delta_2}/\langle \oS_\sA \rangle
\end{align*}
is true, cf. remark \ref{rem:factorfullsubcategory}.
\end{rem}

\begin{proof}
Suppose given $\big(\xymatrix@1{X\,\ar[r]^s&\,Y}\big) \in \Mor \sA$. Then $\big(\xymatrix@1{X\, \ar[r]^1 &\, X\, \ar[r]^s & \,Y}\big)$ is a zero object in $\Adel(\sA)$ since we have $1 \colon X \to X$ and $0 \colon Y \to X$ with $1 \cdot 1 + s \cdot 0 = 1$, cf. remark \ref{rem:adelbasic}~(c). Similarly, we have $0 \colon Y \to X$ and $1 \colon Y \to Y$ with $0 \cdot s + 1 \cdot 1 = 1$, so $\big(\xymatrix@1{X \,\ar[r]^s&\,Y\, \ar[r]^1&\,Y} \big)$ is also a zero object in $\Adel(\sA)$.

Therefore objects in $\oS_\sA$ are zero objects in $\Adel(\sA)$ and since direct sums of zero objects are again zero objects, all objects in $\langle \oS_\sA \rangle$ are zero objects in $\Adel(\sA)$.

Suppose given $\big(\xymatrix@1{A\,\ar[r]^f&\,B}\big) \in \F_{\langle \oS_\sA \rangle}$. \\ There exists a factorisation $\big(\xymatrix@1{A \,\ar[r]^{f} & \,B}\big) = \big(\xymatrix@1{A \,\ar[r]^{g} &\,N\, \ar[r]^{h} & \,B}\big)$ in $\sA^{\Delta_2}$ with $N \in \Ob(\langle \oS_\sA \rangle)$. In $\Adel(\sA)$ we have $[f] = [g] [h]$, so $[f]$ factors through a zero object in $\Adel(\sA)$. We conclude that $[f]=0$ holds, so $f \in \J_\sA$ is true.

\smallskip

Conversely, suppose given $\big(\xymatrix@1{A\,\ar[r]^f&\,B}\big) \in \J_\sA$. We have a diagram
\begin{align*}
\xymatrix{ A_0 \ar[r]^{a_0} \ar[d]_{f_0}& A_1 \ar[dl]_{s} \ar[r]^{a_1} \ar[d]_{f_1} &  A_2\ar[dl]_{t}  \ar[d]^{f_2} \\
B_0 \ar[r]_{b_0} & B_1 \ar[r]_{b_1} &  B_2}
\end{align*}
in $\sA$ with $sb_0+a_1t=f_1$. \\
Consider the objects $N:=\big(\xymatrix@1{A_0 \,\ar[r]^{a_0 a_1} &\, A_2 \,\ar[r]^1 & \,A_2}\big)$ and $S:=\big(\xymatrix@1{A_1\, \ar[r]^1 &\, A_1\, \ar[r]^{a_1} &\, A_2}\big)$ in $\oS_\sA$.

We get a factorisation of $f$ through $N \oplus S \in \Ob(\langle \oS_\sA \rangle)$ in $\sA^{\Delta_2}$ as follows.
\begin{align*}
\xymatrix{A_0\ar[d]_{\sm{1 &a_0}} \ar[rr]^{a_0} && A_1\ar[d]^{\sm{a_1 & 1}} \ar[rr]^{a_1} &&  A_2\ar[d]^{\sm{1&1}}
 \\ 
 A_0 \oplus A_1 \ar[rr]^{\sm{a_0 a_1&0\\0&1}} \ar[d]_{\sm{f_0-a_0 s \\ s }}&& A_2 \oplus A_1 \ar[d]^{\sm{t\\s b_0}} \ar[rr]^{\sm{1&0\\0&a_1}} &&  A_2 \oplus A_2 \ar[d]^{\sm{t b_1 \\ f_2 - t b_1}}
 \\
  B_0 \ar[rr]^{b_0} && B_1 \ar[rr]^{b_1} &&  B_2}
\end{align*}
This is a well-defined factorisation, since the following equations hold.
\begin{align*}
&\sm{1 &a_0} \sm{a_0 a_1&0\\0&1} = \sm{a_0 a_1 & a_0}  =  a_0 \sm{a_1 & 1}
\\
& \sm{a_1 & 1} \sm{1&0\\0&a_1} = \sm{a_1&a_1}= a_1 \sm{1&1} 
\\
&\sm{f_0-a_0 s \\ s } b_0= \sm{f_0b_0 - a_0s b_0\\s b_0}=\sm{a_0 f_1 - a_0sb_0\\ sb_0} =\sm{a_0(sb_0+a_1t) - a_0 s b_0 \\ sb_0} =\sm{a_0a_1t\\s b_0} = \sm{a_0 a_1&0\\0&1}\sm{t\\s b_0} 
\\
&\sm{t\\s b_0}  b_1 = \sm{tb_1 \\ s b_0 b_1} = \sm{t b_1 \\(sb_0 + a_1 t) b_1 - a_1 t b_1}  = \sm{t b_1 \\f_1 b_1 - a_1 t b_1} = \sm{t b_1 \\a_1f_2 - a_1 t b_1}= \sm{1&0\\0&a_1}\sm{t b_1 \\ f_2 - t b_1}
\\
&\sm{1 & a_0} \sm{f_0 - a_0s \\ s} = f_0 - a_0 s + a_0 s = f_0
\\
&\sm{a_1 & 1}\sm{t\\sb_0} = a_1 t + s b_0 = f_1 
\\
&\sm{1 & 1}\sm{tb_1\\f_2 - t b_1} = t b_1 + f_2 - t b_1 = f_2
\end{align*}
We conclude that $f \in \F_{\langle \oS_\sA \rangle}$ is true.
\end{proof}

\begin{lem}\label{lem:iso} Suppose given $A \in \Ob(\Adel(\sA))$ and isomorphisms $\varphi_i \colon A_i \to B_i$  in $\sA$ for $i \in [0,2]$. Let $B:=\big(\xymatrix@1{B_0\, \ar[rr]^{\varphi_0^{-1} a_0 \varphi_1}&&\,B_1\, \ar[rr]^{\varphi_1^{-1}a_1\varphi_2}&&\,B_2}\big) \in \Ob(\Adel(\sA))$. Then $[\varphi_0,\varphi_1,\varphi_2]\colon A \to B$ is an isomorphism in $\Adel(\sA)$ with inverse $\left[\varphi_0^{-1},\varphi_1^{-1},\varphi_2^{-1}\right]$.
\end{lem}

\begin{proof}
We have $\varphi_0 (\varphi_0^{-1} a_0 \varphi_1 )= a_0 \varphi_1$ and $\varphi_1 (\varphi_1^{-1}a_1\varphi_2) = a_1 \varphi_2$.
Therefore $[\varphi_0,\varphi_1,\varphi_2]$ and, consequently, $[\varphi_0^{-1},\varphi_1^{-1},\varphi_2^{-1}]$ are in fact morphisms in $\Adel(\sA)$, mutually inverse. 
\end{proof}

\begin{thrmdefn} \label{defn:duality} Recall that $\sA$ is an additive category.
We have an isomorphism of categories $\D_\sA \colon \Adel(\sA)^{\op} \to \Adel(\sA^{\op})$ with $\D_\sA (A)=\Big( \xymatrix@1{A_2\, \ar[r]^{a_1^{\op}} & \,A_1\,\ar[r]^{a_0^{\op}}&\,A_0}\Big)$
and $\D_\sA([f]^{\op})=\big[f_2^{\op},f_1^{\op},f_0^{\op}\big]$ for $\xymatrix@1{A\,\ar[r]^{[f]}&\,B}$ in $\Adel(\sA)$.

Note that $\D_\sA$ is additive, as follows from $\D_\sA$ being an isomorphism between additive categories or from the construction given in the proof.
\end{thrmdefn}

\begin{proof}
The functor $D \colon \sA^{\Delta_2} \to \Adel(\sA^{\op})^{\op}$ shall be defined by $D(A)=\Big( \xymatrix@1{A_2 \,\ar[r]^{a_1^{\op}} &\, A_1\,\ar[r]^{a_0^{\op}}&\,A_0}\Big)$ and $D(f)=\big[f_2^{\op},f_1^{\op},f_0^{\op}\big]^{\op}$ for $\xymatrix@1{A\,\ar[r]^{f}&\,B}$ in $\sA^{\Delta_2}$.

This is a well-defined additive functor because we have $b_1^{\op} f_1^{\op} = (f_1 b_1)^{\op} = (a_1 f_2)^{\op} = f_2^{\op} a_1^{\op}$, $b_0^{\op} f_0^{\op} = (f_0 b_0)^{\op} =(a_0 f_1)^{\op} = f_1^{\op} a_0^{\op}$,
\begin{align*}
D(fg)&=\big[(f_2g_2)^{\op},(f_1g_1)^{\op},(f_0g_0)^{\op}\big]^{\op} \\
&=\big[g_2^{\op}f_2^{\op},g_1^{\op}f_1^{\op},g_0^{\op}f_0^{\op}\big]^{\op} \\
&= \Big(\big[g_2^{\op},g_1^{\op},g_0^{\op}\big] \big[f_2^{\op},f_1^{\op},f_0^{\op}\big] \Big)^{\op}\\
&= \big[f_2^{\op},f_1^{\op},f_0^{\op}\big]^{\op}\big[g_2^{\op},g_1^{\op},g_0^{\op}\big]^{\op} \\
&= D(f)D(g),
\end{align*}
$D(1_A)= [1^{\op},1^{\op},1^{\op}]^{\op}= 1_{D(A)}$ and
\begin{align*}
D(f+h)&=\big[(f_2\mathord+h_2)^{\op},(f_1\mathord+h_1)^{\op},(f_0\mathord+h_0)^{\op}\big]^{\op}\\
&= \big[f_2^{\op}\mathord+h_2^{\op},f_1^{\op}\mathord+h_1^{\op},f_0^{\op}\mathord+h_0^{\op}\big]^{\op}\\
&=\Big(\big[f_2^{\op},f_1^{\op},f_0^{\op}\big]+\big[h_2^{\op},h_1^{\op},h_0^{\op}\big]\Big)^{\op}\\
&=\big[f_2^{\op},f_1^{\op},f_0^{\op}\big]^{\op}+\big[h_2^{\op},h_1^{\op},h_0^{\op}\big]^{\op}\\
&=D(f)+D(h)
\end{align*}
for $\xymatrix@1{A \,\ar@<0.4ex>[r]^{f}\ar@<-0.4ex>[r]_{h}&\,B\,\ar[r]^{g}&\,C}$ in $\sA^{\Delta_2}$.

Suppose given $\xymatrix@1{A\, \ar[r]^f &\,B}$ in $\sA^{\Delta_2}$ such that $f$ is null-homotopic with morphisms $s \colon A_1 \to B_0$ and $t \colon A_2 \to B_1$ satisfying $sb_0+a_1t=f_1$. \\ Using $t^{\op}\colon B_1 \to A_2$ and $s^{\op} \colon B_0 \to A_1$, we obtain $t^{\op}a_1^{\op} + b_0^{\op}s^{\op} = (a_1t+sb_0)^{\op} = f_1^{\op}$. Therefore $\big[f_2^{\op},f_1^{\op},f_0^{\op}\big]=0$ and $D(f)=\big[f_2^{\op},f_1^{\op},f_0^{\op}\big]^{\op}=0$ hold.

\hyperref[thrm:upf]{Theorem} \ref{thrm:upf} gives the additive functor $\hat D\colon \Adel(\sA) \to \Adel(\sA^{\op})^{\op}$ with $\hat D \circ \R_{\sA} = D$. We set 
\begin{align*}
\D_\sA := \big(\hat D\big)^{\op}\colon \Adel(\sA)^{\op} \to \Adel(\sA^{\op}).
\end{align*}

It is now sufficient to show that $\D_{\sA}$ and $(\D_{\sA^{\op}})^{\op}$ are mutually inverse.

Suppose given $\xymatrix@1{A \,\ar[r]^{[f]} & \,B}$ in $\Adel(\sA)$. We have
\begin{align*}
((\D_{\sA^{\op}})^{\op} \circ \D_\sA) (A) = (\D_{\sA^{\op}})^{\op} \Big(\xymatrix{A_2 \ar[r]^{a_1^{\op}} & A_1 \ar[r]^{a_0^{\op}} & A_0}\Big) = \big( \xymatrix{A_0 \ar[r]^{a_0} & A_1 \ar[r]^{a_1} & A_2} \big) = A
\end{align*}
and
\begin{align*}
((\D_{\sA^{\op}})^{\op} \circ \D_\sA) ([f]^{\op}) = (\D_{\sA^{\op}})^{\op}([f_2^{\op},f_1^{\op},f_0^{\op}]) = [f_0,f_1,f_2]^{\op}=[f]^{\op}.
\end{align*}
Suppose given 
\begin{align*}
\xymatrix{
A_2 \ar[d]_{f_2^{\op}} \ar[r]^{a_1^{\op}} & A_1\ar[d]^{f_1^{\op}} \ar[r]^{a_0^{\op}} & A_0 \ar[d]^{f_0^{\op}}\\
B_2\ar[r]^{b_1^{\op}} & B_1 \ar[r]^{b_0^{\op}} & B_0 
}
\end{align*}
in $(\sA^{\op})^{\Delta_2}$.

We have
\begin{align*}
(  \D_\sA \circ   (\D_{\sA^{\op}})^{\op}) \Big(\xymatrix{ A_2 \ar[r]^{a_1^{\op}} & A_1 \ar[r]^{a_0^{\op}} & A_0} \Big) &=  \D_\sA \big(\xymatrix{A_0 \ar[r]^{a_0} & A_1 \ar[r]^{a_1} & A_2}\big) \\&= \Big(\xymatrix{ A_2 \ar[r]^{a_1^{\op}} & A_1 \ar[r]^{a_0^{\op}} & A_0} \Big)
\end{align*}
and
\begin{align*}
(  \D_\sA \circ   (\D_{\sA^{\op}})^{\op})([f_2^{\op},f_1^{\op},f_0^{\op}]) = \D_\sA( [f_0,f_1,f_2]^{\op}) = [f_2^{\op},f_1^{\op},f_0^{\op}].
\end{align*}
\end{proof}


\section{Kernels and cokernels}
\label{sec:kernelsandcokernels}

\begin{thrmdefn} \label{thrm:kernelcokernel}
Recall that $\sA$ is an additive category.
Suppose given $\xymatrix@1{ A \,\ar[r]^{f} &\, B}$ in $\sA^{\Delta_2}$.
\begin{itemize}
\item[(a)] We set, using \hyperref[nota:functor]{notation} \ref{nota:functor},
\begin{align*}
\K(f):=\Bigg(\xymatrix@C=12mm{ A_0 \oplus B_0 \ar[r]^{\sm{a_0&0\\0&1}} & A_1 \oplus B_0 \ar[r]^{\sm{a_1&f_1\\0&-b_0}} &  A_2 \oplus B_1} \Bigg)\in \Ob(\Adel(\sA))=\Ob(\sA^{\Delta_2})
\end{align*} 
 and $\k(f):=\big(\sm{1\\0},\sm{1\\0},\sm{1\\0}\big) \in \Hom_{\sA^{\Delta_2}}(\K(f),A)$. 
 
The morphism $[\k(f)] \in \Hom_{\Adel(\sA)}(\K(f),A)$ is a kernel of $[f]\in \Hom_{\Adel(\sA)}(A,B)$.

\item[(b)] We set
\begin{align*}
\C(f):=\Bigg(\xymatrix@C=12mm{ B_0 \oplus A_1 \ar[r]^{\sm{b_0&0\\f_1&-a_1}} & B_1 \oplus A_2 \ar[r]^{\sm{b_1&0\\0&1}} &  B_2 \oplus A_2} \Bigg)\in \Ob(\Adel(\sA))= \Ob(\sA^{\Delta_2})
\end{align*} 
 and $\c(f):=(\sm{1&0},\sm{1&0},\sm{1&0}) \in \Hom_{\sA^{\Delta_2}}(B,\C(f))$.
 
 The morphism $[\c(f)]\in \Hom_{\Adel(\sA)}(B,\C(f))$ is a cokernel of $[f]\in \Hom_{\Adel(\sA)}(A,B)$.
\end{itemize}

\begin{align*}
\xymatrix{
 A_0 \oplus B_0 \ar[d]_{\sm{1\\0}} \ar[rr]^{\sm{a_0&0\\0&1}} && A_1 \oplus B_0 \ar[d]_{\sm{1\\0}} \ar[rr]^{\sm{a_1&f_1\\0&-b_0}} &&  A_2 \oplus B_1 \ar[d]_{\sm{1\\0}} \\
 A_0 \ar[d]_{f_0} \ar[rr]^{a_0} && A_1 \ar[d]_{f_1} \ar[rr]^{a_1} && A_2 \ar[d]_{f_2} \\
 B_0 \ar[d]_{\sm{1&0}} \ar[rr]^{b_0} && B_1 \ar[d]_{\sm{1&0}} \ar[rr]^{b_1} && B_2 \ar[d]_{\sm{1&0}} \\
 B_0 \oplus A_1 \ar[rr]_{\sm{b_0&0\\f_1&-a_1}} && B_1 \oplus A_2 \ar[rr]_{\sm{b_1&0\\0&1}} &&  B_2 \oplus A_2
}
\end{align*}

\end{thrmdefn}

\begin{proof} Ad (a).
We have $\k(f) \in \Mor(\sA^{\Delta_2})$, since the equations
 $\sm{1\\0}a_0 = \sm{a_0\\0}=\sm{a_0&0\\0&1} \sm{1\\0}$ and $\sm{1\\0}a_1 = \sm{a_1\\0}=\sm{a_1&f_1\\0&-b_0}\sm{1\\0}$ hold.

Next, we show that $[\k(f)][f]$ is equal to $0$.

Using $\sm{0\\1} \colon A_1 \oplus B_0 \to B_0$ and $\sm{0\\1} \colon A_2 \oplus B_1 \to B_1$ in $\sA$, we obtain
\begin{align*}
\sm{0\\1} b_0 + \sm{a_1&f_1\\0&-b_0}\sm{0\\1} = \sm{0\\b_0} + \sm{f_1\\-b_0} = \sm{f_1\\0} = \sm{1\\0}f_1=\left(\big{(}\sm{1\\0},\sm{1\\0},\sm{1\\0}\big{)}f\right)_1=(\k(f)f)_1
\end{align*}
and therefore $[\k(f)][f]=[\k(f)f]=0$.
\begin{align*}
\xymatrix{
 A_0 \oplus B_0 \ar[d]_{\sm{1\\0}} \ar[r]^{\sm{a_0&0\\0&1}} 
& A_1 \oplus B_0 \ar[ddl]|\hole^(.65){\sm{0\\1}} \ar[d]^{\sm{1\\0}} \ar[r]^{\sm{a_1&f_1\\0&-b_0}} 
&  A_2 \oplus B_1 \ar[ddl]|\hole^(.65){\sm{0\\1}} \ar[d]^{\sm{1\\0}}
\\
A_0 \ar[d]_{f_0} \ar[r]^(.3){a_0} & A_1\ar[d]_{f_1}  \ar[r]^(.3){a_1} & A_2\ar[d]^{f_2} 
\\
B_0 \ar[r]_{b_0} & B_1 \ar[r]_{b_1} & B_2
}
\end{align*}

\smallskip

We show that $[\k(f)]$ is a monomorphism:

Suppose given $[g]=\left[ \sm{g_0^0&g_0^1},\sm{g_1^0&g_1^1},\sm{g_2^0&g_2^1}\right] \colon C \to \K(f)$ in $\Adel(\sA)$ with $[g][\k(f)]=0$. There exist morphisms $s \colon C_1 \to A_0$ and $t \colon C_2 \to A_1$ in $\sA$ satisfying $sa_0+c_1t=(g\k(f))_1=\sm{g_1^0&g_1^1}\sm{1\\0}=g_1^0$.

Using $\sm{s &g_1^1} \colon C_1 \to A_0 \oplus B_0$ and $\sm{t&0} \colon C_2 \to A_1 \oplus B_0$ in $\sA$, we obtain
\begin{align*}
\sm{s & g_1^1}\sm{a_0&0\\0&1} + c_1 \sm{t&0} = \sm{sa_0&g_1^1}+\sm{c_1t&0}= \sm{sa_0+c_1t&g_1^1} = \sm{g_1^0&g_1^1}.
\end{align*}
 We conclude that $[g]=0$.
\begin{align*}
\xymatrix{
C_0 \ar[d]_{\sm{g_0^0&g_0^1}}\ar[r]^{c_0} 
&  C_1 \ar@{-}{+(-11,-13)}_s
\save[]+(-16.94,-20.02) \ar{+(-5.06,-5.98)}  \restore
\ar[d]^{\sm{g_1^0&g_1^1}} \ar[r]^{c_1} 
& C_2  \ar@{-}{+(-11,-13)}_t
\save[]+(-18.92,-22.36) \ar{+(-3.08,-3.64)}  \restore
\ar[d]^{\sm{g_2^0&g_2^1}}
\\
 A_0 \oplus B_0  \ar[d]_{\sm{1\\0}} \ar[r]_{\sm{a_0&0\\0&1}} 
& A_1 \oplus B_0  \ar[d]_{\sm{1\\0}} \ar[r]_*+[d]{\sm{a_1&f_1\\0&-b_0}} 
&  A_2 \oplus B_1  \ar[d]^{\sm{1\\0}}
\\
A_0 \ar[r]^{a_0}& A_1 \ar[r]^{a_1} & A_2
}
\end{align*}

\smallskip

The factorisation property is seen as follows.

Suppose given $[g] \colon C \to A$ in $\Adel(\sA)$ with $[g][f]=0$. There exist morphisms $s \colon C_1 \to B_0$ and $t \colon C_2 \to B_1$ in $\sA$ satisfying $sb_0+c_1t=(gf)_1=g_1f_1$.
\begin{align*}
\xymatrix{
 C_0 \ar[d]_{g_0} \ar[r]^{c_0} 
& C_1 \ar[ddl]|\hole^(.65){s} \ar[d]^{g_1} \ar[r]^{c_1} 
&  C_2 \ar[ddl]|\hole^(.65){t} \ar[d]^{g_2}
\\
A_0 \ar[d]_{f_0} \ar[r]^(.3){a_0} & A_1\ar[d]_{f_1}  \ar[r]^(.3){a_1} & A_2\ar[d]^{f_2} 
\\
B_0 \ar[r]_{b_0} & B_1 \ar[r]_{b_1} & B_2
}
\end{align*}
Now we set $\left[ \sm{g_0 & c_0s}, \sm{g_1&s}, \sm{g_2 &t}\right] \colon C \to \K(f)$. We get in fact a well-defined morphism in $\Adel(\sA)$, since the equations
\begin{align*}
\sm{g_0&c_0s}\sm{a_0&0\\0&1}=\sm{g_0a_0&c_0s}=\sm{c_0g_1&c_0s}= c_0 \sm{g_1&s}
\end{align*} and 
\begin{align*}
\sm{g_1&s}\sm{a_1&f_1\\0&-b_0}=\sm{g_1a_1&g_1f_1-sb_0} =\sm{c_1g_2&c_1t}  =c_1\sm{g_2&t}
\end{align*}
hold.

Finally, we have
\begin{align*}
\left[ \sm{g_0 & c_0s}, \sm{g_1&s}, \sm{g_2 &t}\right] \cdot [\k(f)] = \left[ \sm{g_0 & c_0s}\sm{1\\0}, \sm{g_1&s}\sm{1\\0}, \sm{g_2 &t}\sm{1\\0}\right] =[g_0,g_1,g_2]=[g].
\end{align*}
\begin{align*}
\xymatrix{
C_0 \ar[d]_{\sm{g_0 & c_0s}} \ar[r]^{c_0} & C_1 \ar[d]_{\sm{g_1&s}} \ar[r]^{c_1} & C_2 \ar[d]_{\sm{g_2 & t}}
\\
  A_0 \oplus B_0  \ar[d]_{\sm{1\\0}} \ar[r]_{\sm{a_0&0\\0&1}} 
& A_1 \oplus B_0  \ar[d]_{\sm{1\\0}} \ar[r]_*+[d]{\sm{a_1&f_1\\0&-b_0}} 
&  A_2 \oplus B_1  \ar[d]^{\sm{1\\0}}
\\
A_0 \ar[r]^{a_0}&A_1\ar[r]^{a_1}&A_2}
\end{align*}
The uniqueness of the induced morphism from $C$ to $\K(f)$ follows from $[\k(f)]$ being a monomorphism.

Ad (b). A  kernel of $\D_{\sA}([f]^{\op})$ is given by 
\begin{align*}
&\left[ \sm{1\\0},\sm{1\\0},\sm{1\\0}\right] = \left[\sm{1&0}^{\op},\sm{1&0}^{\op},\sm{1&0}^{\op}\right] \colon \\& \hspace{5cm} \Bigg(\xymatrix@C=17mm{B_2 \oplus A_2 \ar[r]^{\sm{b_1^{\op} & 0 \\ 0 & 1}} & B_1 \oplus A_2 \ar[r]^{\sm{b_0^{\op}&f_1^{\op}\\0&-a_1^{\op}}} & B_0 \oplus A_1
} \Bigg) \to \D_{\sA}(B).
\end{align*}
Therefore $\left[\sm{1&0},\sm{1&0},\sm{1&0}\right] \colon B \to \C(f)$ is a cokernel of $[f]$. Cf. definition~\ref{defn:duality}.
\end{proof}

 

\begin{kor} \label{kor:monoepicrit}
Suppose given $\xymatrix@1{A \,\ar[r]^{[f]} &\,B}$ in $\Adel(\sA)$.
\begin{itemize}
\item[(a)] The morphism $[f]$ is a monomorphism if and only if there exist morphisms $s\colon A_1 \to A_0$, $t\colon B_0 \to A_0$, $u\colon A_2 \to A_1$ and $v\colon B_1 \to A_1$ in $\sA$ satisfying $s a_0+a_1 u + f_1 v=1$ and $t a_0 = b_0 v$.
\begin{align*}
\xymatrix{
A_0 \ar@<0.4ex>[r]^{a_0} \ar@<0.4ex>[d]^{f_0} & A_1 \ar@<0.4ex>[r]^{a_1} \ar@<0.4ex>[l]^s \ar@<0.4ex>[d]^{f_1}& A_2\ar[d]^{f_2}\ar@<0.4ex>[l]^u
\\
B_0 \ar@<0.4ex>[u]^t \ar[r]^{b_0}& B_1\ar@<0.4ex>[u]^v \ar[r]^{b_1} & B_2
}
\end{align*}
\item[(b)] The morphism $[f]$ is an epimorphism if and only if there exist morphisms $s \colon B_1 \to B_0$, $t\colon B_1 \to A_1$, $u\colon B_2 \to  B_1$ and $v\colon B_2 \to A_2$ in $\sA$ satisfying $s b_0 + t f_1 + b_1 u =1$ and $t a_1=b_1 v$.
\begin{align*}
\xymatrix{
A_0 \ar[r]^{a_0} \ar[d]^{f_0}& A_1 \ar[r]^{a_1}\ar@<0.4ex>[d]^{f_1} & A_2\ar@<0.4ex>[d]^{f_2}
\\
B_0  \ar@<0.4ex>[r]^{b_0}& B_1\ar@<0.4ex>[l]^s \ar@<0.4ex>[u]^t \ar@<0.4ex>[r]^{b_1} & B_2 \ar@<0.4ex>[l]^u \ar@<0.4ex>[u]^v
}
\end{align*}
\end{itemize}
\end{kor}

\begin{proof}
Ad (a).
The morphism $[f]$ is a monomorphism if and only if its kernel $[\k(f)]$ is zero. This is the case if and only if there exist morphisms $\sm{s\\t} \colon A_1 \oplus B_0 \to A_0$ and $\sm{u\\v} \colon A_2 \oplus B_1 \to A_1$ with $\sm{s\\t} a_0 + \sm{a_1 & f_1\\0&-b_0} \sm{u\\v} = \sm{1\\0}$, cf. remark \ref{rem:adelbasic}~(b).

Ad (b). The morphism $[f]$ is an epimorphism if and only if its cokernel $[\c(f)]$ is zero. This is the case if and only if there exist morphisms $\sm{s&t} \colon B_1 \to B_0 \oplus A_1$ and $\sm{u&v} \colon B_2 \to B_1 \oplus A_2$ with $\sm{s&t} \sm{b_0 & 0 \\f_1 &-a_1} + b_1 \sm{u&v} = \sm{1 & 0}$.
\end{proof}

\begin{exa}
Given $\xymatrix@1{A\, \ar[r]^f &\,B}$ in $\sA$, the morphism $\I_{\sA}(f)$ is monomorphic if and only if $f$ is a coretraction. Dually, $\I_{\sA}(f)$ is epimorphic if and only if $f$ is a retraction.
\end{exa}

\section{The Adelman category is abelian}
\label{sec:adelmanabelian}
\begin{thrm} \label{thrm:iso} Recall that $\sA$ is an additive category.
Suppose given $\xymatrix@1{ A \,\ar[r]^{[f]} &\, B}$ in $\Adel(\sA)$. We have an isomorphism
\begin{align*}
\I_f:=\left[\sm{f_0&\sm{0&a_0}\\\sm{0\\1}&\sm{0&1\\0&0}}   ,\sm{f_1&\sm{0&1}\\\sm{0\\1}&0}   ,\sm{f_2&\sm{0&1}\\\sm{0\\b_1}&\sm{0&0\\1&0}}   \right] \colon \C(\k(f)) \to \K(\c(f))
\end{align*}
with inverse
\begin{align*}
\J_f:=\left[\sm{0&\sm{0&1}\\0&\sm{0&-1\\1&0}}   ,\sm{0&\sm{0&1}\\\sm{0\\1}&\sm{0&-b_0\\-a_1&-f_1}}   ,\sm{0&0\\\sm{0\\1}&\sm{0&1\\-1&0}}  \right] \colon \K(\c(f)) \to \C(\k(f))
\end{align*}
in $\Adel(\sA)$. 
\begin{align*}
\xymatrix{
\K(f) \armfl{.58}[r]^-{[\k(f)]} & A \arefl{.45}[dr]_{[\c(\k(f))]} \ar[rrr]^{[f]} &&& B \arefl{.4}[r]^-*+<0.6mm,0.6mm>{\scriptstyle [\c(f)]} & \C(f) \\
&& \C(\k(f)) \ar@<0.7ex>[r]^{\I_f}_{\sim} & \K(\c(f))\arm[ur]_{[\k(\c(f))]} \ar@<0.7ex>[l]^{\J_f}
}
\end{align*}
Moreover, the equation $[\c(\k(f))]\I_f[\k(\c(f))]=[f]$ holds, so $\I_f$ is the induced morphism of this kernel-cokernel-factorisation of $[f]$, cf. convention \ref{conv:kernelcokernelfactorisation}.
\end{thrm}

\begin{proof}
By definition, we have 
\begin{align*}
\C(\k(f))
&=\xymatrix{ A_0\moplus(A_1\moplus B_0) \ar[rrr]^{\sm{a_0&0\\\sm{1\\0}&\ -\sm{a_1&f_1\\0&-b_0}}}&&& A_1\moplus(A_2\moplus B_1)\ar[r]^{\sm{a_1&0\\0&1}} & A_2\moplus(A_2\moplus B_1)
}
\end{align*}
and
\begin{align*}
\K(\c(f))
&=\xymatrix{B_0\moplus(B_0\moplus A_1)\ar[r]^{\sm{b_0&0\\0&1}}& B_1\moplus(B_0\moplus A_1)  \ar[rrr]^{\sm{b_1&\sm{1&0}\\0&\ -\sm{b_0&0\\f_1&-a_1}}} &&& B_2\moplus(B_1\moplus A_2)  .
}
\end{align*}

The morphism $\I_f$ is well-defined in $\Adel(\sA)$, since we have
\begin{align*}
\sm{f_0&\sm{0&a_0}\\\sm{0\\1}&\sm{0&1\\0&0}} \sm{b_0&0\\0&1} 
= \sm{f_0b_0&\sm{0&a_0}\\\sm{0\\b_0}&\sm{0&1\\0&0}}
= \sm{a_0f_1& \sm{0&a_0}\\\sm{f_1\\0}-\sm{f_1\\-b_0}&\sm{0&1\\0&0}}= \sm{a_0&0\\\sm{1\\0}&\ -\sm{a_1&f_1\\0&-b_0}} \sm{f_1&\sm{0&1}\\\sm{0\\1}&0} 
\end{align*}
and
\begin{align*}
\sm{f_1&\sm{0&1}\\\sm{0\\1}&0} \sm{b_1&\sm{1&0}\\0&\ -\sm{b_0&0\\f_1&-a_1}} 
=\sm{f_1b_1&\sm{f_1&0}+\sm{-f_1&a_1}\\\sm{0\\b_1}&\sm{0&0\\1&0}}
=\sm{a_1f_2&\sm{0&a_1}\\\sm{0\\b_1}&\sm{0&0\\1&0}}
= \sm{a_1&0\\0&1} \sm{f_2&\sm{0&1}\\\sm{0\\b_1}&\sm{0&0\\1&0}}.
\end{align*}

The morphism $\J_f$ is well-defined in $\Adel(\sA)$, since we have
\begin{align*}
\sm{0&\sm{0&1}\\0&\sm{0&-1\\1&0}}\sm{a_0&0\\\sm{1\\0}&\ -\sm{a_1&f_1\\0&-b_0}}
=\sm{0&\sm{0&b_0}\\\sm{0\\1}&\ -\sm{0&b_0\\a_1&f_1}}
=\sm{b_0&0\\0&1} \sm{0&\sm{0&1}\\\sm{0\\1}&\sm{0&-b_0\\-a_1&-f_1}} 
\end{align*}
and
\begin{align*}
\sm{0&\sm{0&1}\\\sm{0\\1}&\sm{0&-b_0\\-a_1&-f_1}}  \sm{a_1&0\\0&1}
=\sm{0&\sm{0&1}\\\sm{0\\a_1}&\sm{0&-b_0\\-a_1&-f_1}}
=\sm{b_1&\sm{1&0}\\0&\ -\sm{b_0&0\\f_1&-a_1}}\sm{0&0\\\sm{0\\1}&\sm{0&1\\-1&0}}.
\end{align*}

We have to show that the equations $\I_f \cdot \J_f = 1_{\C(\k(f))}$ and $\J_f \cdot \I_f = 1_{\K(\c(f))}$ hold.

\begin{itemize}
\item

We have to show that $\I_f \cdot \J_f - 1_{\C(\k(f))}$ is equal to $0$.

Using
$0 \colon A_1 \oplus (A_2 \oplus B_1) \to A_0 \oplus (A_1 \oplus B_0)$ and $\sm{0 & \sm{-1 & 0}\\0 & \sm{-1&0\\0&0}} \colon A_2 \oplus (A_2 \oplus B_1) \to A_1 \oplus (A_2 \oplus B_1)$ in $\sA$, we obtain
\begin{align*}
0 \cdot \sm{a_0 & 0 \\ \sm{1\\0} &\ -\sm{a_1&f_1\\0&-b_0}} + \sm{a_1&0\\0&1} \sm{0&\sm{-1&0}\\0&\sm{-1&0\\0&0}} &= \sm{0&\sm{-a_1&0}\\0&\sm{-1&0\\0&0}}
\\ &
=\sm{f_1&\sm{0&1}\\\sm{0\\1}&0}   \sm{0&\sm{0&1}\\\sm{0\\1}&\sm{0&-b_0\\-a_1&-f_1}}-\sm{1&0\\0&\sm{1&0\\0&1}}.
\end{align*}

This implies $\I_f \cdot \J_f -  1_{\C(\k(f))} = 0$.

\item

We have to show that $\J_f \cdot \I_f - 1_{\K(\c(f))}$ is equal to $0$.

Using $\sm{0&0\\\sm{-1\\0}&\sm{-1&0\\0&0}} \colon B_1 \oplus (B_0 \oplus A_1) \to B_0 \oplus (B_0 \oplus A_1)$ and \\$0 \colon B_2 \oplus (B_1 \oplus A_2) \to B_1 \oplus (B_0 \oplus A_1)$ in $\sA$, we obtain
\begin{align*}
\sm{0&0\\\sm{-1\\0}&\sm{-1&0\\0&0}} \sm{b_0 & 0 \\ 0&1} + \sm{b_1&\sm{1&0}\\0&\ -\sm{b_0&0\\f_1&-a_1}} \cdot 0
&=\sm{0&0\\\sm{-b_0\\0}&\sm{-1&0\\0&0}}
\\&
=\sm{0&\sm{0&1}\\\sm{0\\1}&\sm{0&-b_0\\-a_1&-f_1}} \sm{f_1&\sm{0&1}\\\sm{0\\1}&0}   -\sm{1&0\\0&\sm{1&0\\0&1}}.
\end{align*}

This implies $\J_f \cdot \I_f -  1_{\K(\c(f))} = 0$.
\end{itemize}

Finally, we verify $[\c(\k(f))]\I_f[\k(\c(f))]=[f]$:
\begin{align*}
&\mathrel {\phantom{=}} [\c(\k(f))]\I_f[\k(\c(f))]\\
 &= \left[\sm{1&0},\sm{1&0},\sm{1&0}\right] \left[\sm{f_0&\sm{0&a_0}\\\sm{0\\1}&\sm{0&1\\0&0}}   ,\sm{f_1&\sm{0&1}\\\sm{0\\1}&0}   ,\sm{f_2&\sm{0&1}\\\sm{0\\b_1}&\sm{0&0\\1&0}}   \right] \left[\sm{1\\0},\sm{1\\0},\sm{1\\0}\right]\\
&=\left[\sm{1&0},\sm{1&0},\sm{1&0}\right] \left[\sm{f_0\\\sm{0\\1}},\sm{f_1\\\sm{0\\1}},\sm{f_2\\\sm{0\\b_1}} \right]\\
&=[f_0,f_1,f_2]\\
&=[f]
\end{align*}

\begin{align*} \hspace{-0.65cm}
\xymatrix{
&A_0\moplus B_0 \ar[dd]_{\sm{1\\ 0}}\ar[rr]^{\sm{a_0&0\\0&1}}&&A_1\moplus B_0\ar[dd]_{\sm{1\\ 0}}\ar[rr]^{\sm{a_1&f_1\\0&-b_0}}&&A_2\moplus B_1\ar[dd]_{\sm{1\\ 0}}
\\
\\
&\save[]+(0,-26) \ar@{-}{+(0,-23)} \restore
\save[]+(0,-57) \ar[dddddd]^{f_0} \restore
A_0\ar@{-}{+(0,-15)} \ar[ldd]_{\sm{1&0}}\ar[rr]^{a_0} && \save[]+(0,-26) \ar@{-}{+(0,-6)} \restore
\save[]+(0,-43) \ar@{-}{+(0,-6)} \restore
\save[]+(0,-61) \ar[dddddd]^{f_1} \restore
A_1\ar@{-}{+(0,-18)} \ar[ldd]_{\sm{1&0}}\ar[rr]^{a_1} &&
 A_2\ar[ldd]_{\sm{1&0}} \ar[dddddd]^{f_2}
\\
\\
A_0\moplus(A_1\moplus B_0) \ar@<-0.3ex>[dd]_{\sm{f_0&\sm{0&a_0}\\\sm{0\\1}&\sm{0&1\\0&0}}}\ar[rr]^{\sm{a_0&0\\\sm{1\\0}&-\sm{a_1&f_1\\0&-b_0}}}&& A_1\moplus(A_2\moplus B_1)\ar@<-0.3ex>[dd]_{\sm{f_1&\sm{0&1}\\\sm{0\\1}&0}}\ar[rr]^{\sm{a_1&0\\0&1}} && A_2\moplus(A_2\moplus B_1)\ar@<-0.3ex>[dd]_{\sm{f_2&\sm{0&1}\\\sm{0\\b_1}&\sm{0&0\\1&0}}}
\\
\\
B_0\moplus(B_0\moplus A_1) \ar[ddr]_{\sm{1\\0}} \ar@<-0.3ex>[uu]_{\sm{0&\sm{0&1}\\0&\sm{0&-1\\1&0}}} \ar[rr]_{\sm{b_0&0\\0&1}}&& B_1\moplus(B_0\moplus A_1) \ar[ddr]_{\sm{1\\0}} \ar@<-0.3ex>[uu]_{\sm{0&\sm{0&1}\\\sm{0\\1}&\sm{0&-b_0\\-a_1&-f_1}}} \ar[rr]_{\sm{b_1&\sm{1&0}\\0&-\sm{b_0&0\\f_1&-a_1}}} && B_2\moplus(B_1\moplus A_2) \ar[ddr]_{\sm{1\\0}} \ar@<-0.3ex>[uu]_{\sm{0&0\\\sm{0\\1}&\sm{0&1\\-1&0}}} 
\\
\\
&B_0 \ar[dd]_{\sm{1&0}}\ar[rr]^{b_0}&&B_1\ar[dd]_{\sm{1&0}}\ar[rr]^{b_1}&&B_2 \ar[dd]_{\sm{1&0}}
\\
\\
&B_0\moplus A_1 \ar[rr]^{\sm{b_0&0\\f_1&-a_1}} && B_1 \moplus A_2 \ar[rr]^{\sm{b_1&0\\0&1}} && B_2 \moplus A_2
}
\end{align*}

\end{proof}

\begin{kor}[Kernel-cokernel-factorisation]
\mbox{}
\label{kor:kernelcokernelfactorisation}

We obtain a
kernel-cokernel-factorisation of $\xymatrix@1{A\,\ar[r]^{[f]}&\,B}$ in $\Adel(\sA)$ by taking residue classes of the $\sA^{\Delta_2}$-morphisms in the following diagram.
\begin{align*}
\xymatrix{
&A_0\moplus B_0 \ar[dd]_{\sm{1\\ 0}}\ar[rr]^{\sm{a_0&0\\0&1}}&&A_1\moplus B_0\ar[dd]_{\sm{1\\ 0}}\ar[rr]^{\sm{a_1&f_1\\0&-b_0}}&&A_2\moplus B_1\ar[dd]_{\sm{1\\ 0}}
\\
\\
&
A_0 \ar[dddddd]|(0.27){\phantom{\rule{0mm}{12mm}}}|(0.72){\phantom{\rule{0mm}{10mm}}}^{f_0} \ar[ldd]_{\sm{1&0 & 0}}\ar[rr]^{a_0} && 
A_1 \ar[dddddd]|(0.28){\phantom{\rule{0mm}{11mm}}}|(0.732){\phantom{\rule{0mm}{12mm}}}^{f_1} \ar[ldd]_{\sm{1&0&0}}\ar[rr]^{a_1} &&
 A_2\ar[ldd]_{\sm{1&0&0}} \ar[dddddd]^{f_2}
\\
\\
A_0\moplus A_1\moplus B_0 \ar@<-0.3ex>[dd]_{\sm{f_0&0&a_0\\0&0&1\\1&0&0}}\ar[rr]^{\sm{a_0&0&0\\1&-a_1&-f_1\\0&0&b_0}}&& A_1\moplus A_2\moplus B_1 \ar@<-0.3ex>[dd]_{\sm{f_1&0&1\\0&0&0\\1&0&0}}\ar[rr]^{\sm{a_1&0&0\\0&1&0\\0&0&1}} && A_2\moplus A_2\moplus B_1 \ar@<-0.3ex>[dd]_{\sm{f_2&0&1\\0&0&0\\b_1&1&0}}
\\
\\
B_0\moplus B_0\moplus A_1  \ar[ddr]_{\sm{1\\0\\0}} \ar@<-0.3ex>[uu]_{\sm{0&0&1\\0&0&-1\\0&1&0}} \ar[rr]_{\sm{b_0&0&0\\0&1&0\\0&0&1}}&& B_1\moplus B_0\moplus A_1 \ar[ddr]_{\sm{1\\0\\0}} \ar@<-0.3ex>[uu]_{\sm{0&0&1\\0&0&-b_0\\1&-a_1&-f_1}} \ar[rr]_{\sm{b_1&1&0\\0&-b_0&0\\0&-f_1&a_1}} && B_2\moplus B_1\moplus A_2 \ar[ddr]_{\sm{1\\0\\0}} \ar@<-0.3ex>[uu]_{\sm{0&0&0\\0&0&1\\1&-1&0}} 
\\
\\
&B_0 \ar[dd]_{\sm{1&0}}\ar[rr]^{b_0}&&B_1\ar[dd]_{\sm{1&0}}\ar[rr]^{b_1}&&B_2 \ar[dd]_{\sm{1&0}}
\\
\\
&B_0\moplus A_1 \ar[rr]^{\sm{b_0&0\\f_1&-a_1}} && B_1 \moplus A_2 \ar[rr]^{\sm{b_1&0\\0&1}} && B_2 \moplus A_2
}
\end{align*}
\end{kor}

\begin{proof}
We apply lemma \ref{lem:iso} to the factorisation obtained in the proof of theorem \ref{thrm:iso} by using isomorphisms of the form
\begin{align*}
\xymatrix{
X \oplus (Y \oplus Z) \ar[rrr]^{\sm{1 & 0 & 0\\\sm{0 \\0}&\sm{1\\0}&\sm{0\\1}}} &&& X \oplus Y \oplus Z
}
\end{align*}
in $\sA$.
\end{proof}

\begin{thrm} \label{thrm:adelabelian}
Recall that $\sA$ is an additive category. The Adelman category $\Adel(\sA)$ is abelian.
\end{thrm}

\begin{proof}
The category $\Adel(\sA)$ is additive, cf. lemma \ref{defn:adel}. Each morphism in $\Adel(\sA)$ has a kernel and a cokernel, cf. theorem \ref{thrm:kernelcokernel}. Given $[f] \in \Mor(\Adel(\sA))$, the induced morphism $\I_f$ of the kernel-cokernel-factorisation obtained in theorem \ref{thrm:iso} is an isomorphism. Therefore the induced morphism of each kernel-cokernel-factorisation of $[f]$ is an isomorphism, cf. remark \ref{rem:kernelcokernelfactorisation}. We conclude that $\Adel(\sA)$ is abelian.
\end{proof}

\begin{kor}
Suppose given $[f] \in \Mor(\Adel(\sA))$. The morphism $[f]$ is an isomorphism if and only if it is a monomorphism and an epimorphism. Therefore the two criteria in corollary~\ref{kor:monoepicrit} may be used to check whether $[f]$ is an isomorpism.
\end{kor}

\section{Projectives and injectives}
\label{sec:projectivesandinjectives}
\begin{defn}
Let $\sR(\sA)$ be the full subcategory defined by
\begin{align*}
\Ob(\sR(\sA)):=\{P \in \Ob(\Adel(\sA)) \colon P_0 = 0 \}.
\end{align*}
Let $\sL(\sA)$ be the full subcategory defined by \begin{align*}
\Ob(\sL(\sA)):=\{I \in \Ob(\Adel(\sA)) \colon I_2 = 0 \}.
\end{align*}
Consequently, objects in $\sR(\sA)$ are of the form $\big(\xymatrix@1{0 \,\ar[r]^-0 &\, P_1\, \ar[r]^{p_1} &\, P_2}\big)$ and objects in $\sL(\sA)$ are of the form $\big(\xymatrix@1{I_0\, \ar[r]^{i_0} &\, I_1\, \ar[r]^{0} & \,0}\big)$.
\end{defn}

\begin{prop}
\label{prop:projectivesinjectives} \mbox{}
\begin{itemize}
\item[(a)] The objects in $\sR(\sA)$ are projective in $\Adel(\sA)$.
\item[(b)] The objects in $\sL(\sA)$ are injective in $\Adel(\sA)$.
\end{itemize}
\end{prop}

\begin{proof}
Ad (a).
Suppose given the following diagram in $\Adel(\sA)$ with $P \in \sR(\sA)$ and $[f]$ epimorphic.
\begin{align*}
\xymatrix{ & P \ar[d]^{[g]} \\ A \are[r]_*+<0.8mm,0.8mm>{\scriptstyle [f]} & B}
\end{align*}
Since $[f]$ is an epimorphism, we have $[\c(f)]=0$, whence $
 [g\c(f)]=[g]  [\c(f)]=0$. Therefore there exist morphisms $\sm{s_0 & s_1} \colon P_1 \to B_0 \oplus A_1$ and $\sm{t_0 & t_1} \colon P_2 \to B_1 \oplus A_2$ satisfying \\
 $\sm{s_0 & s_1} \sm{b_0&0\\f_1 & -a_1}+p_1 \sm{t_0 & t_1} = \sm{g_1 & 0}$.
So the equations $s_0 b_0 + s_1 f_1 + p_1 t_0 = g_1$ and $s_1 a_1 = p_1 t_1$ hold.
\begin{align*}
\xymatrix{
0 \ar[d]_{g_0} \ar[r]^{0} 
& P_1 \ar[ddl]|\hole^(.78){\sm{s_0&s_1}} \ar[d]^{g_1} \ar[r]^{p_1} 
&  P_2 \ar[ddl]|\hole^(.65){\sm{t_0&t_1}} \ar[d]^{g_2}
\\
B_0 \ar[d]_{\sm{1 & 0}} \ar[r]^(.3){b_0} & B_1\ar[d]_{\sm{1&0}}  \ar[r]^(.3){b_1} & B_2\ar[d]^{\sm{1&0}} 
\\
B_0 \oplus A_1 \ar[r]_*+[d]{\sm{b_0&0\\f_1&-a_1}} & B_1 \oplus A_2 \ar[r]_{\sm{b_1&0\\0&1}} & B_2\oplus A_2
}
\end{align*}

Consider $[0,s_1,t_1]\colon P \to A$. This is a well-defined morphism in $\Adel(\sA)$, since we have $0 \cdot a_0 = 0 = 0 \cdot s_1 $ and $s_1 a_1 = p_1 t_1$.

Finally, we show that $[0,s_1,t_1] [f] = [g]$ holds.

Using $s_0 \colon P_1 \to B_0$ and $t_0 \colon P_2 \to B_1$, we obtain $s_0 b_0 + p_1 t_0 = g_1 - s_1 f_1$.

This implies $[g]-[0,s_1,t_1][f] = [g_0,g_1-s_1 f_1 ,  g_2- t_1 f_2] = 0$, so the following diagram commutes.
\begin{align*}
\xymatrix{ & P\ar[dl]_{[0,s_1,t_1]} \ar[d]^{[g]}\\ A \are[r]_*+<0.8mm,0.8mm>{\scriptstyle [f]} & B}
\end{align*}
We conclude that $P$ is projective.

Ad (b). This is dual to (a) using the isomorphism of categories $\D_{\sA} \colon \Adel(\sA)^{\op} \to \Adel(\sA^{\op})$.
\end{proof}

\begin{exa} \label{exa:inclusion} Given $A \in \Ob \sA$, we have $\I_\sA(A) \in \Ob(\sL(\sA)) \cap \Ob(\sR(\sA))$.
 Therefore $\I_\sA(A)$ is injective and projective, cf. proposition \ref{prop:projectivesinjectives}.
\end{exa}

\begin{lem}
\label{lem:Rkernel}
Suppose given $P \in \Ob(\sR(\sA))$. \\
There exists a left-exact sequence $\xymatrix@1{P\,\armfl{0.39}[r]^(.4)k &\,\I_\sA(A)\,\ar[r]^{\I_\sA(f)}&\,\I_\sA(B)}$ in $\Adel(\sA)$ with $A,B \in  \Ob \sA$ and $f\in \Hom_{\sA}(A,B)$.
\end{lem} 

\begin{proof}
A kernel of $\I_\sA(p_1)\colon \I_\sA(P_1) \to \I_\sA(P_2)$ is given by $\xymatrix@1@C=10mm{0 \oplus 0\, \ar[r]^-{\sm{0&0\\0&0}}&\,P_1 \oplus 0 \,\ar[r]^{\sm{0&p_1\\0&0}}&\,0\oplus P_2}$, cf. theorem \ref{thrm:kernelcokernel}. Lemma \ref{lem:iso} says that $P$ is also a kernel of $\I_\sA(p_1)$ by using the isomorphisms
$0\colon 0 \to 0\oplus 0$, $\sm{1&0}\colon P_1 \to P_1 \oplus 0$ and $\sm{0&1}\colon P_2 \to 0 \oplus P_2$.
\end{proof}

\begin{thrm} \label{thrm:adelenough}
Recall that $\sA$ is an additive category and that, by theorem \ref{thrm:adelabelian}, $\Adel(\sA)$ is abelian.
The Adelman category $\Adel(\sA)$ has enough projectives and injectives. More precisely, we have the following statements.
\begin{itemize}
\item[(a)] Suppose given $A \in \Ob(\Adel(\sA))$. We have an epimorphism \\
$[0,1,1] \colon \big(\xymatrix@1{0\,\ar[r]^{0}&\,A_1\,\ar[r]^{a_1}&\,A_2} \big) \to A$ in $\Adel(\sA)$ with $\big(\xymatrix@1{0\,\ar[r]^{0}&\,A_1\,\ar[r]^{a_1}&\,A_2} \big) \in \Ob(\sR(\sA))$.
\item[(b)] Suppose given $A \in \Ob(\Adel(\sA))$. We have a monomorphism \\
 $[1,1,0] \colon A \to \big(\xymatrix@1{A_0\,\ar[r]^{a_0}&\,A_1\,\ar[r]^0&\,0} \big)$ in $\Adel(\sA)$ with $\big(\xymatrix@1{A_0\,\ar[r]^{a_0}&\,A_1\,\ar[r]^0&\,0} \big) \in \Ob(\sL(\sA))$.
\end{itemize}
\end{thrm}

\begin{proof} Ad (a).
The morphism $[0,1,1] \colon \big(\xymatrix@1{0\,\ar[r]^{0}&\,A_1\,\ar[r]^{a_1}&\,A_2} \big) \to A$ is a well-defined morphism in $\Adel(\sA)$ since we have $0 \cdot a_0 =0= 0 \cdot 1$ and $1 \cdot a_1 = a_1 = a_1 \cdot 1$. Using $0 \colon A_1 \to A_0$, $1 \colon A_1 \to A_1$, $0 \colon A_2 \to A_1$ and $1 \colon A_2 \to A_2$, we obtain $0 \cdot a_0 + 1 \cdot 1 + a_1 \cdot 0 = 1$ and $1 \cdot a_1 = a_1 \cdot 1$. Corollary~\ref{kor:monoepicrit}~(b) says that $[0,1,1]$ is an epimorphism.
\begin{align*}
\xymatrix{
0 \ar[d]^0 \ar[r]^0 & A_1 \ar[r]^{a_1} \ar@<.4ex>[d]^1 & A_2 \ar@<.4ex>[d]^1 \\
A_0 \ar@<.4ex>[r]^{a_0} & A_1 \ar@<.4ex>[l]^0 \ar@<.4ex>[r]^{a_1} \ar@<.4ex>[u]^1 & A_2 \ar@<.4ex>[l]^0 \ar@<.4ex>[u]^1
}
\end{align*}
Ad (b). This is dual to (a) using the isomorphism of categories $\D_{\sA} \colon \Adel(\sA)^{\op} \to \Adel(\sA^{\op})$.
\end{proof}

\begin{lem} \label{lem:Adelcokernel}
Given $A \in \Ob(\Adel(\sA))$, there exists a right-exact sequence $\xymatrix@1{P \,\ar[r]^{[f]} &\, Q\, \are[r]^{[c]} &\,A}$ in $\Adel(\sA)$ with $P,Q \in \Ob(\sR(\sA))$ such that a kernel of $[f]$ is projective.
\end{lem}

\begin{proof}
We may choose an epimorphism $[c] \colon Q \to A$ with $Q \in \Ob(\sR(\sA))$, cf. theorem \ref{thrm:adelenough}.
Let $[k] \colon K \to Q$ be a kernel of $[c]$. Again, we may choose an epimorphism $[e] \colon P \to K$ with $P \in \Ob(\sR(\sA))$. Let $[f] := [e][k]$.
\begin{align*}
\xymatrix{
P \are[dr]_{[e]} \ar[rr]^{[f]} && Q \are[r]^*+<0.3mm,0.3mm>{\scriptstyle [c]} &A\\
& K \arm[ur]_{[k]}
}
\end{align*}
Remark \ref{rem:kernelcokernelabeliancategory} says that $[c]$ is a cokernel of $[k]$. Since $[e]$ is epimorphic, $[c]$ is also a cokernel of $[f]$.

A kernel of $[f]$ is given by
\begin{align*}
\K((0,f_1,f_2)) = \Big( \xymatrix@1@C=12mm{0 \oplus 0 \ar[r]^-0 & P_1 \oplus 0 \ar[r]^-{\sm{p_1 & f_1\\0&0}}&P_2 \oplus Q_1}\Big),
\end{align*}
cf. theorem \ref{thrm:kernelcokernel}. Now $\K((0,f_1,f_2))$ is isomorphic to $\big( \xymatrix@1@C=10mm{0\, \ar[r]^0 & \,P_1\, \ar[r]^-{\sm{p_1&f_1}} & \,P_2 \oplus Q_1} \big) \in \Ob(\sR(\sA))$ in $\Adel(\sA)$, cf. lemma \ref{lem:iso}, and therefore projective.
\end{proof}

\begin{kor}
The projective dimension of $\Adel(\sA)$ is at most two due to the previous lemma~\ref{lem:Adelcokernel}. Dually, the injective dimension of $\Adel(\sA)$ is also at most two since\linebreak $\D_{\sA} \colon \Adel(\sA)^{\op} \to \Adel(\sA^{\op})$ is an isomorphism of categories.
\end{kor}

\section{Additive functors and transformations}
\label{sec:additivefunctorsandtransformations}
\begin{thrmdefn} \label{defn:adelfunctortransformation}
Recall that $\sA$ is an additive category. \\
Suppose given an additive category $\sB$.
\begin{itemize}
	\item[(a)] Suppose given an additive functor $F \colon \sA \to \sB$. \\
By setting 
\begin{align*}
(\Adel(F))(X) := \Big( \xymatrix{F(X_0) \ar[r]^{F(x_0)} & F(X_1) \ar[r]^{F(x_1)} & F(X_2)} \Big)
\end{align*}
for $X \in \Ob(\Adel(\sA))$ and 
\begin{align*}
(\Adel(F))([f]) := [F(f_0), F(f_1),F(f_2)]
\end{align*}
for $[f] \in \Mor(\Adel(\sA))$, we obtain an exact functor $\Adel(F) \colon \Adel(\sA) \to \Adel(\sB)$.

The equation $\Adel(F)^{\op}=\D_{\sB}^{-1} \circ  \Adel(F^{\op}) \circ \D_{\sA}$ holds, cf. definition \ref{defn:duality}.

The transformation $\varepsilon^F \colon \I_{\sB} \circ F \to \Adel(F) \circ \I_{\sA}$ defined by
\begin{align*}
(\varepsilon^F)_S := [0,1,0] \colon \Big( \xymatrix{0 \ar[r]^-{0} & F(S) \ar[r]^-{0} & 0} \Big) \to \Big( \xymatrix{F(0) \ar[r]^-{0} & F(S) \ar[r]^-{0} & F(0)} \Big)
\end{align*}
for $S \in \Ob \sA$ is an isotransformation. \\
In particular, we have $\I_{\sB} \circ F = \Adel(F)  \circ \I_{\sA}$ if $F(0_{\sA})= 0_{\sB}$ holds.
\begin{align*}
\xymatrix@C=14mm{
\sA \ar[r]^F \ar[d]_{\I_{\sA}}& \sB \ar[d]^{\I_{\sB}}\\
\Adel(\sA) \ar[r]^{\Adel(F)} & \Adel(\sB)
}
\end{align*}
\item[(b)] Suppose given additive functors $F,G \colon \sA \to \sB$ and a transformation $\alpha \colon F \Rightarrow G$.\\
By setting
\begin{align*}
(\Adel(\alpha))_X := \left[\alpha_{X_0}, \alpha_{X_1}, \alpha_{X_2}\right] \colon (\Adel(F))(X) \to (\Adel(G))(X)
\end{align*}
for $X \in \Ob(\Adel(\sA))$, we obtain a transformation $\Adel(\alpha) \colon \Adel(F) \to \Adel(G)$.

If $\alpha$ is an isotransformation, then $\Adel(\alpha)$ is also an isotransformation.
\end{itemize}
Cf. remark \ref{rem:last}.
\end{thrmdefn}

\begin{proof}
Ad (a).

Let $\tilde F \colon \sA^{\Delta_2} \to \Adel(\sB)$ be defined by 
$\tilde F(X) :=  \Big( \xymatrix@1{F(X_0)\, \ar[r]^{F(x_0)} & \,F(X_1)\, \ar[r]^{F(x_1)} &\, F(X_2)} \Big)$ for \linebreak $X \in \Ob(\sA^{\Delta_2})$ and $\tilde F(f) := [F(f_0), F(f_1),F(f_2)]$ for $f \in \Mor(\sA^{\Delta_2})$. This is a well-defined additive functor because we have 
\begin{align*}
F(x_0) F(f_1) = F(x_0f_1) = F(f_0 y_0) F(f_0) F(y_0),
\end{align*}
\begin{align*}
F(x_1) F(f_2) = F(x_1 f_2) = F(f_1 y_1) = F(f_1)F(y_1),
\end{align*}
\begin{align*}
\tilde F(fg)& = [F(f_0g_0),F(f_1g_1),F(f_2g_2)] \\&= [F(f_0)F(g_0), F(f_1)F(g_1), F(f_2)F(g_2)] \\&=[F(f_0), F(f_1),F(f_2)][F(g_0), F(g_1),F(g_2)]
\\&=  \tilde F(f) \tilde F(g), 
\end{align*}
\begin{align*}
\tilde F(1_X) = [F(1_{X_0}),F(1_{X_1}),F(1_{X_2})] = [1_{F(X_0)},1_{F(X_1)},1_{F(X_2)}] = 1_{\tilde F(X)}
\end{align*}
and
\begin{align*}
F(f+h)&= [F(f_0+h_0),F(f_1+h_1),F(f_2+h_2)]\\ &=[F(f_0)+F(h_0),F(f_1)+F(h_1),F(f_2)+F(h_2)]\\& = [F(f_0), F(f_1),F(f_2)]+[F(h_0), F(h_1),F(h_2)]\\&
 =\tilde F(f) + \tilde F(h)
\end{align*}
for $\xymatrix@1{A \,\ar@<0.4ex>[r]^{f}\ar@<-0.4ex>[r]_{h}&\,B\,\ar[r]^{g}&\,C}$ in $\sA^{\Delta_2}$.

Suppose given $\xymatrix@1{X \,\ar[r]^f &\,Y}$ in $\sA^{\Delta_2}$ with $f$ null-homotopic. There exist morphisms $s \colon X_1 \to Y_0$ and $t \colon X_2 \to Y_1$ in $\sA$ with $sy_0 + x_1 t = f_1$. Using $F(s) \colon F(X_1) \to F(Y_0)$ and $F(t) \colon F(X_2) \to F(Y_1)$, we obtain $F(s)F(y_0) + F(x_1) F(t) = F(sy_0+x_1t) = F(f_1)$.

Theorem \ref{thrm:upf} now gives the additive functor $\Adel(F) \colon \Adel(\sA) \to \Adel(\sB)$ with $\Adel(F) \circ \R_{\sA} = \tilde F$.

We want to show that $\Adel(F)^{\op}=\D_{\sB}^{-1} \circ  \Adel(F^{\op}) \circ \D_{\sA}$ holds. \\
Suppose given $\xymatrix@1{X \,\ar[r]^{[f]} &\, Y}$ in $\Adel(\sA)$. We have
\begin{align*}
(\D_{\sB}^{-1} \circ  \Adel(F^{\op}) \circ \D_{\sA})(X)& = (\D_{\sB}^{-1} \circ  \Adel(F^{\op}))\Big( \Big(\xymatrix{X_2 \ar[r]^{x_1^{\op}} & X_1 \ar[r]^{x_0^{\op}} & X_0} \Big)\Big)\\
&=\D_{\sB}^{-1} \Big( \Big(\xymatrix{F(X_2) \ar[r]^{F(x_1)^{\op}} & F(X_1) \ar[r]^{F(x_0)^{\op}} & F(X_0)} \Big)\Big)\\
&=\Big(\xymatrix{F(X_0) \ar[r]^{F(x_0)} & F(X_1) \ar[r]^{F(x_1)} & F(X_2)} \Big)\\
&= (\Adel(F)^{\op})(X)
\end{align*}
and
\begin{align*}
(\D_{\sB}^{-1} \circ  \Adel(F^{\op}) \circ \D_{\sA})([f]^{\op})& =  (\D_{\sB}^{-1} \circ  \Adel(F^{\op}))([f_2^{\op},f_1^{\op},f_0^{\op}])\\
&=\D_{\sB}^{-1} ([F(f_2)^{\op},F(f_1)^{\op},F(f_0)^{\op}])\\
&=[F(f_0),F(f_1),F(f_2)]^{\op}\\
&=(\Adel(F)^{\op})([f]^{\op}).
\end{align*}
Next, we show that $\Adel(F)$ is an exact functor.\\
Suppose given $\xymatrix@1{X \,\ar[r]^{[f]} &\, Y}$ in $\Adel(\sA)$. A kernel of $(\Adel(F))([f])$ is given by
\begin{align*}
&\left[\sm{1\\0},\sm{1\\0},\sm{1\\0}\right] \colon \\
&\Bigg(\xymatrix@C=20mm{ F(X_0)\mathord \oplus F(Y_0) \ar[r]^{\sm{F(x_0)&0\\0&1}} & F(X_1)\mathord \oplus F(Y_0) \ar[r]^{\sm{F(x_1)&F(f_1)\\0&-F(y_0)}} &  F(X_2) \mathord\oplus F(Y_1)} \Bigg) \to (\Adel(F))(A),
\end{align*}
cf. theorem \ref{thrm:kernelcokernel}.\\
By applying lemma \ref{lem:iso} with isomorphisms of the form $F(S) \oplus F(T) \stackrel{\sim}\longrightarrow F(S\oplus T)$  in $\sB$ for 
$S,T \in \Ob \sA$, we see that 
\begin{align*}
&(\Adel(F))([\k(f)]) = \left[F(\sm{1\\0}),F(\sm{1\\0}),F(\sm{1\\0})\right] \colon \\
&\Bigg(\xymatrix@C=20mm{ F(X_0\mathord \oplus Y_0) \ar[r]^{F\left(\sm{x_0&0\\0&1}\right)} & F(X_1\mathord \oplus Y_0) \ar[r]^{F\left(\sm{x_1&f_1\\0&-y_0}\right)} &  F(X_2 \mathord\oplus Y_1)} \Bigg) \to (\Adel(F))(A)
\end{align*}
is a kernel of $(\Adel(F))([f])$ as well, cf. proposition \ref{prop:additivefunctor}.

We conclude that $\Adel(F)$ is left-exact, cf. lemma \ref{lem:exactfunctor}.

The functor $\Adel(F^{\op}) \colon \Adel(\sA^{\op}) \to \Adel(\sB^{\op})$ is left-exact, so
 \\ $\Adel(F)^{\op}=\D_{\sB}^{-1} \circ  \Adel(F^{\op}) \circ \D_{\sA} \colon \Adel(\sA)^{\op} \to \Adel(\sB)^{\op}$ is left-exact too. 

We conclude that $\Adel(F)$ is also right-exact and therefore exact.

Now we want to show that $\varepsilon^F$ is an isotransformation. 

The morphisms $(\varepsilon^F)_S$ are well-defined for $S \in \Ob \sA$ since we have $0 \cdot 1 = 0=0 \cdot 0$ and $0 \cdot 0 = 0 = 1 \cdot 0$.\\
Note that $F(0_{\sA})$ is a zero object in $\sB$, cf. proposition \ref{prop:additivefunctor}.\\
The inverse of $(\varepsilon^F)_S$ is given by 
\begin{align*}
[0,1,0] \colon  \Big( \xymatrix{F(0) \ar[r]^-{0} & F(S) \ar[r]^-{0} & F(0)} \Big)      \to   \Big( \xymatrix{0 \ar[r]^-{0} & F(S) \ar[r]^-{0} & 0} \Big).
\end{align*}
for $S \in \Ob \sA$.\\
We have
\begin{align*}
(\I_{\sB} \circ F)(u) (\varepsilon^F)_T = [0,F(u),0] [0,1,0] = [0,1,0][F(0),F(u),F(0)] =(\varepsilon^F)_S (\Adel(F) \circ \I_{\sA})(u)
\end{align*}
for $\xymatrix@1{S\, \ar[r]^{u}&\,T}$ in $\sA$, which implies the naturality of $\varepsilon^F$.

Ad (b).

The transformation $\Adel(\alpha)$ is well-defined since we have 
$F(x_0)\alpha_{X_1}=\alpha_{X_0}G(x_0)$, $F(x_1) \alpha_{X_2} = \alpha_{X_1} G(x_1)$ and 
\begin{align*}
(\Adel(F))([f]) (\Adel(\alpha))_Y &= [F(f_0),F(f_1),F(f_2)] [\alpha_{Y_0},\alpha_{Y_1},\alpha_{Y_2}] \\
&= [F(f_0)\alpha_{Y_0},F(f_1)\alpha_{Y_1},F(f_2)\alpha_{Y_2}]\\
&= [\alpha_{X_0}G(f_0),\alpha_{X_1}G(f_1),\alpha_{X_2}G(f_2)]\\
&=[\alpha_{X_0},\alpha_{X_1},\alpha_{X_2}][G(f_0),G(f_1),G(f_2)] \\
&= (\Adel(\alpha))_X (\Adel(G))([f])
\end{align*}
for $\xymatrix@1{X\, \ar[r]^{[f]} & \,Y}$ in $\Adel(\sA)$.

If $\alpha$ is an isotransformation, then the inverse of $(\Adel(\alpha))_X$ is given by\\ $\left[(\alpha_{X_0})^{-1},(\alpha_{X_1})^{-1},(\alpha_{X_2})^{-1}\right]$ for $X \in  \Ob(\Adel(\sA))$.

\end{proof}

\begin{prop}
Suppose given $\xymatrix@1{ \sA\,  \ar@/^5mm/[rr]^F="F"   \ar[rr]|{\phantom{\rule{0.4mm}{0mm}}G\phantom{\rule{0.4mm}{0mm}}}="G" \ar@/_5mm/[rr]_H="H" & &\,\sB\, \ar@/^/[r]^K="K" \ar@/_/[r]_L="L" &\, \sC 
\ar@2"F"+<0mm,-2.6mm>;"G"+<0mm,1.4mm>_\alpha
\ar@2"G"+<0mm,-1.5mm>;"H"+<0mm,2.4mm>_\beta
\ar@2"K"+<0mm,-2.6mm>;"L"+<0mm,2.4mm>_\gamma}$ with additive categories $\sA$, $\sB$ and $\sC$ and additive functors $F, G,H,K$ and $L$.

The following equations hold.
\begin{itemize}
	\item[(a)] $\Adel(K \circ F) = \Adel(K) \circ \Adel(F)$
	\item[(b)] $\Adel(\alpha \beta) = \Adel(\alpha) \Adel(\beta)$
	\item[(c)] $\Adel(\gamma \star \alpha) = \Adel(\gamma) \star \Adel(\alpha)$
\end{itemize}
\end{prop}

\begin{proof}
Ad (a). Suppose given $\xymatrix@1{X\, \ar[r]^{[f]} &\, Y}$ in $\Adel(\sA)$. We have
\begin{align*}
( \Adel(K) \circ \Adel(F)) (X)&=(\Adel(K))\Big(\Big( \xymatrix{F(X_0) \ar[r]^{F(x_0)} & F(X_1) \ar[r]^{F(x_1)} & F(X_2)} \Big)\Big)\\
&= \Big( \xymatrix@C=16mm{(K \circ F)(X_0) \ar[r]^{(K \circ F)(x_0)} & (K \circ F)(X_1) \ar[r]^{(K \circ F)(x_1)} & (K \circ F)(X_2)} \Big)\\
&=(\Adel(K \circ F) )(X)
\end{align*}
and
\begin{align*}
( \Adel(K) \circ \Adel(F)) ([f])&=(\Adel(K))([F(f_0),F(f_1),F(f_2)])\\
&= [(K\circ F)(f_0),(K\circ F)(f_1),(K\circ F)(f_2)]\\
&=(\Adel(K \circ F) )([f]).
\end{align*}

Ad (b). Suppose given $X \in \Ob(\Adel(\sA))$. We have
\begin{align*}
(\Adel(\alpha)\Adel(\beta))_X &= (\Adel(\alpha))_X (\Adel(\beta))_X \\
&= [\alpha_{X_0},\alpha_{X_1},\alpha_{X_2}][\beta_{X_0},\beta_{X_1},\beta_{X_2}] \\
&= [\alpha_{X_0}\beta_{X_0},\alpha_{X_1}\beta_{X_1},\alpha_{X_2}\beta_{X_2}] \\
&= [(\alpha\beta)_{X_0},(\alpha\beta)_{X_1},(\alpha\beta)_{X_2}]\\
&= (\Adel(\alpha \beta))_X.
\end{align*}

Ad (c). Suppose given $X \in \Ob(\Adel(\sA))$. We have
\begin{align*}
(\Adel(\gamma) \star \Adel(\alpha))_X &= (\Adel(K))((\Adel(\alpha))_X) (\Adel(\gamma))_{(\Adel(G))(X)}\\
&= [K(\alpha_{X_0}),K(\alpha_{X_1}),K(\alpha_{X_2})] [\gamma_{G(X_0)},\gamma_{G(X_1)},\gamma_{G(X_2)}]\\
&= [K(\alpha_{X_0})\gamma_{G(X_0)},K(\alpha_{X_1})\gamma_{G(X_1)},K(\alpha_{X_2})\gamma_{G(X_2)}]\\
&=[ (\gamma \star \alpha)_{X_0},(\gamma \star \alpha)_{X_1},(\gamma \star \alpha)_{X_2}]\\
&= (\Adel(\gamma \star \alpha))_X.
\end{align*}
\end{proof}

\chapter{A universal property for the Adelman category}
\thispagestyle{empty}
\label{ch:universalproperty}
\section{The homology functor in the abelian case}
\label{sec:abeliancase}
Suppose given an abelian category $\sB$ throughout this section \ref{sec:abeliancase}.

\begin{lem}\label{lem:h1}
Suppose given $\xymatrix@1@C=10mm{A \oplus B \,\ar[r]^{\sm{1&f\\0&-b}}& \,A \oplus C}$ in $\sB$.

If $c \colon C \to D$ is a cokernel of $b$, then $\sm{fc\\-c}\colon A \oplus C \to D$ is a cokernel of $\sm{1 & f\\0&-b}$.
\end{lem}

\begin{proof}
We have $\sm{1&f\\0&-b} \sm{fc\\-c} = \sm{0\\bc} =0$. Suppose given $\sm{g_1\\g_2} \colon A\oplus C \to X$ in $\sB$ with $\sm{1&f\\0&-b}\sm{g_1\\g_2}=0$. We have $g_1 + fg_2 = 0$ and $b g_2 = 0$. There exists $u \colon D \to X$ with $(-c)(-u) = cu=g_2$ since $c$ is a cokernel of $b$. We conclude that $\sm{fc\\-c} (-u) =\sm{-fg_2\\g_2} = \sm{g_1\\g_2}$ holds. The induced morphism $-u$ is unique because we necessarily have $(-c)(-u)=cu=g_2$.
\end{proof}

\begin{lem} \label{lem:h2}
Suppose given $\xymatrix@1{A \oplus B \,\ar[r]^{\sm{f&0\\0&1}}& \,C \oplus B}$ in $\sB$.
\begin{itemize}
	\item[(a)] 
If $c \colon C \to D$ is a cokernel of $f$ then $\sm{c\\0} \colon C \oplus B \to D$ is a cokernel of $\sm{f & 0\\0&1}$.
	\item[(b)]If $k \colon K \to A$ is a kernel of $f$, then $\sm{k&0}\colon K \to A \oplus B$ is a kernel of $\sm{f & 0\\0&1}$.
\end{itemize}
\end{lem}

\begin{proof}
Ad (a). We have $\sm{f&0\\0&1} \sm{c\\0} = \sm{fc\\0} = 0$. Suppose given $\sm{g_1 \\g_2} \colon C \oplus B \to X$ in $\sB$ with $\sm{f&0\\0&1} \sm{g_1\\g_2} = 0$. We have $fg_1=0$ and $g_2=0$. There exists $u \colon D \to X$ with $cu=g_1$ since $c$ is a cokernel of $f$. We conclude that $\sm{c\\0} u = \sm{cu\\0} =\sm{g_1\\0} = \sm{g_1 \\ g_2}$ holds. The induced morphism $u$ is unique because we necessarily have $cu=g_1$.

Ad (b). This is dual to (a).
\end{proof}

\begin{rem}
To prove the previous lemmata \ref{lem:h1} and \ref{lem:h2}, one may alternatively argue that a sequence $\xymatrix@1@C=11mm{A'\oplus B'\, \ar[r]^-{\sm{u&uf\\0&-b}} &\, A \oplus B\, \ar[r]^-{\sm{v&fc\\0&-c}}&\,A''\oplus B''}$ in $\sB$ is left-(right-)exact if and only if the sequence $\xymatrix@1@C=11mm{A'\oplus B'\, \ar[r]^-{\sm{u&0\\0&b}} & \,A \oplus B\, \ar[r]^-{\sm{v&0\\0&c}}&\,A''\oplus B''}$ is left-(right-)exact using the isomorphism $\sm{1&f\\0&-1} \colon A \oplus B \to A \oplus B$.
\begin{align*}
\xymatrix{
A'\oplus B'\ar[dd]^{1} \ar[rr]^{\sm{u&uf\\0&-b}} && A \oplus B \ar@<.5ex>[dd]^{\sm{1&f\\0&-1}} \ar[rr]^{\sm{v&fc\\0&-c}}&&A''\oplus B'' \ar[dd]^{1}  \\ \\
A'\oplus B' \ar[rr]^{\sm{u&0\\0&b}} && A \oplus B \ar@<.5ex>[uu]^{\sm{1&f\\0&-1}} \ar[rr]^{\sm{v&0\\0&c}}&&A''\oplus B''
}
\end{align*}

\end{rem}

\begin{lem} \label{lem:schubert}
Suppose given the following commutative diagram in $\sB$ with a right-exact sequence $\xymatrix@1{A_1\, \ar[r]^{a_1} & \,A_2\, \are[r]^*+<0.3mm,0.3mm>{\scriptstyle a_2}&\, A_3}$ and a left-exact sequence $\xymatrix@1{B_1 \,\arm[r]^{b_1} &\, B_2\, \ar[r]^{b_2} &\, B_3}$.
\begin{align*}
\xymatrix{
A_1 \ar[r]^{a_1} \ar[d]^{f_1}& A_2 \are[r]^*+<0.3mm,0.3mm>{\scriptstyle a_2}  \ar[d]^{f_2}& A_3  \ar[d]^{f_3}\\
B_1 \arm[r]^{b_1} & B_2 \ar[r]^{b_2} & B_3
}
\end{align*}
Suppose given kernels $k_i \colon K_i \to A_i$ of $f_i$ for $i \in [1,3]$. The induced morphisms between the kernels shall be $u \colon K_1 \to K_2$ and $v \colon K_2 \to K_3$, cf. lemma \ref{lem:inducedmorphisms}~(a).
\begin{align*}
\xymatrix{
K_1 \arm[d]^{k_1} \ar[r]^u & K_2 \arm[d]^{k_2} \ar[r]^v & K_3 \arm[d]^{k_3}\\
A_1 \ar[r]^{a_1} \ar[d]^{f_1}& A_2 \are[r]^*+<0.3mm,0.3mm>{\scriptstyle a_2}  \ar[d]^{f_2}& A_3  \ar[d]^{f_3}\\
B_1 \arm[r]^{b_1} & B_2 \ar[r]^{b_2} & B_3
}
\end{align*}
The sequence $\xymatrix@1{K_1\, \ar[r]^u & \,K_2 \,\ar[r]^v & \,K_3}$ is exact.
\end{lem}

\begin{proof}
See \cite[prop~13.5.9~(a)]{schubert}.
\end{proof}

The following lemma is a part of \cite[prob~13.6.8]{schubert}.

\begin{lem} \label{lem:exact}
Suppose given $\xymatrix@1{A \,\ar[r]^f &\, B\, \ar[r]^g & \,C}$ in $\sB$. Suppose that $c\colon B \to D$ is a cokernel of $f$, $k \colon K \to A$ is a kernel of $fg$ and $\ell \colon L \to B$ is a kernel of $g$.
\begin{align*}
\xymatrix{
K \arm[d]_k & L \arm[d]^{\ell} \\
A \ar[r]^f \ar[d]_{fg} & B \are[r]^*+<0.8mm,0.8mm>{\scriptstyle c} \ar[d]^g &D\\
C \ar[r]^{1_C} & C
}
\end{align*}
The induced morphism between the kernels shall be  $u\colon K \to L$.

The sequence $\xymatrix@1{K \,\ar[r]^u &\, L\, \ar[r]^{\ell c} & \,D}$ is exact.
\end{lem}

\begin{proof}
We choose an image $\xymatrix@1{B \,\are[r]^*+<0.6mm,0.6mm>{\scriptstyle p} &\, I\, \arm[r]^i & \,C}$ of $g$. Let $d \colon C \to E$ be a cokernel of $g$. Note that $d$ is a cokernel of $i$ as well since $p$ is epimorphic.

Therefore we have the following commutative diagram in $\sB$ with $\xymatrix@1{A\,\ar[r]^f &\, B\, \are[r]^*+<0.8mm,0.8mm>{\scriptstyle c}&\,D}$ right-exact and $\xymatrix@1{I\, \arm[r]^i &\,C\, \are[r]^*+<0.8mm,0.8mm>{\scriptstyle d}&\,E}$ left-exact, cf. remark \ref{rem:kernelcokernelabeliancategory}~(b).
\begin{align*}
\xymatrix{
A \ar[r]^{f} \ar[d]_{fp}& B \are[r]^*+<0.8mm,0.8mm>{\scriptstyle c}  \ar[d]^{g}& D  \ar[d]^{0}\\
I \arm[r]^{i} & C \are[r]^*+<0.8mm,0.8mm>{\scriptstyle d} & E
}
\end{align*}
Note that $k$ is a kernel of $fp$ since $i$ is monomorphic and that $1 \colon D \to D$ is a kernel of $0 \colon D \to E$. The following diagram commutes.
\begin{align*}
\xymatrix{
K \arm[d]^k \ar[r]^u & L \arm[d]^\ell \ar[r]^{\ell c}& D \ar[d]^1\\
A \ar[r]^{f} \ar[d]_{fp}& B \are[r]^*+<0.8mm,0.8mm>{\scriptstyle c}  \ar[d]^{g}& D  \ar[d]^{0}\\
I \arm[r]^{i} & C \are[r]^*+<0.8mm,0.8mm>{\scriptstyle d} & E
}
\end{align*}
Lemma \ref{lem:schubert} says that the sequence $\xymatrix@1{K \,\ar[r]^u &\, L\, \ar[r]^{\ell c} &\, D}$ is exact.
\end{proof}

\begin{lem} \label{lem:epi}
Suppose given $\xymatrix@1{A\,\ar[r]^f&B\are[r]^*+<0.8mm,0.8mm>{\scriptstyle c}&\,C\,\are[r]^*+<0.8mm,0.8mm>{\scriptstyle d}&\,D}$ in $\sB$ with $c$ and $d$ epimorphisms. Suppose $cd$ to be a cokernel of $f$. 

The morphism $d$ is a cokernel of $fc$.
\end{lem}

\begin{proof}
Suppose given $g \colon C \to T$ in $\sB$ with $fcg=0$. There exists $u \colon D \to T$ with $cdu=cg$. Since $c$ is epimorphic, we conclude that $du=g$ holds. The induced morphism $u$ is unique since $d$ is an epimorphism.
\end{proof}

\begin{lemdefn} \label{lem:homology} 
 Suppose given tuples $\k=\big(\xymatrix@1{\K_A \,\ar[r]^{\k_A}& \,A_1}\big)_{A \in \Ob(\sB^{\Delta_2})}$,\\
 $\c=\big(\xymatrix@1{ A_1\, \ar[r]^{\c_A}& \,\C_A}\big)_{A \in \Ob(\sB^{\Delta_2})}$,
  $\p=\big(\xymatrix@1{ \K_A \,\ar[r]^{\p_A}&\, \Im_A}\big)_{A \in \Ob(\sB^{\Delta_2})}$ and $\i=\big(\xymatrix@1{ \Im_A \,\ar[r]^{\i_A}& \,\C_A}\big)_{A \in \Ob(\sB^{\Delta_2})}$ of morphisms in $\sB$ such that for $A \in \Ob(\sB^{\Delta_2})$, the morphism $\k_A$ is a kernel of $a_1$, the morphism $\c_A$ is a cokernel of $a_0$ and the diagram $\xymatrix@1{\K_A \,\ar[r]^-{\p_A} &\, \Im_A \ar[r]^-{\i_A}\, &\, \C_A}$ is an image of $\k_A \c_A$ in $\sB$.

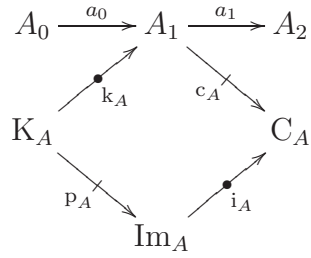
\begin{align*}
\xymatrix{A_0 \ar[r]^{a_0} & A_1 \are[dr]_{\c_A} \ar[r]^{a_1} & A_2 \\ \K_A \are[dr]_{\p_A} \arm[ru]_{\k_A} & & \C_A\\&\Im_A \arm[ru]_{\i_A}}
\end{align*}
\begin{itemize}
	\item[(a)]
The functor $\tilde{\H}_\sB(\k,\c,\p,\i) \colon \sB^{\Delta_2} \to \sB$ shall be defined as follows. 

Let $(\tilde{\H}_\sB(\k,\c,\p,\i))(A):=\Im_A$ for $A \in \Ob(\sB^{\Delta_2})$.

Suppose given $\xymatrix@1{A\, \ar[r]^f&\,B}$ in $\sB^{\Delta_2}$. We have uniquely induced morphisms $f_{\K} \colon \K_A \to \K_B$, $f_{\C} \colon \C_A \to \C_B$ and $(\tilde{\H}_{\sB}(\k,\c,\p,\i))(f) \colon \Im_A \to \Im_B$ such that the following diagram commutes, cf. lemmata \ref{lem:inducedmorphisms} and  \ref{lem:inducedimage}.
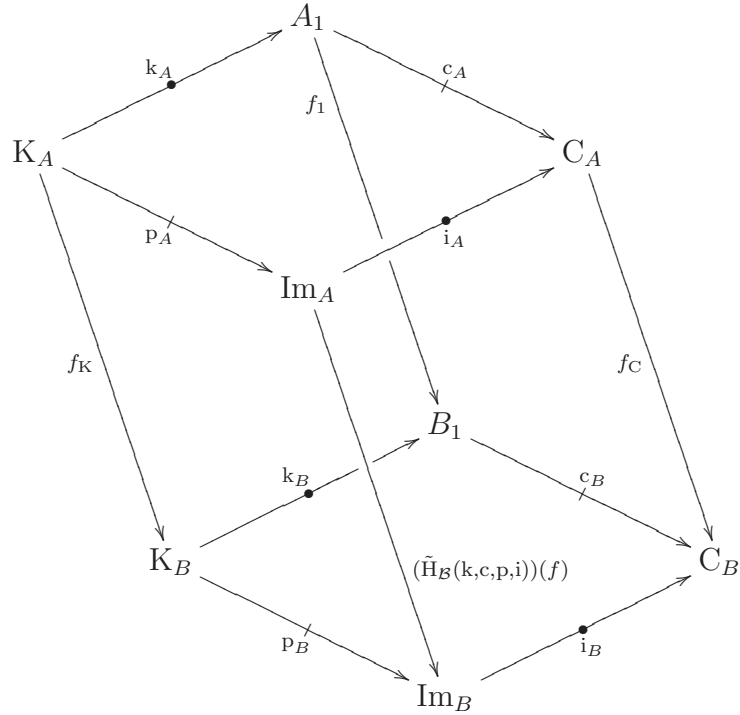
\begin{align*}
\xymatrix@!{ && A_1 \ar[dddr]|(.57){\phantom{\rule{1mm}{10mm}}}_(.2){f_1}\are[drr]^{\c_A}  & 
\\
\K_A \ar[dddr]_{f_{\K}} \are[drr]_{\p_A} \arm[rru]^{\k_A}  & & & & \C_A \ar[dddr]_{f_{\C}}\\ &&\Im_A \ar[dddr]^(.7){(\tilde{\H}_{\sB}(\k,\c,\p,\i))(f)} \arm[rru]_{\i_A}\\
&&&B_1 \are[drr]^{\c_B} 
 \\ &\K_B \are[drr]_{\p_B} \arm[rru]|(.71){\phantom{\rule{2mm}{2mm}}}^{\k_B} &&  && \C_B\\
 &&&\Im_B \arm[rru]_{\i_B}
}
\end{align*}
Then $\tilde{\H}_{\sB}(\k,\c,\p,\i)$ is a well-defined additive functor. We have $(\tilde{\H}_{\sB}(\k,\c,\p,\i))(f) = 0$ for \\
$f \in \Mor (\sB^{\Delta_2})$ null-homotopic. 

\item[(b)] There exists a unique additive functor $\H_{\sB}(\k,\c,\p,\i) \colon \Adel(\sB) \to \sB$ such that \\
$\H_{\sB}(\k,\c,\p,\i) \circ \R_{\sB} =\tilde{\H}_{\sB}(\k,\c,\p,\i)$ holds, cf. theorem \ref{thrm:upf}.

The functor $\H_{\sB}(\k,\c,\p,\i)$ is left-exact. (In fact it is exact, cf. theorem \ref{thrm:Hexact}.)
\end{itemize}
\end{lemdefn}

\begin{proof}
We abbreviate $\tilde \H_{\sB} :=\tilde \H_{\sB}(\k,\c,\p,\i)$ and $ \H_{\sB} := \H_{\sB}(\k,\c,\p,\i)$.

Ad (a).
Note that the induced morphisms between kernels, cokernels and images are unique. Therefore it is sufficient to check commutativity.

Suppose given $\xymatrix@1{A\, \ar@<0.4ex>[r]^{f}\ar@<-0.4ex>[r]_{h}&\,B\,\ar[r]^{g}&\,C}$ in $\sB^{\Delta_2}$.

We have $(fg)_{\K} =f_{\K} g_{\K}$ since $f_{\K} g_{\K} \cdot \k_C = f_{\K} \k_B g_1 = \k_A \cdot f_1 g_1$ holds.

We have $\tilde{\H}_{\sB}(fg) = \tilde{\H}_{\sB}(f)\tilde{\H}_{\sB}(g)$ since $\p_A \cdot  \tilde{\H}_{\sB}(f)\tilde{\H}_{\sB}(g) = f_{\K} \p_B \tilde{\H}_{\sB}(g) = f_{\K} g_{\K} \p_C = (fg)_{\K} \p_C = \p_A \cdot  \tilde{\H}_{\sB}(fg)$ holds and since $\p_A$ is epimorphic. 

We have $(1_A)_{\K} = 1_{\K_A}$, since $1_{\K_A} \k_A = \k_A = \k_A 1_{A_1}$ holds.

We have $\tilde{\H}_{\sB}(1_A) = 1_{\tilde{\H}_{\sB}(A)}$ since $\p_A 1_{\tilde{\H}_{\sB}(A)} = \p_A = 1_{\K_A} \p_A = \p_A \tilde{\H}_{\sB}(1_A) $ holds and since $\p_A$ is epimorphic. 

We have $(f+h)_{\K} = f_{\K} + h_{\K}$ since $(f_{\K} + h_{\K})\k_B = f_{\K}\k_B + h_{\K} \k_B = \k_A f_1 + \k_A h_1 = \k_A(f_1 + h_1)$ holds.

We have $\tilde{\H}_{\sB}(f+h) = \tilde{\H}_{\sB}(f) + \tilde{\H}_{\sB}(h)$ since $\p_A (\tilde{\H}_{\sB}(f) + \tilde{\H}_{\sB}(h)) = \p_A \tilde{\H}_{\sB}(f)+ \p_A \tilde{\H}_{\sB}(h) = f_{\K} \p_B + h_{\K}\p_B = (f_{\K}+h_{\K})\p_B = \p_A \tilde{\H}_{\sB}(f+h)$ holds and since $\p_A$ is epimorphic. 

So $\tilde{\H}_{\sB}$ is a well-defined additive functor.

Now suppose given $\xymatrix@1{A\, \ar[r]^f&\,B}$ in $\sB^{\Delta_2}$ with $f$ null-homotopic. There exist morphisms $s \colon A_1 \to B_0$ and $t \colon A_2 \to B_1$ satisfying $sb_0 + a_1t = f_1$.

We have 
\begin{align*}
\p_A \tilde{\H}_{\sB}(f) \i_B = f_{\K} \p_B \i_B = f_{\K} \k_B \c_B = \k_A f_1 \c_B = \k_A(s b_0 + a_1 t) \c_B = \k_A s\underbrace{ b_0 \c_B}_{=0} + \underbrace{\k_A a_1 }_{=0}t \c_B = 0.
\end{align*}
This implies $\tilde{\H}_{\sB}(f)=0$ since $\p_A$ is an epimorphism and since $\i_B$ is a monomorphism.

\bigskip

Ad (b). We want to show that $\H_{\sB}$ is left-exact.

Suppose given $\xymatrix@1{A \,\ar[r]^{[f]}&\,B}$ in $\Adel(\sB)$. By theorem \ref{thrm:kernelcokernel}, the morphism $[k(f)]$ is a kernel of $[f]$ in $\Adel(\sB)$. 
Lemma \ref{lem:exactfunctor} says that it is sufficient to show that $  \H_{\sB}([\k(f)])=\tilde{\H}_{\sB}(\k(f))$ is a kernel of $ \H_{\sB}([f])=\tilde{\H}_{\sB}(f)$.

The morphism $\c_{\K(f)} (\k(f))_{\C} = (\k(f))_1 \c_A = \sm{1\\0} \c_A = \sm{\c_A\\0}$ is a cokernel of
 $(\K(f))(\xymatrix@1@C=4mm{0\ar[r]&1})=\sm{a_0&0\\0&1} \colon A_0\oplus B_0 \to A_1 \oplus B_0$, cf. lemma \ref{lem:h2}~(a).\\
  The morphism $\c_{\K(f)}$ is a cokernel of $(\K(f))(\xymatrix@1@C=4mm{0\ar[r]&1})=\sm{a_0&0\\0&1}$ by definition.

Therefore $(\k(f))_{\C}$ is an isomorphism, so $\tilde{\H}_{\sB}(\k(f)) \i_A = \i_{\K(f)} (\k(f))_{\C}$ is a monomorphism. 

This implies that $\tilde{\H}_{\sB}(\k(f))$ is a monomorphism and that, consequently, the diagram \\
$\xymatrix@1@C=14mm{\K_{\K(f)}\, \arefl{.42}[r]^-*+<0.3mm,0.3mm>{\scriptstyle \p_{\K(f)}} &\, \tilde{\H}_{\sB}(\K(f)) \,\armfl{.55}[r]^-{\tilde{\H}_{\sB}(\k(f))} &\,\tilde{\H}_{\sB}(A)}$ is an image of $\p_{\K(f)} \tilde{\H}_{\sB}(\k(f)) = (\k(f))_{\K} \p_A$.

\bigskip

We choose an image $\xymatrix@1{\tilde{\H}_{\sB}(A) \,\arefl{.62}[r]^-*+<0.6mm,0.6mm>{\scriptstyle p} &\, I\, \armfl{.36}[r]^-i & \,\tilde{\H}_{\sB}(B)}$ of the morphism $\tilde{\H}_{\sB}(f)$.

Then $\xymatrix@1{\K_A\, \are[r]^*+<0.6mm,0.6mm>{\scriptstyle \p_A p} & \,I\, \armfl{.45}[r]^{i \i_B}&\, \C_B}$ is an image of $\p_A p i \i_B = \p_A \tilde{\H}_{\sB}(f) \i_B = f_{\K} \p_B \i_B  =f_{\K} \k_B \c_B = \k_A f_1 \c_B$.

Consider the morphisms $\sm{1&f_1 \\ 0&-b_0} \colon A_1 \oplus B_0 \to A_1 \oplus B_1$ and $\sm{a_1&0\\0&1} \colon A_1 \oplus B_1 \to A_2 \oplus B_1$. 

A cokernel of $\sm{1&f_1 \\ 0&-b_0}$ is given by $\sm{f_1 \c_B \\ - \c_B} \colon A_1 \oplus B_1 \to \C_B$ since $\c_B$ is a cokernel of $b_0$, cf. lemma~\ref{lem:h1}.

A kernel of $\sm{a_1&0\\0&1}$ is given by $\sm{\k_A &0}\colon \K_A \to A_1 \oplus B_1$, cf. lemma~\ref{lem:h2}~(b). 

A kernel of $(\K(f))(\xymatrix@1@C=4mm{1\ar[r]&2})=\sm{a_1 & f_1\\0&-b_0} = \sm{1&f_1 \\ 0&-b_0}\sm{a_1&0\\0&1}$ is given by \\ $\sm{x&y}:=\k_{\K(f)}  \colon \K_{\K(f)} \to A_1 \oplus B_0$. Consequently, we have $xf_1-yb_0=0$. 
\begin{align*}
\xymatrix{
\K_{\K(f)}\ar[dd]_{\sm{x&y}} & &\K_A\ar[dd]^{\sm{\k_A & 0}}\\ \\
A_1 \oplus B_0 \ar[rr]^{\sm{1&f_1\\0&-b_0}} \ar[dd]_{\sm{a_1&f_1\\0&-b_0}}&& A_1 \oplus B_1 \ar[dd]^{\sm{a_1&0\\0&1}} \ar[rr]^{\sm{f_1\c_B\\-\c_B}}&& \C_B\\ \\
A_2 \oplus B_1 \ar[rr]^{1_{A_2\mathord \oplus B_1}}&& A_2 \oplus B_1
}
\end{align*}
The induced morphism between the kernels is given by $(\k(f))_{\K}$, since we have 
\begin{align*}
(\k(f))_{\K}\k_A = \k_{\K(f)} (\k(f))_1 = \sm{x& y }\sm{1\\0}=x
\end{align*}
and therefore
\begin{align*}
\sm{x&y}\sm{1&f_1\\0&-b_0} = \sm{x&\ xf_1-yb_0} = \sm{(\k(f))_{\K} \k_A &\ 0} = (\k(f))_{\K} \sm{\k_A & 0}.
\end{align*}

Lemma \ref{lem:exact} now implies that the sequence $\xymatrix@1@C=14mm{\K_{\K(f)}\, \ar[r]^{(\k(f))_{\K}}&\,\K_A \,\ar[r]^{\k_A f_1 \c_B}&\,\C_B}$ is exact since \\
$\sm{\k_A &0} \sm{f_1 \c_B\\-\c_B} = \k_A f_1 \c_B$ holds.

We conclude that $\p_A p$ is a cokernel of $(\k(f))_{\K}$, cf. remark \ref{rem:exactsequence}. Lemma \ref{lem:epi} says that $p$ is a cokernel of $(\k(f))_{\K} \p_A$ because $\p_A$ and $p$ are epimorphic.

Therefore the sequence $\xymatrix@1@C=18mm{\K_{\K(f)} \,\ar[r]^-{(\k(f))_{\K} \p_A}&\,\tilde{\H}_{\sB}(A)\, \ar[r]^-{\tilde{\H}_{\sB}(f)}&\,\tilde{\H}_{\sB}(B)}$ is exact, so $\tilde{\H}_{\sB}(\k(f))$ is a kernel of $\tilde{\H}_{\sB}(f)$.
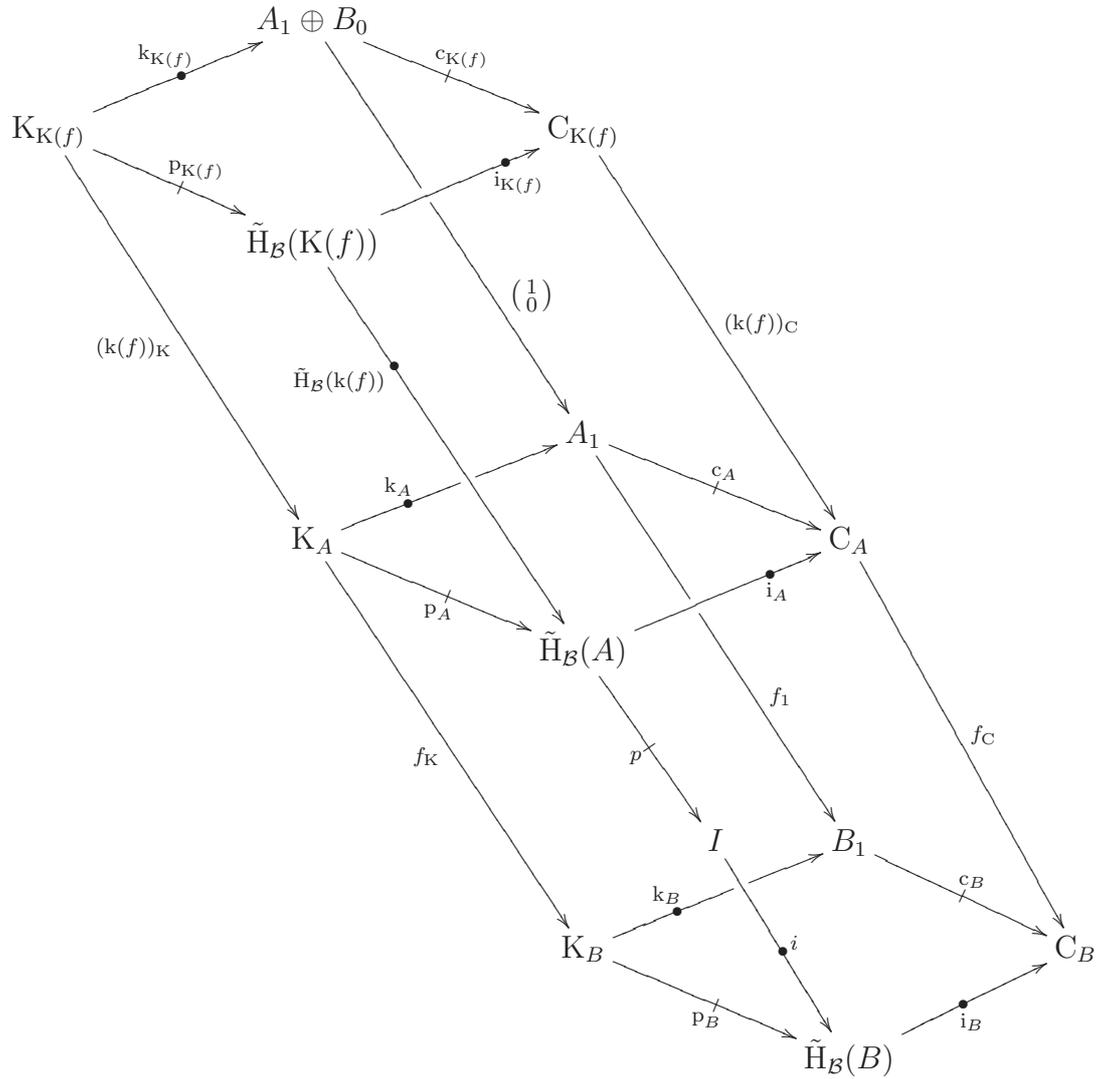
\begin{align*}
\xymatrix{ && A_1 \oplus B_0 \ar[ddddrr]|(.42){\phantom{\rule{2mm}{2mm}}}^(.7){\sm{1\\0}}\are[drr]^{\c_{\K(f)}}  & 
\\
\K_{\K(f)} \ar[ddddrr]_{(\k(f))_{\K}} \are[drr]^{\p_{\K(f)}} \arm[rru]^{\k_{\K(f)}}  & & & & \C_{\K(f)} \ar[ddddrr]^{(\k(f))_{\C}}
\\
 &&\tilde{\H}_{\sB}(\K(f)) \armfl{.3}[ddddrr]_(.3){\tilde{\H}_{\sB}(\k(f))} \armfl{.71}[rru]_(.71){\i_{\K(f)}}\\
\\
&&&& A_1 \ar[ddddrr]|(.42){\phantom{\rule{2mm}{2mm}}}^(.67){f_1} \are[drr]^{\c_A}  & 
\\
&&\K_A \ar[ddddrr]_{f_{\K}} \are[drr]_{\p_A} \armfl{.35}[rru]|(.575){\phantom{\rule{2mm}{2mm}}}^(.35){\k_A}  & & & & \C_A \ar[ddddrr]^{f_{\C}}\\
&&&&\tilde{\H}_{\sB}(A) \are[ddr]_{p} \armfl{.7}[rru]_(.7){\i_A}\\
\\
&&&&&I\arm[ddr]^{i}&B_1 \are[drr]^{\c_B} \\
&&& &\K_B \are[drr]_{\p_B} \armfl{.35}[rru]|(.595){\phantom{\rule{2mm}{2mm}}}^(.35){\k_B} &&  && \C_B\\
&&&&&&\tilde{\H}_{\sB}(B) \arm[rru]_{\i_B}
}
\end{align*}
\end{proof}

\begin{rem}
To show that null-homotopic morphisms are sent to $0$ by $\tilde{\H}_{\sB}(\k,\c,\p,\i)$ in the previous lemma \ref{lem:homology}, one may also use remark \ref{rem:adelideal} and check $(\tilde{\H}_{\sB}(\k,\c,\p,\i))(S)=0$ for $S \in \oS_{\sB}$.
\end{rem}

\begin{lemdefn}[Homology functor] \label{defn:H}
Let $\nZ_{\sB} \subseteq \Ob \sB$ be the set of zero objects in $\sB$.

Suppose given $A \in \Ob(\sB^{\Delta_2})$.

We choose a kernel $\k_A \colon \K_A \to A_1$ of $a_1$. In case $A_2 \in \nZ_{\sB}$, we choose $ \k_{A}:= 1_{A_1}$.

We choose a cokernel $\c_A \colon A_1 \to \C_A$ of $a_0$. In case $A_0 \in \nZ_{\sB}$, we choose $ \c_{A}:= 1_{A_1}$. 

We choose an image $\xymatrix@1{\K_A \,\ar[r]^-{\p_A}&\, \Im_A \,\ar[r]^-{\i_A} & \,\C_A}$ of $\k_A \c_A$.\\
 In case $A_0,A_2 \in \nZ_{\sB}$, we choose $\big(\xymatrix@1{\K_A\, \ar[r]^-{\p_A}& \,\Im_A\, \ar[r]^-{\i_A} & \,\C_A}\big):= \big( \xymatrix{A_1\,\ar[r]^{1}& \,A_1\,  \ar[r]^{1}& \,A_1 } \big)$.

By applying lemma \ref{lem:homology} with the tuples 
$\k:=\big(\xymatrix@1{\K_A \,\ar[r]^{\k_A}& \,A_1}\big)_{A \in \Ob(\sB^{\Delta_2})}$, \\
$\c:=\big(\xymatrix@1{ A_1\, \ar[r]^{\c_A}& \,\C_A}\big)_{A \in \Ob(\sB^{\Delta_2})}$, 
$\p:=\big(\xymatrix@1{ \K_A \,\ar[r]^{\p_A}& \,\Im_A}\big)_{A \in \Ob(\sB^{\Delta_2})}$ and \\
$\i:=\big(\xymatrix@1{ \Im_A \,\ar[r]^{\i_A}&\, \C_A}\big)_{A \in \Ob(\sB^{\Delta_2})}$ of morphisms in $\sB$,
 we obtain functors\\ $\tilde{\H}_\sB:=\tilde{\H}_\sB(\k,\c,\p,\i) \colon \sB^{\Delta_2} \to \sB$ and $\H_{\sB} :=\H_{\sB}(\k,\c,\p,\i) \colon \Adel(\sB) \to \sB$. 
 We call $\H_{\sB}$ the \emph{homology functor} of $\sB$.

For $\xymatrix@1{A \,\ar[r]^f&\,B}$ in $\sB^{\Delta_2}$ with $A_0,A_2,B_0,B_2 \in \nZ_{\sB}$, we have \begin{align*}
\H_{\sB}(A)=\tilde{\H}_\sB (A)=A_1 \qquad \text{ and } \qquad \H_{\sB}([f])=\tilde{\H}_\sB(f)=f_1.
\end{align*}
In particular, we have $\tilde{\H}_{\sB} \circ \tilde{\I}_{\sB} = 1_{\sB}$ and $\H_{\sB} \circ \I_{\sB} = 1_{\sB}$.
\begin{align*}
\xymatrix{
\sB \ar[rr]^{\tilde{\I}_{\sB}} \ar[drdr]_{\I_{\sB}}&& \sB^{\Delta_2} \ar[dd]^{\R} \ar[drdr]^{\tilde{\H}_{\sB}} \\ \\
&&\Adel(\sB) \ar[rr]_-{\H_{\sB}} && \sB
}
\end{align*}
Recall that $\D_{\sB} \colon \Adel(\sB)^{\op} \to \Adel(\sB^{\op})$ is an isomorphism of categories, cf. definition \ref{defn:duality}.

 
 By applying lemma \ref{lem:homology} with the tuples 
 \begin{align*}
 &\c^{\circ}:=\Bigg(\xymatrix@C=20mm{ \C_{\D_{\sB}^{-1}(A)} \ar[r]^-{\big(\c_{\D_{\sB}^{-1}(A)}\big)^{\op}}& A_1}\Bigg)_{A \in \Ob\left((\sB^{\op})^{\Delta_2}\right)}, \\
&\k^{\circ}:=\Bigg(\xymatrix@C=20mm{A_1 \ar[r]^{\big(\k_{\D_{\sB}^{-1}(A)}\big)^{\op}}& \K_{\D_{\sB}^{-1}(A)}}\Bigg)_{A \in \Ob\left((\sB^{\op})^{\Delta_2}\right)},
 \\
& \i^{\circ}:=\Bigg(\xymatrix@C=16mm{\C_{\D_{\sB}^{-1}(A)} \ar[r]^{\big(\i_{\D_{\sB}^{-1}(A)}\big)^{\op}}& \Im_{\D_{\sB}^{-1}(A)} }\Bigg)_{A \in \Ob\left((\sB^{\op})^{\Delta_2}\right)}
 && \text{and }
 \\ 
 &\p^{\circ}:=\Bigg(\xymatrix@C=16mm{\Im_{\D_{\sB}^{-1}(A)} \ar[r]^{\big(\p_{\D_{\sB}^{-1}(A)}\big)^{\op}}& \K_{\D_{\sB}^{-1}(A)} }\Bigg)_{A \in \Ob\left((\sB^{\op})^{\Delta_2}\right)}
 \end{align*}
of morphisms in $\sB^{\op}$,
 we obtain functors \\
 $\tilde{\H}_\sB^\circ := \tilde{\H}_{\sB^{\op}}(\c^{\circ},\k^{\circ},\i^{\circ},\p^{\circ})\colon (\sB^{\op})^{\Delta_2} \to \sB$ and $\H_{\sB}^\circ := \H_{\sB^{\op}}(\c^{\circ},\k^{\circ},\i^{\circ},\p^{\circ})\colon \Adel(\sB^{\op}) \to \sB$.

 The equation $(\H_{\sB})^{\op} = \H_{\sB}^{\circ} \circ \D_{\sB}$ holds.
 \begin{align*}
 \xymatrix{
 \Adel(\sB)^{\op} \ar[dd]_{\sD_{\sB}} \ar[rr]^-{(\H_{\sB})^{\op}} && \sB^{\op} \\
 \\
 \Adel(\sB^{\op}) \ar[urur]_{\H_{\sB}^\circ}
 }
 \end{align*}
\end{lemdefn}

\begin{proof}
We want to show that $(\H_{\sB})^{\op} = \H_{\sB}^\circ \circ \D_{\sB}$ holds.

Suppose given $\xymatrix@1{A\, \ar[r]^f&\,B}$ in $\sB^{\Delta_2}$. Consider $(f_2^{\op},f_1^{\op},f_0^{\op}) \colon \D_{\sB}(B) \to \D_{\sB}(A)$ in $(\sB^{\op})^{\Delta_2}$.

We have $(f_2^{\op},f_1^{\op},f_0^{\op})_{\K} = f_{\C}^{\op}$ since $f_{\C}^{\op}\c_{A}^{\op} = (\c_Af_{\C})^{\op} = (f_1 \c_B)^{\op} = \c_{B}^{\op} f_1^{\op}$ holds.


We have ${\tilde \H}_{\sB}^{\circ}((f_2^{\op},f_1^{\op},f_0^{\op}))= \tilde \H_{\sB}(f)^{\op}$ since \begin{align*}
\i_B^{\op} \tilde \H_{\sB}(f)^{\op} = (\tilde \H_{\sB}(f) \i_B)^{\op}= (\i_A f_{\C})^{\op} = f_{\C}^{\op} \i_A^{\op} = (f_2^{\op},f_1^{\op},f_0^{\op})_{\K} \i_A^{\op} = \i_B^{\op} {\tilde \H}_{\sB}^{\circ}((f_2^{\op},f_1^{\op},f_0^{\op}))
\end{align*}
 holds and since $\i_{B}^{\op}$ is epimorphic.

We conclude that
\begin{align*}
(\H_{\sB}^\circ \circ \D_{\sB})(A) = \H_{\sB}^\circ ( \D_{\sB}(A)) = \Im_A = (\H_{\sB})^{\op}(A)
\end{align*}
and that
\begin{align*}
(\H_{\sB}^\circ \circ \D_{\sB})([f]^{\op}) &=   \H_{\sB}^\circ ([f_2^{\op},f_1^{\op},f_0^{\op}]) = {\tilde \H}_{\sB}^{\circ} ((f_2^{\op},f_1^{\op},f_0^{\op})) = \tilde \H_{\sB}(f)^{\op}= \H_{\sB}([f])^{\op} \\&=(\H_{\sB})^{\op}([f]^{\op}).
\end{align*}
\end{proof}

\begin{thrm} \label{thrm:Hexact}
Recall that $\sB$ is an abelian category.

The functor $\H_\sB \colon \Adel(\sB) \to \sB$ is an exact functor.
\end{thrm}

\begin{proof}
We have left-exact functors $\H_{\sB} \colon \Adel(\sB) \to \sB$ and $\H_{\sB}^{\circ} \colon \Adel(\sB^{\op}) \to \sB^{\op}$, cf. definition \ref{defn:H} and lemma  \ref{lem:homology}.
 The functor $(\H_{\sB})^{\op} = \H_{\sB}^{\circ} \circ \D_{\sB} \colon \Adel(\sB)^{\op} \to \sB^{\op}$ is also left-exact since $\D_{\sB}$ is an isomorphism of categories. We conclude that $\H_{\sB}$ is exact.
\end{proof}

\section{The universal property}
\label{sec:universalproperty}
\begin{prop} \label{prop:transformationsubcategories}
Suppose given an abelian category $\sA$ and full subcategories $\sC$ and $\sD$ satisfying $\sC \subseteq \sD \subseteq \sA$. The embedding functor from $\sC$ to $\sD$ shall be denoted by $E \colon \sC \to \sD$.

Suppose given an additive category $\sB$ and additive functors $F,G \colon \sA \to \sB$.

\begin{itemize}
\item[(a)] Suppose that $F$ and $G$ are left-exact.

Suppose that for $D \in \Ob \sD$, there exists a left-exact sequence $\xymatrix@1{D \,\arm[r]^d &\, C\, \ar[r]^c \ar[r] &\, C'}$ in $\sA$ with $C,C' \in \Ob \sC$.

Suppose that the objects in $\sC$ are injective in $\sA$.

Then we have bijections
\begin{align*}
\gamma \colon  \Hom_{\sB^{\sD}}(F|_{\sD}, G|_{\sD}) \to & \Hom_{\sB^{\sC}}(F|_{\sC},G|_{\sC}) \colon
 \tau   \mapsto  \tau \star E
\end{align*}
and
\begin{align*}
\gamma' \colon  \Hom^{\iso}_{\sB^{\sD}}(F|_{\sD}, G|_{\sD}) \to & \Hom^{\iso}_{\sB^{\sC}}(F|_{\sC},G|_{\sC}) \colon
 \tau  \mapsto  \tau \star E.
\end{align*}
\item[(b)] Suppose that $F$ and $G$ are right-exact.

Suppose that for $D \in \Ob \sD$, there exists a right-exact sequence $\xymatrix@1{C\, \ar[r]^c & \,C'\, \are[r]^*+<0.8mm,0.8mm>{\scriptstyle d} & \,D}$ in $\sA$ with $C,C' \in \Ob \sC$.

Suppose that the objects in $\sC$ are projective in $\sA$.

Then we have bijections
\begin{align*}
 \Hom_{\sB^{\sD}}(F|_{\sD}, G|_{\sD}) \to & \Hom_{\sB^{\sC}}(F|_{\sC},G|_{\sC}) \colon \tau \mapsto \tau \star E
\end{align*}
and
\begin{align*}
 \Hom^{\iso}_{\sB^{\sD}}(F|_{\sD}, G|_{\sD}) \to & \Hom^{\iso}_{\sB^{\sC}}(F|_{\sC},G|_{\sC}) \colon
 \tau  \mapsto\tau \star E.
\end{align*}
\end{itemize}
\end{prop}

\begin{proof}
Ad (a). The map $\gamma$ is well-defined, cf. convention \ref{conv:transformation}. Note that $\tau \gamma = (\tau_C)_{C \in \Ob \sC}$ holds for $\tau \in \Hom_{\sB^{\sD}}(F|_{\sD}, G|_{\sD})$.

The map $\gamma'$ is well-defined since $\tau_C$ is an isomorphism for $\tau \in \Hom^{\iso}_{\sB^{\sD}}(F|_{\sD}, G|_{\sD})$ and $C \in \Ob \sC$.

We want to show the injectivity of $\gamma$.

Suppose given $\tau, \sigma \in  \Hom_{\sB^{\sD}}(F|_{\sD}, G|_{\sD})$ with $\tau\gamma=\sigma\gamma$. Then $\tau_C = \sigma_C$ holds for $C \in \Ob \sC$. Suppose given $D \in \Ob \sD$. We have to show that $\tau_D  =\sigma_D$ is true.

There exists a left-exact sequence $\xymatrix@1{D \,\arm[r]^d &\, C \,\ar[r]^c \ar[r] &\, C'}$ in $\sA$ with $C,C' \in \Ob \sC$.

We obtain the left-exact sequences $\xymatrix@1{F(D)\, \arm[r]^{F(d)} &\, F(C)\,\ar[r]^{F(c)} \ar[r] &\, F(C')}$ and   \\ $\xymatrix@1{G(D)\, \arm[r]^{G(d)} &\,G(C)\, \ar[r]^{G(c)} \ar[r] &\, G(C')}$ in $\sB$  by applying $F$ resp. $G$.

Since $\tau$ and $\sigma$ are transformations and since $\tau_C = \sigma_C$ holds, we get $\tau_D G(d) = F(d) \tau_C = F(d) \sigma_C = \sigma_D G(d)$.

Since $G(d)$ is monomorphic, we conclude that $\tau_D = \sigma_D$ holds.

The injectivity of $\gamma'$ is inherited from $\gamma$.

Now we want to show the surjectivity of $\gamma$.

Suppose given $\rho \in \Hom_{\sB^{\sC}}(F|_{\sC},G|_{\sC})$. We have to show that there exists $\tau \in \Hom_{\sB^{\sD}}(F|_{\sD},G|_{\sD})$ such that $\tau \star E = \rho$.

For $D \in \Ob \sD$, we choose a left-exact sequence $\xymatrix@1{D\, \armfl{.45}[r]^{d_D} &\, C_D\, \ar[r]^{c_D} &\, C'_{D}}$ in $\sA$ with $C_D,C'_D \in \Ob \sC$. In case $D \in \Ob \sC$, we choose the left-exact sequence $\xymatrix@1{D \,\arm[r]^1 & \,D\, \ar[r]^0 &\, D}$.

Suppose given $D \in \Ob \sD$.

We obtain the following commutative diagram.
\begin{align*}
\xymatrix{
F(D) \arm[r]^{F(d_D)} & F(C_D) \ar[r]^{F(c_D)} \ar[d]^{\rho_{C_D}} & F(C'_D) \ar[d]^{\rho_{C'_D}}\\
G(D) \arm[r]^{G(d_D)} & G(C_D) \ar[r]^{G(c_D)} & G(C'_D)
}
\end{align*}
The induced morphism between the kernels shall be denoted by $\tau_D \colon F(D) \to G(D)$. Note that $\tau_C = \rho_C$ holds for $C \in \Ob \sC$. If $\rho \in \Hom^{\iso}_{\sB^{\sC}}(F|_{\sC},G|_{\sC})$ is true, then the map $\tau_D$ is an isomorphism, cf. lemma \ref{lem:inducedmorphisms}~(a).

Next, we show the naturality of $\tau = (\tau_D)_{D \in \Ob \sD}$.

Suppose given $\xymatrix@1{X \,\ar[r]^f &\, Y}$ in $\sD$. Since $d_X$ is monomorphic and since $C_Y$ is injective, there exists 
$g \colon C_X \to C_Y$ such that $d_X g = f d_Y$ holds.
\begin{align*}
\xymatrix{
X \ar[d]_{f d_y} \arm[r]^{d_X} & C_X \ar[dl]^g\\ C_Y 
}
\end{align*}
We have
\begin{align*}
&\mathrel {\phantom{=}} \tau_X G(f) G(d_Y) = \tau_X G(d_X) G(g) = F(d_X) \rho_{C_X} G(g)\\
& = F(d_X) F(g) \rho_{C_Y}= F(f) F(d_Y) \rho_{C_Y} = F(f) \tau_Y G(d_Y).
\end{align*}

Since $G(d_Y)$ is monomorphic, we conclude that $\tau_X G(f) = F(f) \tau_Y$ holds.

We have $\tau \gamma = \tau \star E = (\rho_C)_{C \in \Ob \sC}=\rho$, so $\gamma$ is surjective.

As seen above, the map $\tau_D$ is an isomorphism for $\rho \in \Hom^{\iso}_{\sB^{\sC}}(F|_{\sC}, G|_{\sC})$ and $D \in \Ob \sD$, so $\gamma'$ is surjective too.

Ad (b). This is dual to (a).
\end{proof}

\begin{nota}
Given a diagram $\xymatrix@1{ A_0 \,\ar[r]^{a_0} &\, A_1}$ in $\sA$, we obtain a functor $A \in \Ob(\sA^{\Delta_1})$ by setting $A(0):=A_0$, $A(1):=A_1$, $A(\xymatrix@1@C=4mm{0\, \ar[r] &\,1}):=a_0$,  $A(\xymatrix@1@C=4mm{0 \,\ar[r]^{1} &\,0}):=1_{A_0}$ and $A(\xymatrix@1@C=4mm{1 \,\ar[r]^{1} &\,1}):=1_{A_1}$.

For $A \in \Ob(\sA^{\Delta_1})$, we therefore set $A_0:=A(0)$, $A_1:=A(1)$, $a_0:=A(\xymatrix@1@C=4mm{0\, \ar[r] &\,1})$ and write \\
$A=\big(\xymatrix@1{ A_0\, \ar[r]^{a_0} &\, A_1}\big)$.

\smallskip

For a transformation $f \in \Hom_{\sA^{\Delta_1}}(A,B)$, we write $f=(f_0,f_1)$ instead of $f=
(f_i)_{i \in \Ob(\Delta_1)}$. We also write
\begin{align*}
 \left(\begin{smallmatrix}{\xymatrix{A\ar[d]^f \\B}}\end{smallmatrix} \! \right) = \left(\begin{smallmatrix}\xymatrix{A_0\ar[d]^{f_0} \ar[r]^{a_0} & A_1\ar[d]^{f_1}  \\ B_0 \ar[r]^{b_0} & B_1 }\end{smallmatrix}\right).
\end{align*}

\smallskip

Given $A,B \in \Ob(\sA^{\Delta_1})$ and morphisms $f_0 \colon A_0 \to B_0$, $f_1 \colon A_1 \to B_1$ in $\sA$ satisfying $a_0 f_1 = f_0 b_0$, we obtain a transformation $f = (f_0,f_1)\in \Hom_{\sA^{\Delta_1}}(A,B)$.

Cf. notation \ref{nota:functor}.
\end{nota}

\begin{lem} \label{lem:delta1exact}
Suppose given an abelian category $\sB$. Suppose given $\xymatrix@1{A \,\ar[r]^f &\,B\, \ar[r]^g & \,C}$ in $\sB^{\Delta_1}$.
\begin{itemize}
	\item[(a)] The sequence $\xymatrix@1{A \,\ar[r]^f &\,B\, \ar[r]^g &\, C}$ in $\sB^{\Delta_1}$ is left-exact in $\sB^{\Delta_1}$ if and only if the sequences $\xymatrix@1{A_0 \,\ar[r]^{f_0} &\,B_0\, \ar[r]^{g_0} &\, C_0}$ and $\xymatrix@1{A_1 \,\ar[r]^{f_1} &\,B_1\, \ar[r]^{g_1} & \,C_1}$ are left-exact in $\sB$.
	\item[(b)] The sequence $\xymatrix@1{A\, \ar[r]^f &\,B\, \ar[r]^g & \,C}$ in $\sB^{\Delta_1}$ is right-exact in $\sB^{\Delta_1}$ if and only if the sequences $\xymatrix@1{A_0\, \ar[r]^{f_0} &\,B_0\, \ar[r]^{g_0} &\, C_0}$ and $\xymatrix@1{A_1 \,\ar[r]^{f_1} &\,B_1 \ar[r]^{g_1} &\, C_1}$ are right-exact in $\sB$.
\end{itemize}
\end{lem}

\begin{proof}
Ad (a). Suppose $\xymatrix@1{A \,\ar[r]^f &\,B\, \ar[r]^g &\, C}$ to be left-exact. 

We want to show that $f_0$ is a kernel of $g_0$.

Suppose given $t \colon X \to B_0$ in $\sB$ with $t g_0=0$. The morphism $(t,tb_0) \colon \big(\xymatrix@1{X\, \ar[r]^1&\,X}\big) \to B$ is well-defined with $(t,tb_0) g =0$ since we have $t \cdot b_0 = 1 \cdot t b_0$ and $(t,tb_0)g=(tg_0,tb_0g_1)=(0,tg_0c_0) = (0,0)=0$.
\begin{align*}
\xymatrix{
X \ar[r]^1 \ar[d]^t & X \ar[d]^{tb_0} \\ B_0 \ar[r]^{b_0}  \ar[d]^{g_0} & B_1 \ar[d]^{g_1}\\C_0 \ar[r]^{c_0} & C_1
}
\end{align*}
There exists $u \colon \big(\xymatrix@1{X \,\ar[r]^1&\,X}\big) \to A$ with $uf=(t,tb_0)$ since $f$ is a kernel of $g$. So $u_0 f_0 = t$. \\
Given $v \colon X \to A_0$ with $vf_0=t$, we have $(v,va_0) \colon \big(\xymatrix@1{X \,\ar[r]^1&\,X}\big) \to A$ in $\sB^{\Delta_1}$ with $(v,va_0)f=(t,tb_0)$ since $v \cdot a_0 = 1 \cdot va_0$ and $(v,va_0)f = (vf_0,va_0f_1)=(t,vf_0b_0)=(t,tb_0)$ hold. So $(v,va_0)=u$ and $v=u_0$, again since $f$ is a kernel of $g$.\\
We conclude that $\xymatrix@1{A_0 \,\ar[r]^{f_0} &\,B_0\, \ar[r]^{g_0} & \,C_0}$ is left-exact.

We want to show that $f_1$ is a kernel of $g_1$.

Suppose given $t \colon X \to B_1$ in $\sB$ with $ t g_1=0$. The morphism $(0,t) \colon \big(\xymatrix@1{0 \,\ar[r]^0 &\,X}\big) \to B$ is well-defined with $(0,t)g=0$ since we have $0 \cdot t = 0 =0 \cdot b_1$ and $(0,t)g=(0g_0,tg_1)=(0,0)=0$.
\begin{align*}
\xymatrix{
0 \ar[r]^0 \ar[d]^0 & X \ar[d]^{t} \\ B_0 \ar[r]^{b_0}  \ar[d]^{g_0} & B_1 \ar[d]^{g_1}\\C_0 \ar[r]^{c_0} & C_1
}
\end{align*}
There exists $u \colon \big(\xymatrix@1{0 \,\ar[r]^0 &\,X}\big) \to A$ with $uf =(0,t)$ since $f$ is a kernel of $g$. So $u_1 f_1=t$.  \\
Given $v \colon X \to A_1$ with $v f_1 = t$, we have $(0,v) \colon \big(\xymatrix@1{0\, \ar[r]^0 &\,X}\big) \to A$ with $(0,v) f = (0,t)$ since $0 \cdot v = 0 = 0 \cdot a_0$ and $(0,v) f = (0 \cdot f_0,v \cdot f_1) = (0,t)$ hold. So $(0,v) = u$ and $v = u_1$, again since $f$ is a kernel of $g$.\\
We conclude that $\xymatrix@1{A_1\, \ar[r]^{f_1} &\,B_1\, \ar[r]^{g_1} & \,C_1}$ is left-exact.

Suppose $\xymatrix@1{A_0 \,\ar[r]^{f_0} &\,B_0\, \ar[r]^{g_0} & \,C_0}$ and $\xymatrix@1{A_1\, \ar[r]^{f_1} &\,B_1\, \ar[r]^{g_1} &\, C_1}$ to be left-exact.

We want to show that $f$ is a kernel of $g$.

Suppose given $t \colon X \to B$ in $\sB^{\Delta_1}$ with $tg=0$. We have $t_0g_0 = 0$ and $t_1 g_1 = 0$, so there exist $u_0 \colon X_0 \to A_0$ and $u_1 \colon X_1 \to A_1$ in $\sB$ such that $u_0f_0 = t_0$ and $u_1 f_1 = t_1$ hold since $f_0$ is a kernel of $g_0$ and $f_1$ is a kernel of $g_1$. The morphism $u=(u_0,u_1) \colon X \to A$ is well-defined since we have $x_0 u_1 f_1 = x_0 t_1 = t_0 b_0 = u_0f_0 b_0 = u_0a_0f_1$, and therefore $x_0u_1=u_0a_0$ because $f_1$ is monomorphic. The equation $uf=t$ holds since we have $uf = (u_0f_0,u_1f_1) = (t_0,t_1)=t$.\\
Given $v \colon X \to A$ with $vf = t$, we have $v_0f_0= t_0$ and $v_1 f_1 = t_1$. So $v_0 = u_0$, $v_1 = u_1$ and therefore $v=u$ hold, again since $f_0$ is a kernel of $g_0$ and since $f_1$ is a kernel of $g_1$.

We conclude that $\xymatrix@1{A \,\ar[r]^{f} &\,B\, \ar[r]^{g} &\, C}$ is left-exact.

Ad (b). \\This is dual to (a) using the canonical isomorphism of categories $(\sB^{\Delta_1})^{\op} \stackrel{\sim}\longrightarrow (\sB^{\op})^{\Delta_1}$. This isomorphism sends $A \in \Ob((\sB^{\Delta_1})^{\op})$ to $\big(\xymatrix@1{A_1 \,\ar[r]^{a_0^{\op}}&\,A_0}\big) \in \Ob((\sB^{\op})^{\Delta_1})$ and $f^{\op} \in \Mor((\sB^{\Delta_1})^{\op})$ to $(f_1^{\op},f_0^{\op}) \in \Mor((\sB^{\op})^{\Delta_1})$.
\end{proof}

\begin{prop} \label{prop:delta1}
Suppose given categories $\sA$ and $\sB$.
\begin{itemize}
	\item[(a)] Suppose that $F,G \colon \sA \to \sB$ are functors and $\alpha\colon F \Rightarrow G$ is a transformation.
We obtain a functor $K \colon \sA \to \sB^{\Delta_1}$ by setting $K(X) := \big( \xymatrix@1{F(X) \,\ar[r]^{\alpha_A} &\, G(X)} \big)$ for $X \in \Ob \sA$ and $K(f) := (F(f),G(f))$ for $f \in \Mor \sA$.

Suppose that $\sA$ and $\sB$ are additive categories. Then $F$ and $G$ are additive if and only if $K$ is additive. 

Suppose that $\sA$ and $\sB$ are abelian categories. Then $F$ and $G$ are exact if and only if $K$ is exact.
	\item[(b)] Conversely, given a functor $K \colon \sA \to \sB^{\Delta_1}$, we get functors $F,G \colon \sA \to \sB$ and a transformation $\alpha \colon F \Rightarrow G$ by setting $F(X) := K(X)_0$, $G(X) := K(X)_1$, $\alpha_X:=(K(X))(\xymatrix@1@C=4mm{0\, \ar[r]& \,1})$ for $X \in \Ob \sA$ and $F(f):=K(f)_0$, $G(f):=K(f)_1$ for $f \in \Mor\sA$.
	
	Therefore every functor $K \colon \sA \to \sB^{\Delta_1}$ is of the form discussed in (a).
\end{itemize}
\end{prop}

I learned this technique from Sebastian Thomas and Denis-Charles Cisinski.

\begin{proof}
Ad (a).  The functor $K \colon \sA \to \sB^{\Delta_1}$ is well-defined since we have
$\alpha_X G(f) = F(f) \alpha_Y$,
\begin{align*}
K(fg) &= (F(fg),G(fg))=(F(f)F(g),G(f)G(g))=(F(f),G(f))(F(g),G(g)) =K(f)K(g)
\end{align*}
and
\begin{align*}
K(1_X) = (F(1_X),G(1_X))=(1_{F(X)},1_{G(X)})=1_{K(X)}
\end{align*}
for $\xymatrix@1{X\, \ar[r]^f &\, Y \, \ar[r]^g & \,Z}$ in $\sA$.

Suppose that $\sA$ and $\sB$ are additive.\\
If $F$ and $G$ are additive, then we have
\begin{align*}
K(f+h)&=(F(f+h),G(f+h))=(F(f)+F(h),G(f)+G(h)) \\&= (F(f),G(f))+(F(h),G(h))= K(f)+K(h)
\end{align*}
for $\xymatrix@1{X\, \ar@<0.4ex>[r]^{f}\ar@<-0.4ex>[r]_{h}&\,Y}$ in $\sA$, so $K$ is additive.\\
If $K$ is additive, then we have
\begin{align*}
(F(f+h),G(f+h)) &= K(f+h)=K(f)+K(h) = (F(f),G(f))+(F(h),G(h))\\
&=(F(f)+F(h),G(f)+G(h))
\end{align*}
for $\xymatrix@1{X\, \ar@<0.4ex>[r]^{f}\ar@<-0.4ex>[r]_{h}&\,Y}$ in $\sA$. Therefore $F(f+h)=F(f)+F(h)$ and $G(f+h)=G(f)+G(h)$ hold, so $F$ and $G$ are additive.

The statement about the exactness of the functors follows from the previous lemma \ref{lem:delta1exact}:\\
Suppose that $\sA$ and $\sB$ are abelian. For a left-(right-)exact sequence $\xymatrix@1{X\, \ar[r]^f & \,Y\, \ar[r]^g & \,Z}$ in $\sA$, the sequence $\xymatrix@1{K(X)\, \ar[r]^{K(f)}&\,K(Y)\, \ar[r]^{K(g)}&\,K(Z)}$ is left-(right-)exact in $\sB^{\Delta_1}$ if and only if the sequences $\xymatrix@1{F(X) \,\ar[r]^{F(f)}&\,F(Y)\, \ar[r]^{F(g)}&\,F(Z)}$ and $\xymatrix@1{G(X)\, \ar[r]^{G(f)}&\,G(Y)\, \ar[r]^{G(g)}&\,G(Z)}$ are left-(right-)exact in $\sB$.

Ad (b). We have 
\begin{align*}
(F(1_X),G(1_X))=K(1_X) = 1_{K(X)} = (1_{F(X)}, 1_{G(X)}),
\end{align*}
\begin{align*}
(F(fg),G(fg)) = K(fg) = K(f)K(g)=(F(f),G(f))(F(g),G(g))=(F(f)F(g),G(f)G(g))
\end{align*}
and therefore $F(1_X)=1_{F(X)}$, $G(1_X)=1_{G(X)}$, $F(fg)=F(f)F(g)$ and $G(fg)=G(f)G(g)$
for $\xymatrix@1{X \,\ar[r]^f &\,Y \,\ar[r]^g&\,Z}$ in $\sA$. 

We have $F(f) \alpha_Y = K(f)_0  (K(Y))(\xymatrix@1@C=4mm{0\, \ar[r]& \,1}) = (K(X))(\xymatrix@1@C=4mm{0\, \ar[r]& \,1}) K(f)_1 = \alpha_X G(f)$ for $\xymatrix@1{X\, \ar[r]^f &\,Y}$ in $\sA$.
\end{proof}

\begin{rem}
Without proof, 	the constructions in the previous proposition \ref{prop:delta1} yield an isomorphism of categories $(\sB^{\sA})^{\Delta_1} \stackrel{\sim}\longrightarrow (\sB^{\Delta_1})^{\sA}$.
\end{rem}

\begin{thrm}[Universal property of the Adelman category] \label{thrm:universalproperty}
Suppose given an additive category $\sA$ and an abelian category $\sB$.
\begin{itemize}
	\item[(a)] Suppose given an additive functor $F \colon \sA \to \sB$.

We set $\hat F:= \H_{\sB} \circ \Adel(F) \colon \Adel(\sA) \to \sB$, cf. definitions \ref{defn:adelfunctortransformation} and \ref{defn:H}.

The functor $\hat F$ is exact with $\hat F \circ \I_\sA = F$.

If $G,\tilde{G} \colon \Adel(\sA) \to \sB$ are exact functors and $\sigma \colon G\circ \I_{\sA} \Rightarrow \tilde G \circ \I_{\sA}$ is an isotransformation, then there exists an isotransformation $\tau \colon G \Rightarrow \tilde G$ with $\tau \star \I_{\sA} = \sigma$. In particular, this holds for $G \circ \I_{\sA} = \tilde G \circ \I_{\sA}$ and $\sigma = 1_{G \circ \I_{\sA}}$.
\begin{align*}
\xymatrix{
\sA \ar[r]^{F} \ar[d]_{\I_{\sA}}& \sB\\
\Adel(\sA) \ar[ur]_{\hat F}
}
\end{align*}
\item[(b)] Suppose given additive functors $F,G \colon \sA \to \sB$ and a transformation $\alpha \colon F \Rightarrow G$.

There exists a unique transformation $\hat \alpha \colon \hat F \Rightarrow \hat G$ satisfying $\hat \alpha \star \I_{\sA} = \alpha$.
\begin{align*}
\xymatrix{
\sA \ar[dd]_{\I_{\sA}} \ar@/^2mm/[rr]^F="F" \ar@/_2mm/[rr]_G="G"  \ar@2"F"+<0mm,-2.6mm>;"G"+<0mm,2.4mm>_\alpha && \sB
\\ \\
\Adel(\sA) \ar@/^2mm/[uurr]^*+<-2mm,-2mm>{\scriptstyle\hat F}="hF" \ar@/_2mm/[uurr]_*+<-2.3mm,-2.3mm>{\scriptstyle\hat G}="hG"  \ar@2"hF"+<1.5mm,-1.8mm>;"hG"+<-1.1mm,1.6mm>_{\hat \alpha}
}
\end{align*}
\end{itemize}
\end{thrm}

\begin{proof}
Ad (a). The functors $\Adel(F) \colon \Adel(\sA) \to \Adel(\sB)$ and $\H_{\sB} \colon \Adel(\sB) \to \sB$ are exact, cf. theorems \ref{defn:adelfunctortransformation} and \ref{thrm:Hexact}. Therefore $\hat F$ is exact too. 

Suppose given $\xymatrix@1{X \,\ar[r]^f&\,Y}$ in $\sA$. Note that $F(0_{\sA}) \in \nZ_{\sB}$, cf. definition \ref{defn:H}.
We have 
\begin{align*}
(\hat F \circ \I_{\sA})(X) = ( \H_{\sB} \circ \Adel(F))(\I_{\sA}(X)) =\H_{\sB}\big(\big( \xymatrix@1{F(0)\ar[r]^-0&F(X)\ar[r]^-0&F(0)}\big)\big) = F(X)
\end{align*}
and
\begin{align*}
(\hat F \circ \I_{\sA})(f) = ( \H_{\sB} \circ \Adel(F))(\I_{\sA}(f)) =\H_{\sB}([F(0),F(f),F(0)]) = F(f).
\end{align*}
We conclude that $\hat F \circ \I_{\sA}=F$ holds.

We abbreviate $\sA' :=\I_{\sA}(\sA)$, cf. definition~\ref{defn:inclusionfunctor}. 

Let $\sigma'_{\I_{\sA}(X)} := \sigma_X$ for $X \in \Ob \sA$. Then $\sigma' \in \Hom^{\iso}_{\sB^{\sA'}}(G|_{\sA'}, \tilde G|_{\sA'})$ since we have
\begin{align*}
\sigma'_{\I_{\sA}(X)} \tilde G(\I_{\sA}(f)) =\sigma_X (\tilde G \circ \I_{\sA})(f) = ( G \circ \I_{\sA})(f)\sigma_Y = G(\I_{\sA}(f)) \sigma'_{\I_{\sA}(Y)}
\end{align*}
for $\xymatrix@1{X\, \ar[r]^f & \,Y}$ in $\sA$ and since $\sigma_X$ is an isomorphism for $X \in \Ob \sA$. 

Let $\I_{\sA}':=\I_{\sA}|^{\sA'} \colon \sA \to \sA'$.

We have $\sigma' \star \I_{\sA}' =\sigma$ since $(\sigma' \star \I_{\sA}')_X = \sigma'_{\I_{\sA}(X)} = \sigma_X$ holds for $X \in \Ob \sA$.

Consider the subcategories $\sA' \subseteq \sR(\sA) \subseteq \Adel(\sA)$. The embedding functor from $\sA'$ to $\sR(\sA)$ shall be denoted by $E \colon \sA' \to  \sR(\sA)$ and the embedding functor from $\sR(\sA)$ to $\Adel(\sA)$ shall be denoted by $E' \colon \sR(\sA) \to \Adel(\sA)$. 
\begin{align*}
\xymatrix{
\sA \ar[r]^-{\I_{\sA}'} & \sA' \ar[r]^-E & \sR(\sA) \ar[r]^-{E'} & \Adel(\sA)
}
\end{align*}
Note that $E' \circ E \circ \I_{\sA}' = \I_{\sA}$ holds.

For $P \in \sR(\sA)$, there exists a left-exact sequence $\xymatrix@1{P\,\armfl{0.37}[r]^-k &\,\I_\sA(X)\,\ar[r]^-{\I_\sA(f)}&\,\I_\sA(Y)}$ in $\Adel(\sA)$ with $\xymatrix@1{X \,\ar[r]^f&\,Y}$ in $\sA$, cf. lemma \ref{lem:Rkernel}.
The objects in $\sA'$ are injective, cf. remark \ref{exa:inclusion}.

Proposition \ref{prop:transformationsubcategories}~(a) now gives $\rho \in \Hom^{\iso}_{\sB^{\sR(\sA)}}(G|_{\sR(\sA)}, \tilde G|_{\sR(\sA)})$ with $\rho \star E= \sigma'$.

Consider the subcategories $ \sR(\sA) \subseteq\Adel(\sA) \subseteq \Adel(\sA)$.
For $A \in \Ob(\Adel(\sA))$, there exists a right-exact sequence $\xymatrix@1{P \,\ar[r]^f &\, Q\, \are[r]^c &\,A}$ with  $P,Q \in \Ob(\sR(\sA))$, cf. remark \ref{lem:Adelcokernel}. The objects in $\sR(\sA)$ are projective, cf. proposition \ref{prop:projectivesinjectives}.

Proposition \ref{prop:transformationsubcategories}~(b) now gives $\tau \in \Hom^{\iso}_{\sB^{\Adel(\sA)}}(G, \tilde G)$ with $\tau \star E'= \rho$.

We have
\begin{align*}
 \tau \star \I_{\sA} = \tau \star (E' \circ E \circ \I_{\sA}') = (\tau \star E') \star (E \circ \I_{\sA}') = \rho \star(E \circ \I_{\sA}')= (\rho \star E) \star \I_{\sA}' = \sigma' \star \I_{\sA}' = \sigma.
\end{align*}
Ad (b). We get an additive functor $K \colon \sA \to \sB^{\Delta_1}$ by setting $K(A) := \big( \xymatrix@1{F(A)\, \ar[r]^{\alpha_A} & \,G(A)} \big)$ for $A \in \Ob \sA$ and $K(f) := (F(f),G(f))$ for $f \in \Mor \sA$, cf. proposition \ref{prop:delta1}.

Part (a) gives an exact functor $\hat K \colon \Adel(\sA) \to \sB^{\Delta_1}$ with $\hat K \circ \I_{\sA} = K$. We obtain exact functors $\tilde F, \tilde G \colon \Adel(\sA) \to \sB$ and a transformation $\tilde \alpha \colon \tilde F \Rightarrow \tilde G$ by setting $\tilde F(A) := \hat K(A)_0$, $\tilde G(A) :=\hat K(A)_1$, $\tilde \alpha_A:=(\hat K(A))(\xymatrix@1@C=4mm{0\,\ar[r]&\, 1})$ for $A \in \Ob(\Adel(\sA))$ and $\tilde F(q):= \hat K(q)_0$, $\tilde G(q):=\hat K(q)_1$ for $q \in \Mor(\Adel(\sA))$, cf. proposition \ref{prop:delta1}.

Next, we show that the equations $\tilde F \circ \I_{\sA} =F$, $\tilde G \circ \I_{\sA} = G$ and $\tilde \alpha \star \I_{\sA} = \alpha$ hold.

For $X \in \Ob \sA$, we have $\hat K(\I_{\sA}(X)) = K(X) = \big( \xymatrix@1{F(X) \,\ar[r]^{\alpha_X} &\, G(X)} \big)$. This implies 
$\tilde F(\I_{\sA}(X)) = \hat K(\I_{\sA}(X))_0 = F(X)$, $\tilde G(\I_{\sA}(X)) = \hat K(\I_{\sA}(X))_1 = G(X)$ and $\tilde \alpha_{\I_{\sA}(X)} = \alpha_X$.

For $f \in \Mor \sA$, we have $\hat K(\I_{\sA}(f)) = K(f) = (F(f),G(f))$. This implies 
$\tilde F(\I_{\sA}(f)) = \hat K(\I_{\sA}(f))_0 = F(f)$ and $\tilde G(\I_{\sA}(f)) = \hat K(\I_{\sA}(f))_1 = G(f)$.

Part (a) gives isotransformations $\tau_F \colon \hat F \Rightarrow \tilde F$ and $\tau_G \colon  \tilde G \Rightarrow \hat G $ satisfying $\tau_F \star \I_{\sA}= 1_F$ and $\tau_G \star \I_{\sA} = 1_G$.

We set $\hat \alpha := \tau_F \tilde \alpha \tau_G \colon \hat F \Rightarrow \hat G$ and obtain $\hat \alpha \star \I_{\sA} = (\tau_F \tilde \alpha \tau_G) \star \I_{\sA} = (\tau_F \star \I_{\sA} )(\tilde \alpha \star \I_{\sA} )(\tau_G \star \I_{\sA} ) = 1_F \alpha 1_G = \alpha$.

We use the functors $\I_{\sA}'$, $E$ and $E'$ defined in (a). 

Suppose given $\beta,\gamma \colon \hat F \Rightarrow \hat G$ with $\beta \star \I_{\sA} = \gamma \star \I_{\sA}$.

We have $(\beta \star E') \star E = \beta \star (E' \circ E) = \gamma \star (E' \circ E) = (\gamma \star E') \star E$ since $(\beta \star (E' \circ E))_{\I'_{\sA}(X)} = \beta_{\I_{\sA}(X)} = (\beta \star \I_{\sA} )_X = (\gamma \star \I_{\sA})_X = \gamma_{\I_{\sA}(X)} = (\gamma \star (E' \circ E))_{\I'_{\sA}(X)}$ holds for $X \in \Ob \sA$.

Consider the subcategories $\sA' \subseteq \sR(\sA) \subseteq \Adel(\sA)$. As seen above, we may apply proposition~\ref{prop:transformationsubcategories}~(a). We conclude that $\beta \star E' = \gamma \star E'$ holds because of the injectivity.

Consider the subcategories $\sR(\sA) \subseteq \Adel(\sA) \subseteq \Adel(\sA)$. Again, we may apply proposition~\ref{prop:transformationsubcategories}~(b). We conclude that $\beta = \gamma$ holds because of the injectivity.
\end{proof}

\begin{rem}
We give an overview of the construction of $\hat F$ in the previous theorem \ref{thrm:universalproperty}.
\begin{align*}
\xymatrix{
\sA \ar[rr]^{F} \ar[dd]^{\I_{\sA}} && \sB \ar@/^12mm/[dddd]^{1_{\sB}} \ar[dd]^{\I_{\sB}}\\
\\
\Adel(\sA) \ar[ddrr]_{\hat F} \ar[rr]^{\Adel(F)} && \Adel(\sB) \ar[dd]^{\H_{\sB}}\\
\\
&&\sB
}
\end{align*}
Note that $\I_{\sB} \circ F = \Adel(F) \circ \I_{\sA}$ is not true in general, cf. theorem \ref{defn:adelfunctortransformation}. However, $\H_{\sB} \circ \Adel(F) = \hat F$, $\H_{\sB} \circ \I_{\sB} = 1_{\sB}$ and $F = 1_{\sB} \circ F = \hat F \circ \I_{\sA}$ hold.
\end{rem}

\begin{rem}\label{rem:last}
Suppose given an additive category $\sA$, an abelian category $\sB$ and an additive functor $F \colon \sA \to \sB$. We saw in theorem \ref{defn:adelfunctortransformation} that $\Adel(F) \colon \Adel(\sA) \to \Adel(\sB)$ is an exact functor and that there exists an isotransformation $\varepsilon^F \colon   \I_{\sB} \circ F  \to  \Adel(F) \circ \I_{\sA}$. By applying theorem \ref{thrm:universalproperty}~(a) to the additive functor $\I_{\sB} \circ F \colon \sA \to \Adel(\sB)$, we get an exact functor $\widehat{\I_{\sB} \circ F} \colon \Adel(\sA) \to \Adel(\sB)$ with $\widehat{\I_{\sB} \circ F} \circ \I_{\sA} = \I_{\sB} \circ F$.\\
We conclude that, by theorem \ref{thrm:universalproperty}~(a), $\Adel(F)$ and $\widehat{\I_{\sB} \circ F}$ are isomorphic in $\Adel(\sB)^{\Adel(\sA)}$.\\
Note that $\widehat{\I_{\sB} \circ F}$ depends on the choices we make in definition \ref{defn:H}, applied to the abelian category $\Adel(\sB)$.
\end{rem}

\begin{exa} \label{exa:Z}
Consider the (additive) inclusion functor $E$ from $\Zf$ to $\Zm$. Theorem~\ref{thrm:universalproperty}~(a) gives an exact functor $\hat E \colon \Adel(\Zf) \to \Zm$ with $\hat E \circ \I_{\Zf} = E$.
\begin{itemize}
	\item[(a)] The functor $\hat E$ is dense:
	
	Suppose given $X \in \Ob(\Zm)$. We choose a free resolution $\xymatrix@1{F_1\, \ar[r]^d &\,F_0\, \are[r]^*+<0.8mm,0.8mm>{\scriptstyle e} &\, X}$ of $X$.\\
	Consider $\big(\xymatrix@1{F_1\, \ar[r]^d &\,F_0\, \ar[r]^0 &\,0}\big)\in \Ob(\Adel(\Zf))$. A kernel of $0 \colon F_1 \to 0$ is given by $1_{F_1}$, a cokernel of $d$ is given by $e$ and an image of $1_{F_1} e$ is given by $\xymatrix@1{F_1\, \ar[r]^{e}& \,X\, 	\ar[r]^{1} &\, X}$.
	\begin{align*}
	\xymatrix{
	F_0 \ar[rr]^d &&F_1 \ar[rr] \are[dr]_e& & 0 \\
	&F_1 \arm[ur]_1 \are[dr]_e &&X \\&&X \arm[ur]_1
	}
	\end{align*}
	Therefore we have $\hat E\big( \big(\xymatrix@1{F_1\, \ar[r]^d &\,F_0\, \ar[r]^0 &\,0}\big) \big) \cong X$ in $\Zm$.

	\item[(b)]The functor $\hat E$ is not full:
	
	Let $A:=\Big( \xymatrix@1@C=12mm{\Z \oplus \Z\, \ar[r]^-{\sm{4&0\\0&1}} &\, \Z \oplus \Z\, \ar[r]^-{\sm{1\\2}}&\,\Z}\Big)$ and $B:=\big(\xymatrix@1{\Z\, \ar[r]^2 &\, \Z\, \ar[r]^0 &\,0}\big)$ in $\Adel(\Zf)$.
	A kernel of $\sm{1\\2} \colon \Z \oplus \Z \to \Z$ is given by $\sm{2&-1} \colon \Z \to \Z \oplus \Z$, a cokernel of\\ $\sm{4&0\\0&1} \colon \Z \oplus \Z \to \Z\oplus \Z$ is given by $\sm{1\\0} \colon \Z \oplus \Z \to \Z/4$ and an image of $\sm{2&-1} \sm{1\\0}$ is given by $\xymatrix@1{\Z \,\ar[r]^-1 & \,\Z/2\, \ar[r]^-2& \,\Z/4}$. So $\hat E(A) \cong \Z/2$.
	\begin{align*}
	\xymatrix{
	\Z \oplus \Z \ar[rr]^-{\sm{4&0\\0&1}} && \Z \oplus \Z \are[dr]_{\sm{1\\0}} \ar[rr]^-{\sm{1\\2}}&&\Z \\
	&\Z \arm[ur]_{\sm{2&-1}} \are[dr]_1 &&\Z/4 \\&&\Z/2 \arm[ur]_2
	}
	\end{align*}
	A kernel of $0 \colon \Z \to 0$ is given by $1_{\Z}$, a cokernel of $2 \colon \Z \to \Z$ is given by $1 \colon \Z \to \Z/2$ and an image of $1_{\Z}\cdot 1$ is given by $\xymatrix@1{\Z \,\ar[r]^-1 & \,\Z/2\, \ar[r]^-1 &\, \Z/2}$.  So $\hat E(B) \cong \Z/2$.
	
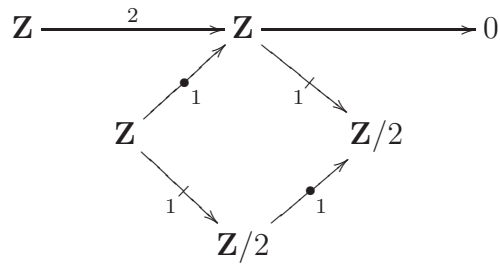
\begin{align*}
	\xymatrix{
	\Z\ar[rr]^2 && \Z \ar[rr] \are[dr]_1& & 0 \\
	&\Z \arm[ur]_1 \are[dr]_1 && \Z/2\\&& \Z/2 \arm[ur]_1
	}
	\end{align*}
	Consider $1 \colon \Z/2 \to \Z/2$. If $\hat E$ was full, there would exist $g \colon \Z/4 \to \Z/2$ such that $2 \cdot g = 1$, which is impossible.
	
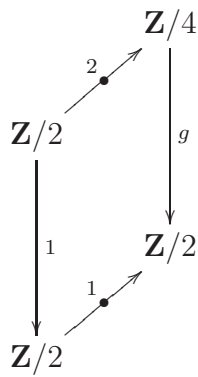
\begin{align*}
	\xymatrix{&\Z/4 \ar[dd]^g \\ \Z/2 \ar[dd]^1\arm[ur]^2 \\&\Z/2\\\Z/2\arm[ur]^1}
	\end{align*}
	
	\item[(c)]The functor $\hat E$ is not faithful:
	
	Let $A:= \big(\xymatrix@1{0 \,\ar[r]^0 &\,\Z\, \ar[r]^2 &\, \Z}\big)$, $B:=\big( \xymatrix@1{0 \,\ar[r]^0 &\,\Z\, \ar[r]^0 &\,0}\big)$ and $[0,1,0] \colon A \to B$ in $\Adel(\Zf)$. We have $0 \neq [0,1,0]$ since there does not exist $t \colon \Z \to \Z$ such that $2 \cdot t = 1_{\Z}$.
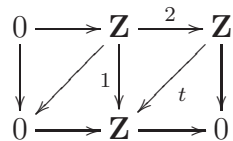
	\begin{align*}
	\xymatrix{
	0\ar[d] \ar[r] & \Z\ar[d]_1 \ar[dl] \ar[r]^2 & \Z\ar[dl]^t\ar[d]\\
	0 \ar[r] &\Z \ar[r]&0
	}
	\end{align*}
	A kernel of $2 \colon \Z \to \Z$ is given by $0 \colon 0 \to \Z$, a cokernel of $0 \colon 0 \to \Z$ is given by $1_{\Z}$ and an image of $0 \cdot 1_{\Z}$ is given by $\xymatrix{0 \ar[r]^-0& 0 \ar[r]^-0 &\Z}$.
	\begin{align*}
	\xymatrix{
	0  \ar[rr] && \Z  \are[dr]_1 \ar[rr]^2 && \Z\\
	&0 \arm[ur]_0 \are[dr]_0 &&\Z \\
	&&0\arm[ur]_0
	}
	\end{align*}
	So $\hat E(A) \cong 0 = \hat E(B)$. We conclude that $\hat E([0,1,0]) = 0 = \hat E(0)$.
	
	\item[(d)] The functor $\hat E$ is not an equivalence due to (b) or (c). This also follows from $\Zm$ having not enough injectives or, alternatively, from the fact that $(\Z/2 \underset{\Z}{\otimes}-)\colon \Zm \to \Zzm$ is additive but not kernel-preserving:
	
	A kernel of $2 \colon \Z \to \Z$ is given by $0 \colon 0 \to \Z$ in $\Zf$, but $0 \colon 0 \to \Z/2$ is not a kernel of $0 \colon \Z/2 \to \Z/2$ in $\Zzm$.\\
	Now if $\hat E$ was an equivalence, the category $\Zm$ would also satisfy the universal property of the Adelman category and there would exist an exact functor\\ $(\widehat{\Z/2 \underset{\Z}{\otimes}-}) \colon \Zm \to \Zzm$ with $(\widehat{\Z/2 \underset{\Z}{\otimes}-}) \circ E = (\Z/2 \underset{\Z}{\otimes}-)$. But $(\widehat{\Z/2 \underset{\Z}{\otimes}-}) \circ E$ is kernel-preserving, whereas $(\Z/2 \underset{\Z}{\otimes}-)$ is not.
\end{itemize}
\end{exa}




\bibliography{ref}
\bibliographystyle{acm}
\thispagestyle{empty}




\end{document}